\documentclass[11pt,a4paper]{article}
\pdfoutput=1

\usepackage{jheppub-nosort,amsthm}

\usepackage{amsmath,amssymb,amsthm,amscd,mathtools}
\usepackage{tikz}
\tikzset{node distance=2cm, auto}
\usepackage{epsfig}
\usepackage{psfrag,comment}
\usepackage{graphicx}
\usepackage{caption}
\usepackage{subcaption}
\usepackage{mathrsfs}
\usepackage{relsize}
\usepackage{rotating}
\usepackage{pdflscape}
\usepackage{changepage}

\usepackage{xcolor}
\usepackage[normalem]{ulem}

\usepackage{tikz}
\usetikzlibrary{calc}
\usetikzlibrary{arrows,shapes}
\usetikzlibrary{matrix}
\usetikzlibrary{positioning}

\usepackage{pdflscape}

\newtheorem{proposition}{\sf PROPOSITION}

\newcommand{\CF}{{\cal F}}

\newcommand{\CN}{{\cal N}}
\newcommand{\CO}{{\cal O}}
\newcommand{\CP}{{\cal P}}

\newcommand{\CX}{{\cal X}}

\def\BN{{\mathbb N}}
\def\BZ{{\mathbb Z}}

\def\BC{{\mathbb C}}
\def\BP{{\mathbb P}}

\def\BS{{\mathbb S}}

\def\BQ{{\mathbb Q}}

\newcommand{\be}{\begin{equation}}
\newcommand{\ee}{\end{equation}}
\newcommand{\ba}{\begin{aligned}}
\newcommand{\ea}{\end{aligned}}
\newcommand{\bea}{\begin{eqnarray}}
\newcommand{\eea}{\end{eqnarray}}
\newcommand{\bean}{\begin{eqnarray*}}
\newcommand{\eean}{\end{eqnarray*}}

\newcommand{\p}{\partial}

\def\r{\right\rangle}

\def\1{\mathbf{1}}
\def\0{|\1\r}

\newcommand{\rme}{{\rm e}}
\newcommand{\rmi}{{\rm i}}
\newcommand{\rmd}{{\rm d}}

\def\XXint#1#2#3{{\setbox0=\hbox{$#1{#2#3}{\int}$}
     \vcenter{\hbox{$#2#3$}}\kern-.5\wd0}}

\DeclarePairedDelimiter\floor{\lfloor}{\rfloor}

\newcommand{\tPTwo}{\text{\tiny $\BP^2$}}
\newcommand{\tPOne}{\text{\tiny $\BP^1$}}

\newcommand{\tconi}{\text{\tiny coni}}
\newcommand{\tABJM}{\text{\tiny ABJM}}
\newcommand{\tQuintic}{\text{\tiny quint}}
\newcommand{\tLC}[1]{\text{\tiny $X_{#1}$}}
\newcommand{\tH}{\text{\tiny H}}

\newcommand{\sh}{q} %Shift on Kahler leading degree
\newcommand{\polyname}{\CP} %Letter for polynomials in leading large-order
\DeclareMathOperator{\GV}{\mathcal{GV}}
\DeclareMathOperator{\GW}{\mathcal{GW}}
\newcommand{\TGW}{\widetilde{\GW}}

\DeclareMathOperator{\MB}{\mathcal{B}}

\newcommand{\HB}{\widehat{\MB}}

\DeclareMathOperator{\mb}{b}
\DeclareMathOperator{\MA}{\mathcal{A}}

\definecolor{pert}{cmyk}{0.5, 0.39, 0., 0.3}
\definecolor{1inst}{cmyk}{0.5, 0., 1., 0.4}
\definecolor{2inst}{cmyk}{0., 0.255, 0.75, 0.1}
\definecolor{3inst}{cmyk}{0., 0.6643, 0.73, 0.27}
\makeatletter
\newsavebox\myboxA
\newsavebox\myboxB
\newlength\mylenA

\newcommand*\widebar[2][0.75]{%
    \sbox{\myboxA}{$\m@th#2$}%
    \setbox\myboxB\null% Phantom box
    \ht\myboxB=\ht\myboxA%
    \dp\myboxB=\dp\myboxA%
    \wd\myboxB=#1\wd\myboxA% Scale phantom
    \sbox\myboxB{$\m@th\overline{\copy\myboxB}$}%  Overlined phantom
    \setlength\mylenA{\the\wd\myboxA}%   calc width diff
    \addtolength\mylenA{-\the\wd\myboxB}%
    \ifdim\wd\myboxB<\wd\myboxA%
       \rlap{\hskip 0.8\mylenA\usebox\myboxB}{\usebox\myboxA}%
    \else
        \hskip -0.5\mylenA\rlap{\usebox\myboxA}{\hskip 0.5\mylenA\usebox\myboxB}%
    \fi}
\makeatother

\newdimen\tableauside\tableauside=1.0ex
\newdimen\tableaurule\tableaurule=0.4pt
\newdimen\tableaustep
\def\phantomhrule#1{\hbox{\vbox to0pt{\hrule height\tableaurule width#1\vss}}}
\def\phantomvrule#1{\vbox{\hbox to0pt{\vrule width\tableaurule height#1\hss}}}
\def\sqr{\vbox{%
  \phantomhrule\tableaustep
  \hbox{\phantomvrule\tableaustep\kern\tableaustep\phantomvrule\tableaustep}%
  \hbox{\vbox{\phantomhrule\tableauside}\kern-\tableaurule}}}
\def\squares#1{\hbox{\count0=#1\noindent\loop\sqr
  \advance\count0 by-1 \ifnum\count0>0\repeat}}
\def\tableau#1{\vcenter{\offinterlineskip
  \tableaustep=\tableauside\advance\tableaustep by-\tableaurule
  \kern\normallineskip\hbox
    {\kern\normallineskip\vbox
      {\gettableau#1 0 }%
     \kern\normallineskip\kern\tableaurule}%
  \kern\normallineskip\kern\tableaurule}}
\def\gettableau#1{\ifnum#1=0\let\next=\null\else
\squares{#1}\let\next=\gettableau\fi\next}
\preprint{
{\small{\textsf{DESY 16-089}}}
}

\title{On Asymptotics and Resurgent Structures of Enumerative Gromov--Witten Invariants}

\author[a]{Ricardo~Couso-Santamar\'\i a,}
\affiliation[a]{CAMGSD, Departamento de Matem\'atica, Instituto Superior T\'ecnico, Universidade de Lisboa,\\ Av. Rovisco Pais 1, 1049-001 Lisboa, Portugal\\}
\emailAdd{santamaria@math.tecnico.ulisboa.pt}

\author[a,b]{Ricardo~Schiappa,}
\affiliation[b]{D\'epartement de Physique Th\'eorique \& Section de Math\'ematiques,\\ Universit\'e de Gen\`eve, Gen\`eve, CH-1211 Switzerland\\}
\emailAdd{schiappa@math.tecnico.ulisboa.pt}

\author[a,c]{Ricardo~Vaz\,}
\affiliation[c]{DESY Theory Group, DESY Hamburg,\\
Notkestrasse 85, D-22603 Hamburg, Germany\\}
\emailAdd{ricardo.carmo.vaz@tecnico.ulisboa.pt}

\abstract{Making use of large-order techniques in asymptotics and resurgent analysis, this work addresses the growth of enumerative Gromov--Witten invariants---in their dependence upon genus and degree of the embedded curve---for several different threefold Calabi--Yau varieties. In particular, while the leading asymptotics of these invariants at large genus or at large degree is exponential, at combined large genus \textit{and} degree it turns out to be factorial. This factorial growth has a resurgent nature, originating via mirror symmetry from the resurgent-transseries description of the B-model free energy. This implies the existence of nonperturbative sectors controlling the asymptotics of the Gromov--Witten invariants, which could themselves have an enumerative-geometry interpretation. The examples addressed include: the resolved conifold; the local surfaces local $\BP^2$ and local $\BP^1 \times \BP^1$; the local curves and Hurwitz theory; and the compact quintic. All examples suggest very rich interplays between resurgent asymptotics and enumerative problems in algebraic geometry.
}


\arxivnumber{1605.07473}

\begin{document}

%%%%%%%%%%%%%%%%%%%%%%%%%%%%%%%%%%%%%%%%%%%%%%%%%%%%%%%%%%%%%%%%%
%%%%%%%%%%%%%%%%%%%%%%%%%%%%%%%%%%%%%%%%%%%%%%%%%%%%%%%%%%%%%%%%%
\maketitle
%%%%%%%%%%%%%%%%%%%%%%%%%%%%%%%%%%%%%%%%%%%%%%%%%%%%%%%%%%%%%%%%%
%%%%%%%%%%%%%%%%%%%%%%%%%%%%%%%%%%%%%%%%%%%%%%%%%%%%%%%%%%%%%%%%%

\vfill

\eject

\allowdisplaybreaks

%%%%%%%%%%%%%%%%%%%%%%%%%%%%%%%%%%%%%%%%%%%%%%%%%%%%%%%%%%%%%%%%%
%%%%%%%%%%%%%%%%%%%%%%%%%%%%%%%%%%%%%%%%%%%%%%%%%%%%%%%%%%%%%%%%%
\section{Introduction}\label{sec:intro}
%%%%%%%%%%%%%%%%%%%%%%%%%%%%%%%%%%%%%%%%%%%%%%%%%%%%%%%%%%%%%%%%%
%%%%%%%%%%%%%%%%%%%%%%%%%%%%%%%%%%%%%%%%%%%%%%%%%%%%%%%%%%%%%%%%%

Geometrical-counting problems, albeit many times rather natural and simple to formulate, may lead to remarkably rich and interesting structures. Among these, enumerative invariants play an important classification role within algebraic geometry. For example, counting pseudo-holomorphic curves inside symplectic manifolds gives rise to the famous Gromov--Witten (GW) invariants. These are invariants associated to the symplectic manifold $\CX$, which are rational numbers (implying a ``virtual'' counting) depending on both genus, $g$, and degree, $d$, of the embedded curve. We shall denote them by $N_{g,d}$. The computation of GW invariants is generically hard, becoming simpler when the manifold is Calabi--Yau (CY) where they are generated by the A-model topological-string free energy. This is a long story which goes back to the discovery of mirror symmetry; see, \textit{e.g.}, \cite{lvw89, gp90, cls90, cdgp91, w93, bcov93, k94, gv95, syz96, hv00} for early references, and, \textit{e.g}, \cite{w91, hkt94, g97, m05, a12} for reviews.

Consider the A-model on a CY $\CX$, in the large-radius phase (valid when the K\"ahler parameter $t$ is large). The A-model free energy is then given by an asymptotic, genus expansion
\be
\label{gstring-expansion}
F (\CX) \simeq \sum_{g=0}^{+\infty} g_{\text{s}}^{2g-2} F_g (t),
\ee
\noindent
where the genus-$g$ contributions to the free energy may be decomposed as \cite{bcov93}
\be
\label{alphaprime-expansion}
F_g (t) = \sum_{d > 0} N_{g,d}\, Q^d.
\ee
\noindent
The sum over degree $d$ corresponds to a sum over topological sectors as classified by worldsheet instantons (where $Q = \rme^{-t}$ in units where $\alpha'=2\pi$). While this explicitly shows how the topological-string free energy is a generating function for the genus $g$, degree $d$, enumerative GW invariants of $\CX$, $N_{g,d}$, the two expansions above have rather different properties: while the fixed-genus \eqref{alphaprime-expansion} is a \textit{convergent} series\footnote{Convergence was proved at planar level in \cite{knn77, th82}, although there is no general mathematical proof at arbitrary genus. As we shall see later on, our examples in the present paper also strongly support this convergence property.}, with a non-zero radius of convergence, \eqref{gstring-expansion} is instead a \textit{divergent} asymptotic series, with zero radius of convergence; see, \textit{e.g.}, \cite{gp88}. The reason for this is the factorial growth of the genus-$g$ contributions with genus, as $F_g \sim (2g)!$.

From the standpoint of defining the string free energy, the asymptotic nature of the perturbative expansion \eqref{gstring-expansion} implies that $F (\CX)$ cannot be properly defined by perturbation theory alone. One way to move forward is to use the theory of resurgence \cite{e81}. In this context, the perturbative expansion gets enlarged into a transseries, an object which fully captures all information concerning the observable that it represents, including both perturbative/analytic components (in powers of the string coupling $g_{\text{s}}$) and nonperturbative/non-analytic components (in powers of the ``instanton'' factor $\rme^{-1/g_{\text{s}}}$). The asymptotic and resurgent nature of the perturbative sequence implies the existence of these instanton-type terms, of which there can be many distinct types and with different strengths. Remarkably, all these seemingly independent perturbative and nonperturbative sectors in the transseries turn out to be related to each other via a tight web of asymptotic resurgence relations. In particular, the leading factorial growth of perturbation theory is a consequence of these asymptotic relations, as is any other subleading growth correcting that factorial term. As a result, one may in fact extract, or decode, nonperturbative information from perturbation theory alone and vice-versa. Moreover, these interrelations have somewhat universal forms, and should be expected to hold across a wide range of different problems.

In recent years resurgence has been applied within\footnote{For an introduction to the main ideas of resurgent asymptotics, and a very complete list of references concerning many other recent applications of resurgence, we refer the reader to \cite{abs16}.} topological string theory \cite{m06, msw07, m08, ps09, sw09, kmr10, dmp11, asv11, sv13, cesv13, gmz14, cesv14, ars14, v15, c15} and its double-scaled limits at special points in moduli space \cite{msw07, msw08, gm08, ps09, gikm10, asv11, sv13}. In particular, nonperturbative transseries-solutions to the holomorphic-anomaly equations of the B-model were constructed in \cite{cesv13, cesv14, c15}. These references further focused on the example of local $\BP^2$, a non-compact CY threefold, where a very rich nonperturbative structure was uncovered, with diverse instanton actions vying for dominance on the Borel plane as the moduli changed. In our present paper we wish to turn our attention to the A-model instead, and in particular to the enumerative invariants it generates.

From the standpoint of computing enumerative invariants, the convergence properties of their generating functions might not seem terribly important at first sight. It is nonetheless the case that these convergence properties will dictate the asymptotic behavior of these invariants, in genus and in degree, and this is one of the main question we address in the present work. Furthermore, within the A-model the GW invariants are the internal ingredients constructing the string free energies, and it seems reasonable to transfer resurgence questions and properties from the free energies to the invariants themselves. In particular, one natural question is to ask exactly how the GW invariants are responsible for the (known) factorial growth of the free energies they build. For example, the convergence of \eqref{alphaprime-expansion} roughly implies that, at fixed genus, the large-degree asymptotics of the GW invariants\footnote{We shall use the notation where the boldface character specifies which index (if any) remains fixed.} $N_{\boldsymbol{g},d}$ corresponds at most to a leading exponential growth. On the other hand, the asymptotic nature of \eqref{gstring-expansion} might seem to imply that, at fixed degree, the large-genus asymptotics of the GW invariants $N_{g,\boldsymbol{d}}$ corresponds instead to a leading factorial growth, giving rise to the factorial growth inside the free energy. But this will turn out \textit{not} to be the case. The fixed degree, large-genus asymptotics of the GW invariants is \textit{not} factorial, and we shall see how the factorial growth of the free energy is more subtly encoded at the level of GW invariants.

%%%%%%%%%%%%%%%%%%%%%%%%%%%%%%%%%%%%%%%%%%%%%%%%%%%%%%%%%%%%%%%%%
\begin{figure}[t!]
\begin{center}
\raisebox{0.5cm}{\includegraphics[width=0.4\textwidth]{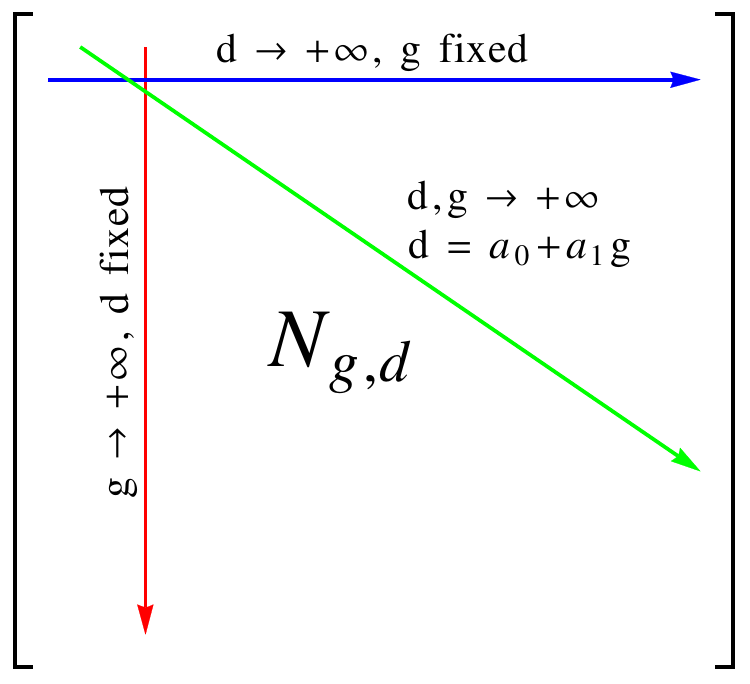}}
\end{center}
\vspace{-1\baselineskip}
\caption{Schematic display of the GW invariants $N_{g,d}$ as a two-dimensional array, with genus and degree representing row and column, respectively. The three arrows are the types of growth that we shall address in this paper: large degree with fixed genus (blue), large genus with fixed degree (red), and combined large degree and genus, with $d=a_0 + a_1 g$ (green).}
\label{fig:GW}
\end{figure}
%%%%%%%%%%%%%%%%%%%%%%%%%%%%%%%%%%%%%%%%%%%%%%%%%%%%%%%%%%%%%%%%%

Note that there are some important differences between addressing resurgent transseries for the B-model free energy, and investigating resurgent asymptotics of A-model enumerative invariants. In the former case, one deals with an asymptotic \textit{series}, which subsequently gets completed into a transseries by the addition of new, nonperturbative sectors. In the latter case, one deals instead with a two-dimensional array of (rational) \textit{numbers}, labeled by both genus and degree (which is represented schematically in figure~\ref{fig:GW}). The GW invariants in this array are not directly the coefficients of any series, so the concept of their transseries extension is not well-defined. However, any asymptotic resurgence relations explaining the different growths of the $N_{g,d}$, in particular along directions with factorial growth, should themselves be dictated by nonperturbative content in the free-energy transseries---possibly also with an enumerative-geometry interpretation. This opens the door to the existence of nonperturbative analogues of the GW invariants. With this idea in mind, we wish to make precise the asymptotic growth of GW invariants along particular directions on this array, as depicted in figure~\ref{fig:GW}:
\begin{itemize}
\item[\color{blue}{$\rightarrow$}] \textbf{Fixed-genus}, large-degree. 
Possibly the most ``classical'' direction previously addressed in the literature, giving rise to leading exponential growth.
\item[\color{red}{$\downarrow$}] Large-genus, \textbf{fixed-degree}. 
Less studied, also giving rise to leading exponential growth. 
\item[\color{green}{$\searrow$}] Large-genus, large-degree. 
Not previously addressed in the literature, finally giving rise to the factorial growth characteristic of the free energy. 
\end{itemize}

Asymptotics of GW invariants\footnote{Asymptotics of related enumerative invariants, such as Donaldson--Thomas or Gopakumar--Vafa invariants, and their relevance towards the computation of M-theoretic black hole entropies, have been addressed in \cite{kkv99, hkmt07}.}, with focus on the fixed-genus and large-degree regime, have been previously addressed in \cite{bcov93, kz99, k05, cgmps06}, where leading exponential growth was found. A fixed-degree, large-genus analysis was done in \cite{z08, mz11}, albeit in a  different set-up\footnote{References \cite{mz99, z08, mz11} address the asymptotics of Weil--Petersson volumes of moduli spaces of algebraic curves, with genus $g$ and $n$ marked punctures (which in some sense corresponds to addressing enumerative invariants of a point). Note that they find some (extra) factorial growth $\sim n!$, but which is associated to the (extra) number of punctures, $n$. In our context this number is $n=0$, as GW invariants arise from the free energy.}, also finding leading exponential behavior. To the best of our knowledge, the ``enumerative source'' of the free-energy factorial growth has never been addressed previously in the literature, and we start filling such gap with our present work. We shall investigate these different asymptotics in several examples, including both compact and non-compact CY threefolds. In particular, our analysis of the exponential growth along horizontal and vertical directions both recovers and generalizes some of the aforementioned previously-known results. The factorial growth is new, and relates to the B-model transseries with its plethora of nonperturbative sectors. Along certain diagonal directions we uncover an universal behavior which is common to geometries in different topological-string universality classes, and which is controlled by the large-radius instanton action. Asymptotic resurgence-like formulae may be written for the ``diagonal'' growth of GW invariants, with their growth dictated by nonperturbative information encoded in the free-energy transseries. In this sense, one should not wonder about transseries completions of GW invariants, but rather about decoding possibly new ``nonperturbative'' enumerative invariants, hidden inside the nonperturbative completions to the B-model topological-string transseries \cite{cesv13, cesv14, c15}.

%%%%%%%%%%%%%%%%%%%%%%%%%%%%%%%%%%%%%%%%%%%%%%%%%%%%%%%%%%%%%%%%%
%%%%%%%%%%%%%%%%%%%%%%%%%%%%%%%%%%%%%%%%%%%%%%%%%%%%%%%%%%%%%%%%%
\section{Setting the Stage and Main Ideas}\label{sec:stage}
%%%%%%%%%%%%%%%%%%%%%%%%%%%%%%%%%%%%%%%%%%%%%%%%%%%%%%%%%%%%%%%%%
%%%%%%%%%%%%%%%%%%%%%%%%%%%%%%%%%%%%%%%%%%%%%%%%%%%%%%%%%%%%%%%%%

Let us formalize the ideas spelled out in our introduction, before addressing an exactly-solvable model (the resolved conifold) in section~\ref{sec:conifold}, and then computationally addressing many different examples in section~\ref{sec:examples}, including the cases of local $\BP^2$, a diagonal slice of local $\BP^1 \times \BP^1$, some local curves, Hurwitz theory, and the quintic compact CY threefold. We begin with general expectations and what sort of structures we wish to unveil, to later materialize in our examples.

Going back to the topological-string asymptotic-series for the free energy \eqref{gstring-expansion}, let us describe it in the B-model as $F^{(0)}(g_{\text{s}};z,\bar{z})$. Here, the string coupling $g_{\text{s}}$ is also the resurgent variable, and the $(0)$ superscript specifies perturbative. The pair $(z,\bar{z})$ may be regarded as just external parameters, or interpreted as complex-structure moduli of the underlying CY threefold. The free energy is asymptotic, of Gevrey-1 type (see, \textit{e.g.}, \cite{glm08}),
\begin{equation}
F^{(0)} \simeq \sum_{g=0}^{+\infty} g_{\text{s}}^{2g-2} F^{(0)}_g, \qquad F^{(0)}_g \sim \Gamma(2g-1) \text{ as } g\to+\infty,
\end{equation}
\noindent
for generic values of $(z,\bar{z})$. Understanding the resurgent properties of $F^{(0)}$ and the role played by the moduli $(z,\bar{z})$ was the main purpose of \cite{cesv13, cesv14}. There, it was shown how to look for a transseries completion to the topological-string free energy of the form
\begin{equation}
F = \sum_{n=0}^{+\infty} \sigma^n\, \rme^{-n A(z)/g_{\text{s}}}\, F^{(n)} (g_{\text{s}}; z,\bar{z}),
\label{eq:transF1param}
\end{equation}
\noindent
where the (multi) instanton sectors $F^{(n)}(g_{\text{s}})$ are also given by asymptotic series. In particular, it was found---both generically and in examples---that \eqref{eq:transF1param} has several nonperturbative sectors, with associated actions $A_\alpha$, all of them holomorphic and determined by the CY geometry.

The transseries \eqref{eq:transF1param} was constructed by combining a nonperturbative interpretation of the holomorphic anomaly equations of \cite{bcov93} with the resurgence relations that transseries generically satisfy, such as, for example,
\begin{equation}
F^{(0)}_g (z,\bar{z}) \sim \frac{\Gamma(2g-1)}{A(z)^{2g-1}}\, F^{(1)}_0(z,\bar{z}), \qquad \text{as } g\to+\infty.
\label{eq:largeorderF0gF10}
\end{equation}
\noindent
Here $F^{(1)}_0$ is the first coefficient of the one-instanton series $F^{(1)}(g_{\text{s}})$ and $A(z)$ is one of the instanton actions (the smallest one in absolute value, for the particular value of $z$). Subleading corrections to \eqref{eq:largeorderF0gF10} lead to further multi-loop coefficients, $F^{(1)}_h$ with  $h=1,2,\ldots$. Generalizations of \eqref{eq:largeorderF0gF10}, now addressing the large-order behavior of the $F^{(n)}_g$ sequences, provide new constraints and relations between higher instanton coefficients. 

This route towards the construction of \eqref{eq:transF1param}, further developed in \cite{c15}, draws a rather complete picture of what a transseries for $F(g_{\text{s}})$ should look like. In principle, such a transseries should contain all nonperturbative information concerning the B-model, but also, via mirror symmetry \cite{hv00}, all A-model nonperturbative information. It is within this context that we shall set our attention upon structures of interest in algebraic and enumerative geometry, arising from the A-model set-up, in particular the case of enumerative GW invariants.

Let us spell out our strategy. The B-model construction \eqref{eq:transF1param} depends upon $(z,\bar{z})$, the complex-structure moduli. From the standpoint of resurgence, these moduli may be regarded as external parameters, without any resurgent properties by themselves. But upon mirror symmetry, they relate the B-model CY threefold $\widetilde{\CX}$, with complex structure $z$, to the A-model mirror-CY threefold $\CX$, with K\"ahler structure $t$. A functional relation $t = t(z)$ is then provided by the mirror map. This means that one may in fact compute the \textit{mirror transseries} to \eqref{eq:transF1param}, where its $F^{(0)}_g (t)$ components are nothing but the GW generating functions as in \eqref{alphaprime-expansion}. Let us next focus on these enumerative invariants in greater detail, with the goal of uncovering which resurgent properties they carry, either intrinsic or merely inherited from the free energy.

%%%%%%%%%%%%%%%%%%%%%%%%%%%%%%%%%%%%%%%%%%%%%%%%%%%%%%%%%%%%%%%%%
\subsection{Enumerative Gromov--Witten Invariants}
%%%%%%%%%%%%%%%%%%%%%%%%%%%%%%%%%%%%%%%%%%%%%%%%%%%%%%%%%%%%%%%%%

GW invariants count embeddings of Riemann surfaces of a given genus into a CY threefold $\CX$, attending to the homology class of the image of this map. Thus, GW invariants are labelled by $g\in\BN$, like the topological-string free energies, and $\beta\in H_2(\CX,\BZ)$,
\begin{equation}
N_{g,\beta} \in \BQ.
\end{equation}
\noindent
Akin to \eqref{alphaprime-expansion}, they show up in the A-model perturbative free-energies through the expansion
\begin{equation}
F^{(0)}_g = \sum_{\beta\in H_2(\CX,\BZ)} N_{g,\beta}\, Q^\beta.
\end{equation}
\noindent
Here we have used the mirror map to translate from complex structure moduli, $z_i$, to K\"ahler moduli, $t_i =: -\log Q_i$ (roughly, the mirror map is $Q_i = \CO(z_i)$). More precisely, if $\omega$ is the (complexified) K\"ahler form in $\CX$ and $[S_i]$, with $i = 1,2,\ldots, b_2(\CX)$, is a basis of $H_2(\CX,\BZ)$, then one finds $\beta = \sum_i n_i [S_i]$ and $t_i := \int_{[S_i]} \omega$, in which case we may denote $Q^\beta = \prod_i Q_i^{n_i} = \exp(-\sum_i n_i\,t_i)$. In order to simplify things in the following, we shall restrict to examples where $b_2(\CX)=1$, in which case the sum over homology classes simplifies to
\begin{equation}
F^{(0)}_g(t) = \sum_{d=1}^{+\infty} N_{g,d}\, Q^d.
\label{eq:F0gtGWQ}
\end{equation}
\noindent
The index $d$ is called the degree of the embedding. See, \textit{e.g.}, \cite{m05} for more details on the relation between the enumerative GW invariants and their A-model generating functions.

Now \eqref{eq:F0gtGWQ} is a convergent series in $Q$, which, in particular, implies that it is \textit{not} resurgent. Its non-vanishing radius of convergence is generically finite, due to a nearby singularity located at the so-called conifold locus \cite{gv95}. This convergence may already suggest that the factorial growth of the free energies $F^{(0)}_g$ in genus must somehow arise from a combined contribution of several different degrees. We shall next try to understand how this might come about.

%%%%%%%%%%%%%%%%%%%%%%%%%%%%%%%%%%%%%%%%%%%%%%%%%%%%%%%%%%%%%%%%%
\subsection{Growth of Enumerative Invariants in Degree and in Genus}
%%%%%%%%%%%%%%%%%%%%%%%%%%%%%%%%%%%%%%%%%%%%%%%%%%%%%%%%%%%%%%%%%

As we introduce most of our main ideas, let us illustrate them with (partial) results from upcoming diverse examples. The simplest such example is naturally attached to the resolved conifold, for which the free energies can be computed exactly (see, \textit{e.g.}, \cite{m04} for a review)
\begin{equation}
F^{(0),\tconi}_g(t) = (-1)^{g-1}\, \frac{B_{2g}}{2g \left(2g-2\right)!}\, \text{Li}_{3-2g} \left( \rme^{-t} \right), \qquad g\geq 2,
\end{equation}
\noindent
where $\text{Li}_p (x)$ is the polylogarithm function. This immediately yields all GW invariants as
\begin{equation}
N^{\tconi}_{g,d} = f^\tconi_g\, d^{2g-3}, \qquad f^\tconi_g := (-1)^{g-1}\, \frac{B_{2g}}{2g \left(2g-2\right)!}.
\label{eq:GWconi}
\end{equation}
\noindent
More interesting geometries we shall later address include the (non-compact) local $\BP^2$ and the (compact) quintic CY threefolds, for which there are no such closed-form expressions. Enumerative invariants may, nonetheless, be generated on the computer to see in more detail how they grow in degree and genus. An example of the sort of numbers we have to work with is show in figure~\ref{fig:sampleGWlocalP2}, in the instance of local $\BP^2$ (to be addressed in section~\ref{sec:P2}).

%%%%%%%%%%%%%%%%%%%%%%%%%%%%%%%%%%%%%%%%%%%%%%%%%%%%%%%%%%%%%%%%%
\begin{figure}[t!]
\begin{center}
\includegraphics[width=\linewidth]{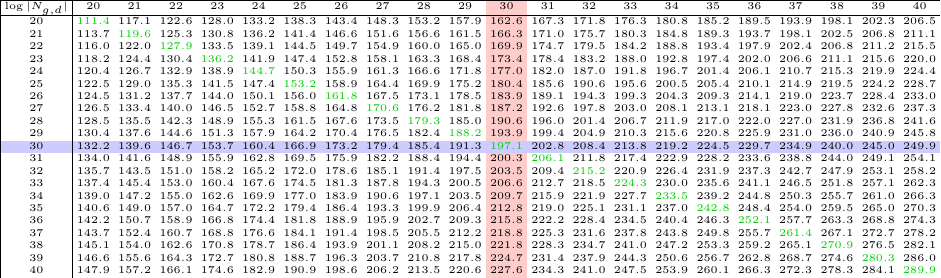}
\includegraphics[width=\linewidth]{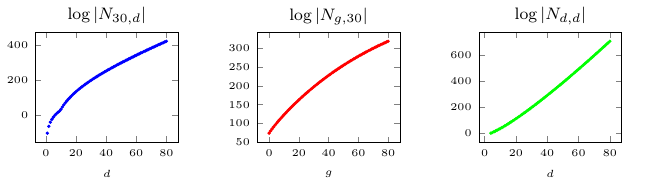}
\end{center}
\caption{Sample of GW invariants for the local $\BP^2$ CY threefold, alongside a visual representation of their growth with respect to degree $d$ (in blue), genus $g$ (in red), and a linear combination of the two (in green). Only in this latter case shall we find a factorial growth.}
\label{fig:sampleGWlocalP2}
\end{figure}
%%%%%%%%%%%%%%%%%%%%%%%%%%%%%%%%%%%%%%%%%%%%%%%%%%%%%%%%%%%%%%%%%

%%%%%%%%%%%%%%%%%%%%%%%%%%%%%%%%%%%%%%%%%%%%%%%%%%%%%%%%%%%%%%%%%
\subsubsection*{Growth in Degree}
%%%%%%%%%%%%%%%%%%%%%%%%%%%%%%%%%%%%%%%%%%%%%%%%%%%%%%%%%%%%%%%%%

Let us first consider the growth in degree at fixed genus. For the resolved conifold the answer is immediate from \eqref{eq:GWconi}: it is given by the degree $d$, raised to a linear function of the genus $g$, namely $2g-3$. For other, more intricate geometries the growth is similar but includes further parameters, such as a critical exponent $\gamma$ which captures distinct topological-string universality classes, \textit{i.e.}, distinct critical behaviors at the phase-transition point (see, \textit{e.g.}, \cite{cgmps06} for a discussion). In general one finds\footnote{Recall the notation where the boldface character specifies which index (if any) is the fixed one.} \cite{cdgp91, bcov93, kz99}
\begin{equation}
N_{\boldsymbol{g},d} \sim d^{(\gamma-2)(1-\boldsymbol{g})-1}\, \rme^{d t_{\text{c}}} \left( \log d \right)^{\alpha + \beta \boldsymbol{g}}, \qquad \text{as } d\to+\infty
\label{larged_general}
\end{equation}
\noindent
(further including a possibly $g$-dependent pre-factor). In this expression, $\rme^{-t_{\text{c}}}=Q_{\text{c}}$ marks the radius of convergence of $F^{(0)}_g$ on the $Q$-plane. Expression \eqref{larged_general} implies that the resolved conifold has $\gamma = 0$, being in the same universality class as, \textit{e.g.}, the local $\BP^2$ or the quintic CY threefolds. For example, for the quintic we have \cite{cdgp91, bcov93}
\begin{equation}
N^\tQuintic_{\boldsymbol{g},d} \sim d^{2\boldsymbol{g}-3}\, \rme^{d t_{\text{c}}} \left( \log d \right)^{2\boldsymbol{g}-2},
\end{equation}
\noindent
where $t_{\text{c}} = 7.58995\ldots$. For local $\BP^2$ this growth is illustrated in the leftmost plot of figure~\ref{fig:sampleGWlocalP2}, at fixed genus $g=30$. For very large degree $d$, the plotted curve must tend to a straight line of slope $|t_{\text{c}}|$. On the other hand, for the family of local curves $X_p = \CO(p-2)\oplus\CO(-p)\to \BP^1$ ($p\geq 3$) the critical exponent is instead $\gamma = -1/2$ \cite{cgmps06}, implying a distinct universality class and we shall discuss this example later in section~\ref{sec:curve}.

As mentioned earlier, the radius of convergence $Q_{\text{c}}$ signals a singularity of the (free energy) generating function. Such critical points correspond to the points in moduli space where the A-model geometric interpretation breaks down, with a phase transition taking place from the large-radius (geometric) phase to a non-geometric phase. Near such a singularity,
\begin{equation}
F^{(0)}_g \sim c_g \left( Q_{\text{c}}-Q \right)^{(1-g)(2-\gamma)}, \qquad g\ge2.
\label{F_larged_crit}
\end{equation}
\noindent
Nearby $Q_{\text{c}}$ all geometries within the same universality class will resemble each other, which implies that the coefficients $c_g$ are universal. For example, for $\gamma=0$ there is a double-scaling limit 
\begin{equation}
g_s \to 0, \quad Q \to Q_c, \qquad \text{with} \quad \kappa := g_{\text{s}} \left( Q_{\text{c}}-Q \right)^{-1} \quad \text{fixed},
\end{equation}
\noindent
such that 
\begin{equation}
F^{(0)}(g_{\text{s}};t) \to F^{(0)}_{\text{ds}} (\kappa) \simeq \sum_{g=2}^{+\infty} \frac{B_{2g}}{2g \left(2g-2\right)}\, \kappa^{2g-2},
\end{equation}
\noindent
which matches the $c=1$ string at self-dual radius \cite{gv95}. For other values of $\gamma$ the coefficients $c_g$ may be more complicated, being solutions to a nonlinear ODE such as Painlev\'e I, for example. 

%%%%%%%%%%%%%%%%%%%%%%%%%%%%%%%%%%%%%%%%%%%%%%%%%%%%%%%%%%%%%%%%%
\subsubsection*{Growth in Genus}
%%%%%%%%%%%%%%%%%%%%%%%%%%%%%%%%%%%%%%%%%%%%%%%%%%%%%%%%%%%%%%%%%

As we turn towards understanding the dependence of GW invariants on genus, at fixed degree, $N_{g,\boldsymbol{d}}$, it becomes useful to introduce the Gopakumar--Vafa (GV) invariants. These invariants are integer numbers, roughly counting the number of BPS states inside a CY threefold $\CX$, and resulting from a reorganization of the A-model free energy as introduced in \cite{gv98a, gv98c}. The complete result involves a Schwinger-type computation which rewrites the free energy as an index that counts string-theoretic BPS states via an M-theory uplift, and which finally yields
\begin{equation}
\sum_{g=0}^{+\infty} g_{\text{s}}^{2g-2} F^{(0)}_g (Q) = g_{\text{s}}^2\, c(t_i) + \ell(t_i) + \sum_{r=0}^{+\infty} \sum_\beta n_r^{(\beta)} \sum_{m=1}^{+\infty} \frac{1}{m} \left(2 \sin \frac{m g_{\text{s}}}{2} \right)^{2r-2} Q^{\beta m}.
\label{eq:GVexpansion}
\end{equation}
\noindent
Here, the $n_r^{(\beta)}\in\BZ$ are the GV invariants, labeled by the K\"ahler class $\beta$ and a spin index $r$. The polynomials $c(t_i)$ and $\ell(t_i)$ will play no role in the following.

It is straightforward to check that, generically, the GW invariants may be written explicitly in terms of the GV invariants as
\begin{equation}
N_{g,d} = \sum_{r=0}^g c_{r,g} \sum_{\beta|d} n_r^{(\beta)} \left( \frac{d}{\beta} \right)^{2g-3}, \qquad \text{using} \quad \left(2 \sin \frac{x}{2} \right)^{2r-2} =: \sum_{h=r}^{+\infty} c_{r,h}\, x^{2h-2}.
\label{eq:GWcGV}
\end{equation}
\noindent
In here we already find the $d^{2g-3}$ dependence which is characteristic of the resolved conifold. Now, an important property of the GV invariants, which will be useful in the following, is that for each degree $d$ there is a specific genus, $G(d)$, after which all these invariants vanish, \textit{i.e.}, $n_r^{(d)} = 0$ for $r>G(d)$ \cite{gv98c}. This function $G(d)$ is a polynomial in $d$, and this will simplify the dependence on $g$ in \eqref{eq:GWcGV} by replacing the upper limit in the $r$-sum; writing for all $g$
\begin{equation}
N_{g,d} = \sum_{r=0}^{G(d)} c_{r,g} \sum_{\beta|d} n_r^{(\beta)} \left( \frac{d}{\beta} \right)^{2g-3}.
\end{equation}
\noindent
In this expression the only remaining dependence upon the genus, $g$, lies in the coefficients $c_{r,g}$ and in the power of $d/\beta$. Since the coefficients $c_{r,g}$ are \textit{independent} of the CY geometry, we should expect a generic formula to hold for the large growth of $N_{g,\boldsymbol{d}}$ in genus. For example, as we shall discuss later on, for the case of local $\BP^2$ and degree $d=4$ we find
\begin{equation}
\label{quanticNd=4}
N^\tPTwo_{g,d=4} \sim (-1)^{g-1} \frac{B_{2g}}{2g \left(2g-2\right)!}\, 4^{2g-3} \left( 3 -\frac{6}{2^{2g-3}} - \frac{192}{4^{2g-3}} \right) + \frac{(-1)^{g-1}}{\left(2g-2\right)!}\, \frac{2^{2g-2}}{4} \left( 120 + \frac{336}{2^{2g-2}} \right),
\end{equation}
\noindent
This formula, involving Bernoulli numbers and factorials, is actually \textit{exact} for $g \geq 2$, not just a large-$g$ approximation. The first numbers ($3$, $-6$, $-192$) can be recognized as the GV invariants $n_0^{(1)}$, $n_0^{(2)}$, and $n_0^{(4)}$ for local $\BP^2$, whereas the other ($120$, $336$) are more complicated combinations involving higher-genera invariants. As such, in general, we can expect the following formula to hold (see appendix~\ref{sec:abc_coefficients} for a proof)
\begin{equation}
N_{g,d} = f^\tconi_g \left\{ \sum_{n|d} a_n \left( \frac{d}{n} \right)^{2g-3} + \frac{2g}{B_{2g}}\, \frac{1}{d} \left( c_d\, \delta_{g,1} + \sum_{n=1}^{G(d)-1} b_{d,n}\, n^{2g-2} \right) \right\}.
\label{eq:GW_abc}
\end{equation}
\noindent
where $a_d \equiv n_0^{(d)}$ and $b_{d,n}, c_d \in \BZ$. In this expression the dependence on the genus $g$ is explicit---one could even plug-in non-integer values of the genus after analytically continuing the Bernoulli numbers. If we fix the degree, as in $N_{g,\boldsymbol{d}}$, it is then simple to see that the leading growth in genus is exponential, $\boldsymbol{d}^{2g-3}$, with further subleading exponential and inverse-of-factorial corrections in $g$. The second plot in figure \ref{fig:sampleGWlocalP2} illustrates this genus dependence, at fixed degree $d=30$, for local $\BP^2$. The plotted curve is asymptotic to a straight line of slope $2 \log \boldsymbol{d}$.

Expression \eqref{eq:GW_abc} shows how the contribution of GW invariants, $N_{g,\boldsymbol{d}}$, to the free energies at a fixed single degree, $\boldsymbol{d}$,  again cannot be responsible for the factorial growth we need to find. In this way, the only option we have left to find the $\sim (2g)!$ factorial growth of the free energies, encoded in the GW invariants, is to address the \textit{combined} growth in genus and degree.

%%%%%%%%%%%%%%%%%%%%%%%%%%%%%%%%%%%%%%%%%%%%%%%%%%%%%%%%%%%%%%%%%
\subsubsection*{Combined Growth in Genus and Degree}
%%%%%%%%%%%%%%%%%%%%%%%%%%%%%%%%%%%%%%%%%%%%%%%%%%%%%%%%%%%%%%%%%

Upon a second look at the (already familiar) characteristic behavior of GW invariants in $d^{2g-3}$, it should be straightforward to deduce that when $d$ and $g$ are linearly related, then the factorial growth is immediately realized. The link is the classical Stirling approximation,
\begin{equation}
n^n \sim \frac{n!\, \rme^n}{\sqrt{2\pi n}}.
\end{equation}
\noindent
Consider one more time the example of the resolved conifold in \eqref{eq:GWconi}, and assume the dependence $d = a_0 + a_1 g$ for some values of $a_0$ and $a_1$. Then, to leading order in $g$, one finds
\begin{equation}
N^\tconi_{g,d=a_0+a_1g} \sim \frac{\Gamma \left(2g-\frac{3}{2}\right)}{\left(\frac{4\pi}{\rme\, a_1}\right)^{2g-\frac{3}{2}}}\, \frac{\left( \frac{2\, \rme}{a_1} \right)^{\frac{3}{2}} \rme^{2 \frac{a_0}{a_1}}}{2\pi^2}.
\label{eq:loGWconia0a1}
\end{equation}
\noindent
The factorial in $g$ is now explicit and it comes from the term $d^{2g-3}$ when $d = a_0+a_1g$. On the other hand, recall that the leading growth of the free energy $F^{(0),\tconi}_g$ in this case is \cite{ps09}
\begin{equation}
F^{(0),\tconi}_g (Q) \sim \frac{\Gamma \left(2g-1\right)}{\left(2\pi t\right)^{2g-1}}\, \frac{t}{\pi},
\label{eq:FconiQexpansion}
\end{equation}
\noindent
with instanton action $A=2\pi t$. To connect this resurgence relation to the one in \eqref{eq:loGWconia0a1}, one has to recall the definition of GW invariants \eqref{eq:F0gtGWQ}
\begin{equation}
F^{(0),\tconi}_g = \sum_{d=1}^{+\infty} N^\tconi_{g,d}\, Q^d,
\label{eq:extraQd}
\end{equation}
\noindent
and then notice that the largest contribution to this sum, for a fixed value of $Q$ on the right-hand side, comes from
\begin{equation}
\frac{\partial}{\partial d} \left( N^\tconi_{g,d} Q^d \right) = 0 \qquad \Rightarrow \qquad d = \frac{2g-3}{t}.
\label{eq:coni_saddle_point}
\end{equation}
\noindent
So we should expect that taking $a_1 = 2/t$ and $a_0=-3/t$ in \eqref{eq:loGWconia0a1} will reproduce something resembling \eqref{eq:FconiQexpansion}. Indeed, one can easily check that we obtain the same instanton action as $\frac{4\pi}{a_1} = 2\pi t$ (where we ignore the exponential factor in \eqref{eq:loGWconia0a1} as such terms should be regrouped into the factor of $Q^d$ in \eqref{eq:extraQd}; \textit{i.e.}, exponentials may be ignored when matching with \eqref{eq:FconiQexpansion}).

This strategy of selecting the leading contribution from the $Q$-expansion inside $F^{(0)}_g (Q)$ can be pushed further. One way to do so is to approximate the sum over the degree by an integration, and then perform a saddle-point approximation---and this will be a main theme throughout our analyses. Consider the following saddle-point approximation around $x = x_0$ (where $V'(x_0)=0$),
\begin{equation}
\Phi (\lambda) = \int_0^{+\infty} \rmd x\, \rme^{\lambda V(x)} \sim \rme^{\lambda V(x_0)}\, \sqrt{-\frac{2\pi}{\lambda V''(x_0)}} \left( 1 + \CO\left( \frac{1}{\lambda} \right) \right).
\label{eq:saddle_point_approx_H}
\end{equation}
\noindent
To apply this generic formula to our problem one just has to identify
\begin{equation}
\Phi (\lambda) \leftrightarrow F^{(0)}_g (Q), \qquad \rme^{\lambda V(x)} \leftrightarrow N_{g,x}\, Q^x, \qquad x \leftrightarrow d, \qquad \lambda \propto g.
\end{equation}
\noindent
The only subtlety in this identification is that the saddle-point $x_0$ is also proportional to the coupling $\lambda$, as we saw for the resolved conifold \eqref{eq:coni_saddle_point}. In any case, our goal is to solve for $\rme^{\lambda V(x_0)}$ in \eqref{eq:saddle_point_approx_H}. Then, the only obstacle we have in order to do so is knowing the value of $V''(x_0)$. For the resolved conifold we had an explicit formula and, as such, we knew that it was $-t/(a_0+a_1 g)^2$ where $x_0 = a_0+a_1 g$ and $\lambda$ chosen the same; but in general there are no such explicit formulae. Nonetheless, let us postulate a completely similar dependence in $g$, namely
\begin{equation}
V''(x_0) \equiv -\frac{a_2(Q)}{\lambda^2} + \CO \left(\frac{1}{\lambda^3}\right),
\end{equation}
\noindent
where we have chosen the explicit relation $\lambda = a_0 + a_1 g$, and introduced the function $a_2(Q)$. If one makes further use of the leading large-order growth of the free energies \cite{cesv13},
\begin{equation}
\Phi (\lambda) = F^{(0)}_g(Q) \sim \frac{\Gamma \left(2g-\beta\right)}{A^{2g-\beta}}\, F^{(1)}_0,
\end{equation}
\noindent
we finally obtain
\begin{equation}
\left. N_{g,x_0}\, Q^{x_0} \right|_{x_0=a_0+a_1 g} \sim \frac{\Gamma \left(2g-\beta-\frac{1}{2}\right)}{A^{2g-\beta-\frac{1}{2}}} \left(\frac{a_2}{\pi a_1 A} \right)^{\frac{1}{2}} F^{(1)}_0.
\label{eq:GW_saddle_point_growth}
\end{equation}
\noindent
Note that this large-order relation depends on $a_1$ and $a_2$, functions of $Q$ which define the position and shape of the saddle. For the resolved conifold, and even for other geometries with actions proportional to a K\"ahler parameter, we find that $a_1 = 2/t$ and $a_2 = t$. But for general geometries we do not know what these functions are or should be, and one has to run computational experiments in order to judiciously try to fix them. Note that once one approximates the sum over the degree by an integration, then different saddles will correspond to different leading actions, which may depend on the value of $Q$. For the resolved conifold there is only one leading action and one saddle. But for general geometries we can expect several of them---albeit one is always proportional to the K\"ahler parameter $t$. This is illustrated in figure~\ref{fig:leading_degrees}, where we have plotted saddles for the resolved conifold and local $\BP^2$ (we shall discuss these plots in greater detail later on). The saddles are identified by numerically selecting, at fixed values of $g$ and $t$ but varying $d$, the GW invariants which contribute the most to the perturbative free energy. Both models clearly show a saddle associated to a K\"ahler action, with $A= 2\pi t$. For local $\BP^2$ there is one further saddle, related to a conifold action, to be discussed in section~\ref{sec:P2}.

%%%%%%%%%%%%%%%%%%%%%%%%%%%%%%%%%%%%%%%%%%%%%%%%%%%%%%%%%%%%%%%%%
\begin{figure}[t!]
\begin{center}
\includegraphics[width=0.46\textwidth]{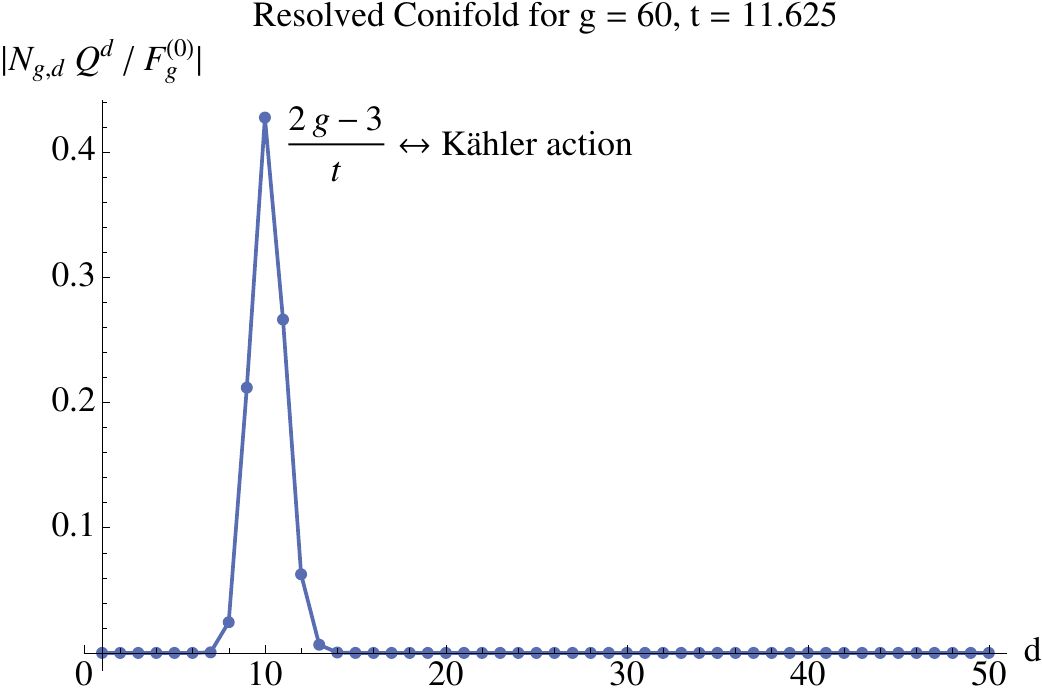}
\hspace{0.05\textwidth}
\includegraphics[width=0.46\textwidth]{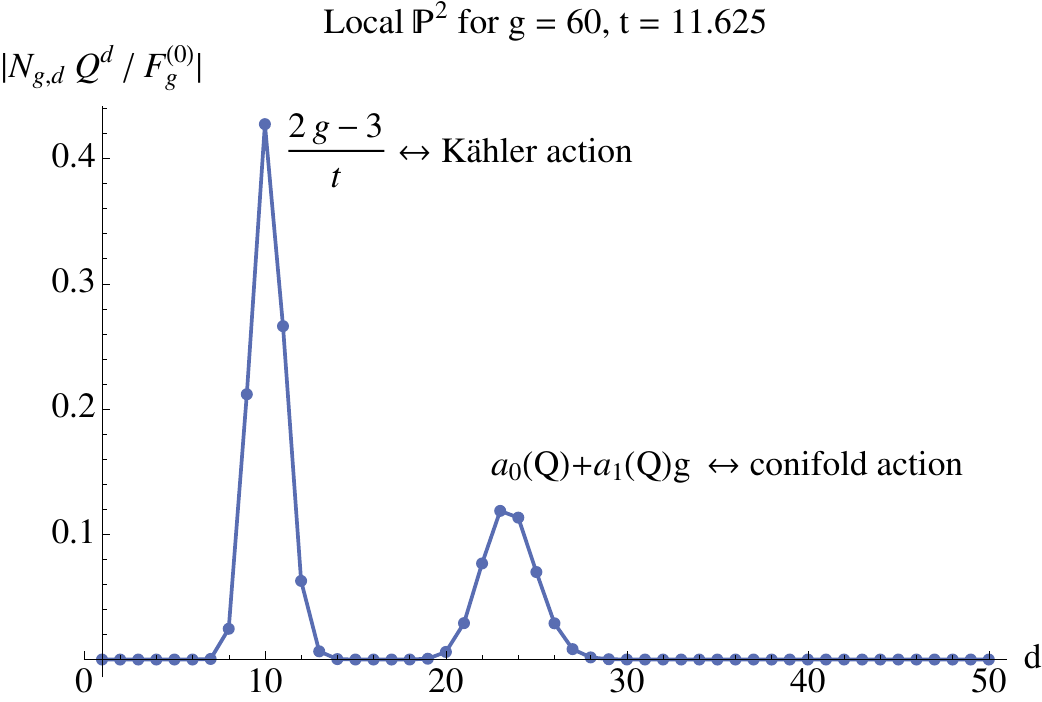}
\end{center}
\caption{Graphical representation of which GW invariants contribute the most to a free energy $F^{(0)}_g (Q)$, for fixed values of $g$ and $Q=\rme^{-t}$. This contribution is estimated by comparing the value of $N_{g,d}\, Q^d$, at different values of the degree $d$, against the total value of the genus-$g$ perturbative free energy $F_g^{(0)}$. The resolved conifold, portrayed on the left, has a single saddle-point corresponding to the action $A=2\pi t$; whereas for local $\BP^2$, portrayed on the right, an extra saddle-point attached to the conifold action is also present. These saddles may exchange dominance depending on the value of $Q$, but the set of leading degrees will always be in correspondence with the set of leading instanton actions.}
\label{fig:leading_degrees}
\end{figure}
%%%%%%%%%%%%%%%%%%%%%%%%%%%%%%%%%%%%%%%%%%%%%%%%%%%%%%%%%%%%%%%%%

%%%%%%%%%%%%%%%%%%%%%%%%%%%%%%%%%%%%%%%%%%%%%%%%%%%%%%%%%%%%%%%%%
\subsubsection*{The Main Question}
%%%%%%%%%%%%%%%%%%%%%%%%%%%%%%%%%%%%%%%%%%%%%%%%%%%%%%%%%%%%%%%%%

Let us finally address the main question motivating this paper. Are there nonperturbative extensions of the enumerative GW invariants---denote them by ``$N^{(n)}_{g,d}$'', with $n$ an ``instanton label''---just like there are nonperturbative extensions $F^{(n)}_g$ of the perturbative free energy? And if so, what is their enumerative interpretation, \textit{i.e.}, which counting problem is associated to these new numbers? An argument in favor of an affirmative answer arises from considering the A-model mirror to the B-model resurgent analysis of the free energy, and its associated transseries constructions \cite{cesv13, cesv14}. This procedure would certainly lead to an A-model transseries (and we shall illustrate this in the example of local $\BP^2$), but the resulting transseries would not be an adequate generating function. In fact, while the perturbative $F_g (Q)$ is a natural generating function, collecting the GW invariants as a $Q$-expansion, the higher instanton sectors $F^{(n)}_g (t)$ will not be regular at $Q=0$ and a naive $Q$-expansion is no longer an option. Then how do we extract the nonperturbative counterparts?

Schematically, we want to make sense of the following diagram
\begin{equation}
\begin{tikzpicture}[node distance=2cm,auto]
  \node (F0g) {$F^{(0)}_g$};
  \node (Fng) [node distance=6.0cm,right of=F0g] {$F^{(n)}_g$};
  \node (Nngd) [node distance=2.0cm,below of=Fng] {``$N^{(n)}_{g,d}$''?};
  \node (Ngd) [node distance=6.0cm,left of=Nngd] {$N_{g,d}$};
  \draw[->] (F0g) to node {{\small resurgence}} (Fng);
  \draw[->] (Fng) to node {{\small expansion?}} (Nngd);
  \draw[->] (F0g) to node [swap] {{\small $Q$-expansion}} (Ngd);
  \draw[thick,densely dashed][->] (Ngd) to node [swap] {{\small resurgence + interpretation?}} (Nngd);
\end{tikzpicture}
\label{eq:resurgene_Qexpansion_diag}
\end{equation}
\noindent
The left and upper arrows are well understood. The left arrow is just the A-model definition of GW invariants, while the upper arrow was made precise within the B-model set-up in \cite{cesv13, cesv14}. In this paper we try to take the first (exploratory) steps towards the definition of lower and right arrows, but a complete answer can only come with a geometric/enumerative interpretation of these conjectured quantities ``$N^{(n)}_{g,d}$'', which is beyond the scope of the present work.

What could these nonperturbative invariants ``$N^{(n)}_{g,d}$'' be counting? With the nonperturbative sectors in the topological-string transseries naturally associated to D-brane sectors \cite{msw07, ps09}, one possibility is that their counting is associated to  embeddings of Riemann surfaces with boundaries, of certain genus and degree. These A-brane open-strings end in middle-dimensional lagrangian submanifolds (see, \textit{e.g.}, \cite{ooy96}) and it is also possible that some counting associated to the corresponding target-space wrappings would play a part. But this type of counting is usually associated to the \textit{open} GW invariants (see, \textit{e.g.}, \cite{gz01}), which means that---should this be the correct interpretation of the nonperturbative invariants---there may be an interesting link between closed and open GW invariants arising from relating perturbative and nonperturbative data in the topological-string transseries. Of course topological-string D-branes also relate to more intricate mathematical constructions, such as Fukaya and derived categories (see, \textit{e.g.}, \cite{a04}), in which case the counting associated to the ``$N^{(n)}_{g,d}$ invariants'' may be much more complicated.

Further note that, as GW invariants themselves have no transseries completions, we do not expect the lower arrow to be defined directly but rather as combination of left, upper, and right arrows (alongside the mirror map). In this way, one will have to extract the ``$N^{(n)}_{g,d}$ invariants'' directly out of the nonperturbative sectors $F^{(n)}_g$. Now, the $Q$-expansion of the perturbative sector arises from a worldsheet-instanton expansion and thus naturally relates to a counting problem. But the nonperturbative sectors lack such power-series expansions in $Q$ (we shall soon illustrate in a couple of examples how they have singularities at $Q=0$), implying that any nonperturbative GW invariants hidden inside the nonperturbative free-energies might be difficult to extract and their enumerative interpretation harder to decode (at the very least they will imply understanding how the aforementioned singularities come about, upon use of the mirror map). Furthermore, even after performing an asymptotic resurgent analysis of $N_{g,d}$, we have to disentangle the dependence in $t$, coming from the parameters $a_0$ and $a_1$, in the linear dependence between degree and genus. At the end of the day, this leaves the right arrow to be defined. What one has to do is to understand, via mirror symmetry, how to relate nonperturbative multi-loop multi-instanton coefficients in the B-model transseries, to the nonperturbative sectors appearing in the asymptotic resurgence relations for the combined genus/degree growth of GW invariants.

Our goal in this paper is to initiate this line of research, computationally exploring diverse CY examples. We try to identify the structure of these new invariants, as they are encoded in the nonperturbative content of the A-model free energy, but shall leave open their subsequent enumerative interpretation for future research.

%%%%%%%%%%%%%%%%%%%%%%%%%%%%%%%%%%%%%%%%%%%%%%%%%%%%%%%%%%%%%%%%%
%%%%%%%%%%%%%%%%%%%%%%%%%%%%%%%%%%%%%%%%%%%%%%%%%%%%%%%%%%%%%%%%%
\section{An Exactly-Solvable Model: The Resolved Conifold}\label{sec:conifold}
%%%%%%%%%%%%%%%%%%%%%%%%%%%%%%%%%%%%%%%%%%%%%%%%%%%%%%%%%%%%%%%%%
%%%%%%%%%%%%%%%%%%%%%%%%%%%%%%%%%%%%%%%%%%%%%%%%%%%%%%%%%%%%%%%%%

This section addresses our first example, concerning an exactly solvable model: the resolved conifold. This toric variety is a non-compact CY threefold which is the total space of the bundle $\CO (-1) \oplus \CO (-1) \to \BP^1$. The perturbative free-energy for the resolved conifold can be computed exactly to all orders in the genus expansion (see, \textit{e.g.}, \cite{m04} and references therein). This of course translates to the fact that one may obtain analytical expressions for all its (infinite) GW invariants \cite{fp98}. For any genus $g$, the results are
\begin{align}
F^{(0)}_0 &= \frac{t^3}{12} - \frac{\pi^2t}{6} + \zeta(3) - \text{Li}_3 \left(\rme^{-t}\right), \\ 
F^{(0)}_1 &= - \frac{t}{24} +\zeta'(-1) + \frac{1}{12}\, \text{Li}_1 \left(\rme^{-t}\right), \\
F^{(0)}_g &= \frac{B_{2g}B_{2g-2}}{2g \left(2g-2\right) \left(2g-2\right)!} + (-1)^{g-1}\, \frac{B_{2g}}{2g \left(2g-2\right)!} \text{Li}_{3-2g} \left(\rme^{-t}\right), \qquad g \ge 2,
\end{align}
\noindent
where $\text{Li}_{p} (z)$ is the polylogarithm of index $p$. In the following we will drop the contribution from the constant map \cite{mm98, fp98} and mostly focus on the large-order contributions $g \ge 2$.

Due to the polylogarithm these free energies grow factorially in the genus and lead to an asymptotic, Gevrey-1 perturbative free-energy \cite{glm08}. The resurgent properties of this series have been studied in detail in \cite{ps09, ars14}, with the result
\begin{equation}
F^{(0)}_g \sim \sum_{n=1}^{+\infty} \sum_{m\in\BZ} \left\{ \frac{\Gamma \left(2g-1\right)}{\left( n A_m \right)^{2g-1}}\, \frac{A_m}{2\pi^2 n} + \frac{\Gamma \left(2g-2\right)}{\left( n A_m \right)^{2g-2}}\, \frac{1}{2\pi^2 n^2} \right\},
\label{largeo_coni}
\end{equation}
\noindent
where $A_m (t) = 2\pi \left( t + 2\pi\rmi m \right)$ are the instanton actions. For our purposes, we shall focus on the \textit{leading} contribution, whose action is $A=2\pi t$, in which case
\begin{equation}
F^{(0)}_g \sim \frac{\Gamma \left(2g-1\right)}{A^{2g-1}}\, \frac{A}{2\pi^2}. 
\label{eq:resconifold_F0g_F10}
\end{equation}
\noindent
Let us next translate these resurgent properties to the level of GW invariants.

The GW invariants for the resolved conifold can be immediately read from the free energies, by simply expanding the polylogarithm in power series. One finds
\begin{equation}
N_{g,d}^\tconi = f^\tconi_g\, d^{2g-3},
\end{equation}
\noindent
where $f^{\tconi}_g$ includes the Bernoulli dependence and is defined in \eqref{eq:GWconi}. These invariants have such a simple form, given that they are actually generated by a single non-vanishing GV invariant
\begin{equation}
n_0^{(1)} = 1.
\end{equation}
\noindent
Likewise, the $abc$-coefficients we introduced in \eqref{eq:GW_abc} vanish except for the one which equals the GV invariant, $a_1 = n_0^{(1)} =1$. This makes this geometry considerably simpler than the ones we shall explore later, allowing for an analytic treatment whose features will also show up later. 

As we anticipated in some detail in the previous section, the factorial growth of the free energy arises from the term $d^{2g-3}$ when $d$ grows linearly with $g$. This is completely precise when the degree is a saddle point, in the sense explained earlier. This point, $d = (2g-3)/t$, was computed in equation \eqref{eq:coni_saddle_point} and graphically represented in figure~\ref{fig:leading_degrees} (left plot). One caveat about the saddle-point approximation is that generically the saddle-point lands on non-integer values of the degree. In order to be able to do numerical analyses with actual GW invariants, we look at nearby (integer) values of the degree. A practical computational choice, one that we will also use for other geometries, is to set
\begin{equation}
g = \frac{t}{2}d + \sh, \quad \text{with} \quad -\floor*{\frac{t}{4}-\frac{3}{2}} \leq \sh \leq \floor*{\frac{t}{4}+\frac{3}{2}},
\label{eq:def_p}
\end{equation}
\noindent
and set $t$ to an even integer (and $q$ must consequentially be an integer).

Since in this example there is an analytic expression for all GW invariants, we can use it to obtain the resurgence relation
\begin{equation}
\left. f^\tconi_g\, d^{2g-3}\, Q^d \right|_{g=\frac{t}{2}d+\sh} \sim \sum_{h=0}^{+\infty}  \frac{\Gamma \left(2g-\frac{3}{2}-h\right)}{\left( 2\pi t \right)^{2g-\frac{3}{2}-h}}\,  \frac{t^{\frac{3}{2}-h}}{2^{2h+1}\, \pi^{h+2}}\, \polyname_h(\sh).
\label{eq:GW_coni_large_order}
\end{equation}
\noindent
This expression is obtained by making use of the following asymptotic formulae (for large degree $d$ and genus $g$):
\bea
\left. f^\tconi_g \right|_{g=\frac{t}{2}d+\sh} &\sim& 2\, \frac{2g-1}{\left(2\pi\right)^{2g}}\, \zeta(2g), \\
\left. d^{2g-3} \right|_{g=\frac{t}{2}d+\sh} &\sim& \frac{\rme^{t d}}{\sqrt{2\pi}}\,  \frac{\Gamma \left(2g-\frac{5}{2}\right)}{t^{2g-3}}\, \sum_{n=0}^{+\infty} \frac{\varrho_n(\sh)}{\left(t d\right)^n},
\eea
\noindent
where $\varrho_n(p)$ is a polynomial of degree $2n$ in $\sh$. As to the polynomials $\polyname_h(\sh)$ in \eqref{eq:GW_coni_large_order} above, these are polynomials in $\sh$ of degree $2h$ with rational\footnote{This depends on the way they are presently written. It is simple to note that the denominators in the first terms of the polynomials, the numbers $(1,12,288,51840)$, correspond to the denominators that appear in the asymptotic expansion of the Gamma-function. In particular, we could rewrite \eqref{polyname0} through \eqref{polyname3} by pulling out these overall factors, in which case we would then find a set of polynomials with integer coefficients.} coefficients. The first of which are
\bea
\label{polyname0}
\polyname_{0}(\sh) &=& 1, \\
\label{polyname1}
\polyname_{1}(\sh) &=& -\frac{71}{12}+12 \sh-4 \sh^2, \\
\label{polyname2}
\polyname_{2}(\sh) &=& \frac{11545}{288}-131 \sh+\frac{419 \sh^2}{3}-\frac{176 \sh^3}{3}+8 \sh^4, \\
\label{polyname3}
\polyname_{3}(\sh) &=& -\frac{17534803}{51840}+\frac{33553 \sh}{24}-\frac{157393 \sh^2}{72}+\frac{15220 \sh^3}{9}-\frac{2062 \sh^4}{3}+\frac{416 \sh^5}{3}-\frac{32\sh^6}{3},
\eea
\noindent
and in general they are such that they make the following asymptotic expansion hold for any $\sh$ as $x\to+\infty$,
\begin{equation}
\sqrt{2\pi}\, \rme^{2\sh-x} \left(x-1\right) \left(x-2\sh\right)^{x-3} \sim \sum_{h=0}^{+\infty} \Gamma \left(x-\frac{3}{2}-h\right) 2^{-h}\, \polyname_h (\sh).
\label{eq:CP_h_polynomials}
\end{equation}
\noindent
Note that expression \eqref{eq:GW_coni_large_order} conforms to the usual resurgence relations, which, for some generic free-energy perturbative expansion, look like (see, \textit{e.g.}, \cite{abs16} for an introduction) 
\bea
F_{g}^{(0)} &\sim& \sum_{n=1}^{+\infty} \frac{\Gamma \left( 2g-n\beta \right)}{\left( n A \right)^{2g-n\beta}}\, \frac{S_1^n}{2\pi\rmi}\, \sum_{h=0}^{+\infty} \frac{\Gamma \left( 2g-n\beta-h \right)}{\Gamma \left( 2g-n\beta \right)}\, F_{h}^{(n)} \left( n A \right)^{h} = \\
&=&
\frac{\Gamma \left( 2g-\beta \right)}{A^{2g-\beta}}\, \frac{S_1}{2\pi\rmi} \left( F_{0}^{(1)} + \frac{A}{2g-\beta-1}\, F_{1}^{(1)} + \frac{A^{2}}{\left( 2g-\beta-1 \right) \left( 2g-\beta-2 \right)}\, F_{2}^{(1)} + \cdots \right) + \nonumber \\
&+&
\frac{\Gamma \left( 2g-2\beta \right)}{\left( 2 A \right)^{2g-2\beta}}\, \frac{S_1^2}{2\pi\rmi} \left( F_{0}^{(2)} + \frac{2 A}{2g-2\beta-1}\, F_{1}^{(2)} + \cdots \right) + \mathcal{O}(3^{-2g}). \nonumber
\eea
\noindent
Indeed, in \eqref{eq:GW_coni_large_order} one immediately identifies the $\sim \left( 2g \right)!$ growth, alongside the instanton action $A = 2\pi t$ which is the \textit{same}  action that appears in the free energies. Of course \eqref{eq:GW_coni_large_order} also has higher instanton corrections which improve the asymptotics further as in the above expression. These arise from including the (exponentially) subleading terms in $\zeta(2g) = \sum_{n=1}^{+\infty} n^{-2g}$ in the large-$g$ expansion of the Bernoulli numbers, $B_{2g}$, in \eqref{eq:GWconi}; and from  computing the complete large-$d$ transseries expansion of $d^{td+2\sh-3}$. The result is
\begin{equation}
\left. f^\tconi_g\, d^{2g-3}\, Q^d \right|_{g = \frac{t}{2}d+\sh} \sim \sum_{n=1}^{+\infty
} \sum_{h=0}^{+\infty} \frac{\Gamma \left(2g-\frac{3}{2}-h\right)}{\left( n A \right)^{2g-\frac{3}{2}-h}}\, \frac{t^{\frac{3}{2}-h}}{2^{2h+1}\, \pi^{h+2}\, n^{\frac{3}{2}+h}}\, \polyname_h (\sh).
\end{equation}

Some computational tests on the validity of \eqref{eq:GW_coni_large_order} are shown in figures~\ref{fig:conifold_Kahler_Action} and~\ref{fig:conifold_LOOPS}. Figure~\ref{fig:conifold_Kahler_Action} presents a test of the instanton action. We plot the analytical $A_{\text{K}} = 2\pi t$ against numerical tests of this action. We use the standard techniques of Richardson extrapolation/Richardson transforms\footnote{Given a $k$-sequence $\BS(k) \simeq s_0 + \frac{s_1}{k} + \frac{s_2}{k^2} + \cdots$, its n-th Richardson transform is defined as
\be
\text{RT}_{\BS} (k,n) = \sum_{m=0}^n (-1)^{m+n}\, \frac{(k+m)^n}{m!(n-m)!}\, \BS(k+m).
\ee
\noindent
Convergence is accelerated by cancellation of subleading terms in the original sequence up to $k^{-n}$ order.} to accelerate convergence (similar to tests done in, \textit{e.g.}, \cite{msw07}). An
illustration of these transforms for different values of $t$ and $q$ is shown in the inclosed figures, where the original sequence is shown in blue together with its first (green), second (yellow) and third (red) Richardson transforms. We do this for varying $t$ (the horizontal axis) but also varying $q$, \textit{i.e.}, each black dot is actually several overlapping black dots, each one the third Richardson transform of the numerical sequence for the instanton action, for that particular value of $t$ and for a range of different values of $q$. Then figure~\ref{fig:conifold_LOOPS} tests the validity of \eqref{polyname0} through \eqref{polyname3} (in fact up to $h=5$), this time around for fixed $q$. Each inverted triangle in the plot is again the third Richardson transform of the tested sequence. All these plots very cleanly illustrate the validity of \eqref{eq:GW_coni_large_order}.

%%%%%%%%%%%%%%%%%%%%%%%%%%%%%%%%%%%%%%%%%%%%%%%%%%%%%%%%%%%%%%%%%
\begin{figure}[t!]
\centering
\includegraphics[width=0.55\textwidth]{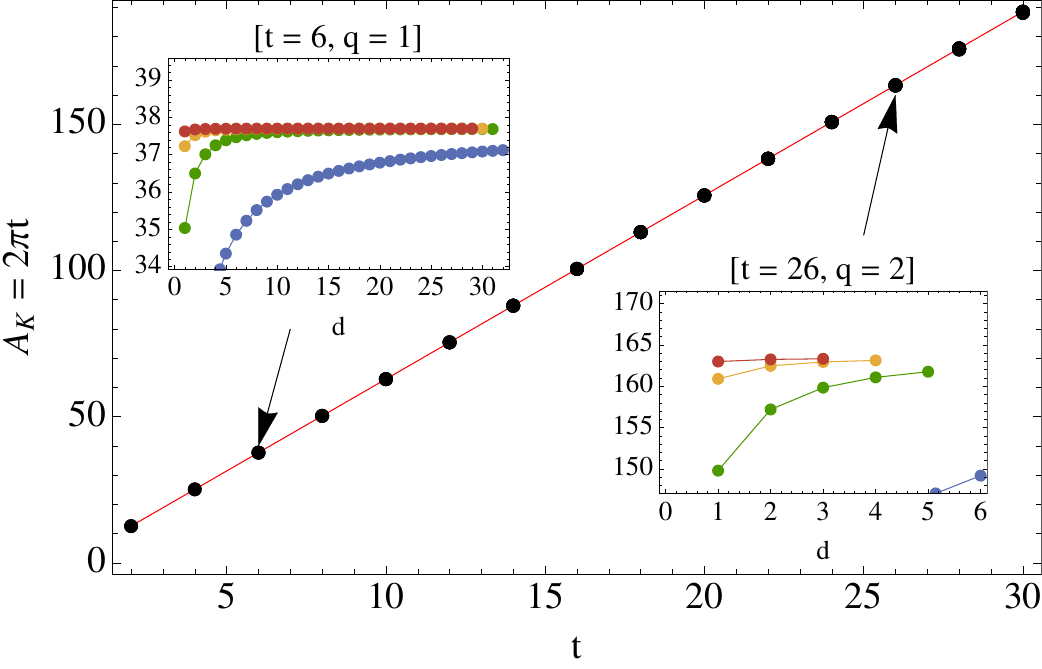}
\caption{Test of the instanton action $A_{\text{K}} = 2\pi t$, from the K\"ahler saddle-point, for the resolved conifold. The inclosed plots show the convergence for a couple of values of $t$ and $\sh$.}
\label{fig:conifold_Kahler_Action}
\end{figure}
%%%%%%%%%%%%%%%%%%%%%%%%%%%%%%%%%%%%%%%%%%%%%%%%%%%%%%%%%%%%%%%%%

%%%%%%%%%%%%%%%%%%%%%%%%%%%%%%%%%%%%%%%%%%%%%%%%%%%%%%%%%%%%%%%%%
\begin{figure}[t!]
\centering
\includegraphics[width=0.7\textwidth]{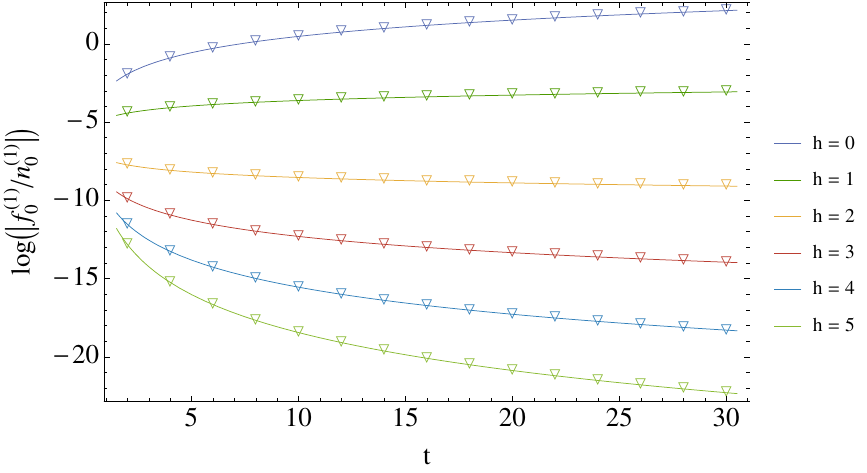}
\caption{Numerical check of the loop-corrections $f^{(1)}_h := \frac{n^{(1)}_0 t^{\frac{3}{2}-h}}{2^{2h+1} \pi^{h+2}}\, \CP_h(\sh)$, for $h=0,\ldots,5$ and $\sh=1$, for the resolved conifold. We plot the logarithm of the ratio $f^{(1)}_0/n^{(1)}_0$ so that all curves fit within the same graph (and where the GV invariant is $n^{(1)}_0 = 1$).}
\label{fig:conifold_LOOPS}
\end{figure}
%%%%%%%%%%%%%%%%%%%%%%%%%%%%%%%%%%%%%%%%%%%%%%%%%%%%%%%%%%%%%%%%%

In section~\ref{sec:stage} we showed how to relate GW asymptotics to free-energy instanton sectors, in particular relating the first term in the right-hand side of \eqref{eq:GW_coni_large_order} with the one-loop one-instanton free energy $F^{(1)}_0$; see \eqref{eq:GW_saddle_point_growth} and the discussion which follows. Ideally, one would now like to do the same for the multi-loop (eventually multi-instanton) one-instanton free energies and their relation to higher terms in \eqref{eq:GW_coni_large_order}. Unfortunately, already finding a direct relation at two-loops, between $F^{(1)}_1$ and any higher term in \eqref{eq:GW_coni_large_order}, turns out not to be possible using the saddle-point approximation from section~\ref{sec:stage}. In fact, our saddle-point approach is non-standard, in the sense that the saddle-point itself grows with $d$ (or $g$), which essentially obscures a clear-cut relation between free-energy asymptotics and GW asymptotics beyond the first term. For the present example of the resolved conifold we can bypass this problem, working directly with the explicit form of the GW invariants, but this will not be possible for more complicated examples.

Let us end our discussion of the resolved conifold by going back to our diagram \eqref{eq:resurgene_Qexpansion_diag}. As we mentioned earlier, one cannot find nonperturbative GW invariants via a $Q$-expansion of the resurgent asymptotic expansion for the perturbative free energy $F^{(0)}_g$. This is already clear in equation \eqref{eq:resconifold_F0g_F10}, where, although the left-hand side does have a regular expansion around $Q=0$ from which one reads the GW invariants (this is just the statement that the left-hand side is a regular generating function), the same does not hold true for the right-hand side, where one finds a logarithmic singularity at that same point (recall that the instanton action is proportional to $t$). In other words, the ``resurgence rewriting'' of the perturbative free energies, $F^{(0)}_g(Q)$, as an asymptotic series in $1/g$ does not respect, term by term, a regular $Q$-expansion. Only when we consider all corrections in $1/g$ and perform their resummation (yielding the polylogarithm, in this case of the resolved conifold) can we recover regularity at $Q=0$. Looking directly at the resurgent GW expansion \eqref{eq:GW_coni_large_order}, one also sees how the right-hand side has a non-regular $Q$-dependence through $t=-\log Q$. Although the possibility remains that there might be a better variable than $Q$ or $t$ to establish the match against nonperturbative GW invariants, it may also be the case that there is no such variable and reading nonperturbative GW invariants (naturally formulated using a $Q$-expansion) from resurgence expressions (naturally written using the $t$ variable) is in fact a  nontrivial problem which might require some \textit{a priori} enumerative interpretation to know what to look for. Perhaps the fact that the polynomials \eqref{polyname0} through \eqref{polyname3}, appearing in the resurgence relation \eqref{eq:GW_coni_large_order}, have rational coefficients much like the GW invariants themselves, is a clue in that direction.

%%%%%%%%%%%%%%%%%%%%%%%%%%%%%%%%%%%%%%%%%%%%%%%%%%%%%%%%%%%%%%%%%
%%%%%%%%%%%%%%%%%%%%%%%%%%%%%%%%%%%%%%%%%%%%%%%%%%%%%%%%%%%%%%%%%
\section{Computational Explorations in Calabi--Yau Threefolds}\label{sec:examples}
%%%%%%%%%%%%%%%%%%%%%%%%%%%%%%%%%%%%%%%%%%%%%%%%%%%%%%%%%%%%%%%%%
%%%%%%%%%%%%%%%%%%%%%%%%%%%%%%%%%%%%%%%%%%%%%%%%%%%%%%%%%%%%%%%%%

We shall now move on towards non-trivial geometries, for which there are no closed-form expressions for enumerative GW invariants. We shall instead resort to computational methods in order to explore their asymptotics and resurgent structures.

%%%%%%%%%%%%%%%%%%%%%%%%%%%%%%%%%%%%%%%%%%%%%%%%%%%%%%%%%%%%%%%%%
\subsection{The Example of Local $\BP^2$}\label{sec:P2}
%%%%%%%%%%%%%%%%%%%%%%%%%%%%%%%%%%%%%%%%%%%%%%%%%%%%%%%%%%%%%%%%%

Our first non-trivial example will be a local-surface toric-variety. We start with the non-compact CY threefold known as local $\BP^2$, which is the total space of the line bundle $\CO (-3) \to \BP^2$. This example of local $\BP^2$ has a single complex modulus $z$, and a mirror map of the schematic form $Q = \rme^{-t} = \CO (z)$, which eventually allows for a calculation of GW invariants \cite{ckyz99, kz99} (the resulting K\"ahler modulus being the size of the $\BP^2$). In fact, the large-order data for the resurgence analysis first arises in the B-model and will thus require translation into A-model expressions. Specifically, the high genus GW invariants for local $\BP^2$ will come out of B-model calculations, both perturbative \cite{hkr08} and nonperturbative \cite{cesv13, cesv14}, followed by mirror symmetry \cite{hv00}.

%%%%%%%%%%%%%%%%%%%%%%%%%%%%%%%%%%%%%%%%%%%%%%%%%%%%%%%%%%%%%%%%%
\subsubsection*{Free Energies and Gromov--Witten Invariants}
%%%%%%%%%%%%%%%%%%%%%%%%%%%%%%%%%%%%%%%%%%%%%%%%%%%%%%%%%%%%%%%%%

The perturbative free energies are best computed within the B-model using the holomorphic anomaly equations, which are recurrence relations in the genus \cite{bcov93, bcov93b}. The GW invariants are then extracted using the mirror map back to the A-model, and removing the anti-holomorphic dependence (in $\bar{z}$) which is introduced by this computation. We shall not get into any details, which may be found in \cite{hkr08}, but it is perhaps worth mentioning the Picard--Fuchs equation. Its solutions, the periods, are the source to the genus-zero free energy, the mirror map, and also the instanton actions \cite{dmp11}. For local $\BP^2$, the Picard--Fuchs equation is
\begin{equation}
\left\{ \left( z \p_z \right)^3 + 3z^2\, \p_z \left( 3 z \p_z + 1 \right) \left( 3 z \p_z  + 2 \right) \right\} f(z) = 0.
\end{equation}
\noindent
One of its three independent solutions is a constant. Another one, having a $\log z$ singularity, can be identified as the mirror map,
\begin{equation}
\log Q = - t = \log z - 6z + 45z^2 - 560z^3 + \cdots.
\end{equation}
\noindent
The last solution, having a $\log^2 z$ singularity, can be associated to $\p_t F^{(0)}_0$. Upon integration of this last solution, and use of the mirror map, one finds the genus-zero free energy as
\begin{equation}
F^{(0)}_0 = c_3 t^3 + c_2 t^2 + c_1 t + 3 Q - \frac{45}{8} Q^2 + \frac{244}{9} Q^3 + \cdots.
\end{equation}
\noindent
One can ignore the coefficients $c_i$ and then read the GW invariants, $N_{0,d}$, from this $Q$-series.

Within the B-model, the higher-genus free energies\footnote{Note that the genus-one free energy is calculated separately (see \cite{bcov93b} for details), and further has a direct relation to the propagator; namely $\p_z F^{(1)}_0 = \frac{1}{2} C_{zzz} S^{zz}$, where $C_{zzz} =  \left( -3z^3 \left(1+27z\right) \right)^{-1}$ is the Yukawa coupling computed out of the Picard--Fuchs equation.}, $F^{(0)}_g$, may be compactly written as polynomials in $z$ and $S^{zz}(z,\bar{z})$, an auxiliary variable called the propagator \cite{yy04}. To extract higher-genus GW invariants one has to use the holomorphic limit of the propagator $S^{zz}$ (in the large-radius frame),
\begin{equation}
S^{zz}_{\text{hol}, [\text{LR}]} = \frac{1}{2} Q^2 + 15 Q^3 + 135 Q^4 + \cdots.
\label{eq:localP2_SzzholLR_Qseries}
\end{equation}
\noindent
Consider for example $F^{(0)}_2 (z,S^{zz})$, which follows from the holomorphic anomaly equations as
\begin{equation}
F^{(0)}_2 = \left( - \frac{1}{3z^3 \left(1+27z\right)}\right)^2 \left( \frac{5}{24} \left( S^{zz} \right)^3 - \frac{3z^2}{16} \left( S^{zz} \right)^2 + \frac{z^4}{16}\, S^{zz} - \frac{\left(11-162z-729z^2\right)z^6}{1920} \right) -\frac{1}{1920}.
\end{equation}
\noindent
Taking the holomorphic limit and using the mirror map, $z = z(Q)$, one then obtains the A-model result 
\begin{equation}
F^{(0)}_2 = \frac{1}{80} Q + \frac{3}{20} Q^3 - \frac{514}{5} Q^4 + \cdots.
\end{equation}
\noindent
Here, the coefficients of the $Q$-expansion are the $N_{2,d}$ GW invariants. In this way, the holomorphic anomaly equations systematically compute $F^{(0)}_g$, out of $F^{(0)}_h$ with $h=1,\ldots,g-1$, and from them one extracts the $N_{g,d}$ GW invariants as described above. An illustrative (\textit{i.e.}, partial) table of GW invariants for local $\BP^2$ may be found in appendix~\ref{sec:appendix:localP2}. In figure~\ref{fig:GWtotalP2} we schematically represent all the GW invariants we have computed and work with in the present paper.

%%%%%%%%%%%%%%%%%%%%%%%%%%%%%%%%%%%%%%%%%%%%%%%%%%%%%%%%%%%%%%%%%
\begin{figure}[t!]
\begin{center}
\raisebox{0.5cm}{\includegraphics[width=0.5\textwidth]{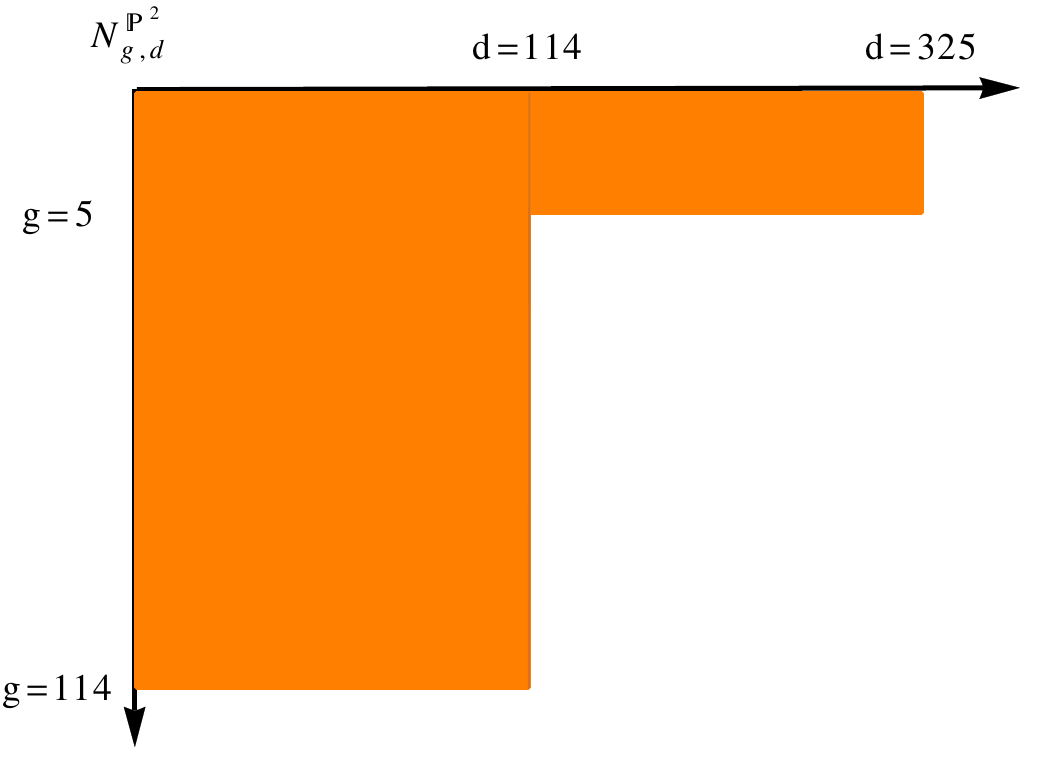}}
\end{center}
\vspace{-1\baselineskip}
\caption{Maximum degree and genus of the GW invariants we computed for local~$\BP^2$.}
\label{fig:GWtotalP2}
\end{figure}
%%%%%%%%%%%%%%%%%%%%%%%%%%%%%%%%%%%%%%%%%%%%%%%%%%%%%%%%%%%%%%%%%

As studied in great detail in \cite{cesv13, cesv14} the free energies for local $\BP^2$, $F^{(0)}_g$, grow factorially fast and render the free-energy expansion asymptotic. The resurgent structure which was uncovered in those references may be summarized as follows.  There are several instanton actions, labelled by $A_1$, $A_2$, $A_3$ and $A_{\text{K}}$, which give rise to corresponding nonperturbative sectors within the total free-energy transseries. Out of these, two actions are leading at large-order, these are $A_1$ and $A_{\text{K}}$, meaning that for some values of the complex-structure modulus $z$ they are the actions controlling the leading growth of the $F^{(0)}_g$. Around the large-radius point in moduli space, $z=0$, it is $A_{\text{K}} = 2\pi t(z)$ which is leading, and elsewhere it is $A_1 = \frac{2\pi\rmi}{\sqrt{3}}\, T_{\text{c}} (z)$, where $T_{\text{c}} (z) = 12 \sqrt{3} \pi^2 \rmi\, \p_t F^{(0)}_0$ is the flat coordinate\footnote{References \cite{cesv13, cesv14} used the notation $t_{\text{c}}$ for this flat coordinate. Herein we use $T_{\text{c}}$ instead so as not to clash with our conifold critical point.} around the conifold point $z= -1/27$. For obvious reasons, we name $A_{\text{K}}$ as the K\"ahler action and $A_1$ as the conifold action (in fact also $A_2$ and $A_3$ are related to the conifold point, but they will not play any role in the present paper). 

In this case, the large-order growth of the free energies may be either
\begin{equation}
F^{(0)}_g \sim \frac{\Gamma \left(2g-1\right)}{A_1^{2g-1}}\, F^{(1)[\text{c}]}_0 \qquad \text{or} \qquad F^{(0)}_g \sim \frac{\Gamma \left(2g-1\right)}{A_{\text{K}}^{2g-1}}\, F^{(1)[\text{K}]}_0,
\end{equation}
\noindent
depending on the value of $Q$. The one-loop one-instanton coefficients are computed from an extension of the holomorphic anomaly equations, alongside the above resurgent relations  (which were needed in order to fix the holomorphic anomaly). They are part of the B-model transseries, and one finds \cite{cesv14}
\begin{equation}
F^{(1)[\text{c}]}_0 = \frac{A_1}{2\pi}\, \rme^{\frac{1}{2} \left(\p_z A_1\right)^2 \left( S^{zz}_{\text{hol},[\text{LR}]} - S^{zz}_{\text{hol},[\text{c}]} \right)} \qquad \text{and} \qquad F^{(1)[\text{K}]}_0 = \frac{3 A_{\text{K}}}{2\pi^2}.
\label{eq:localP2_F10s}
\end{equation}
\noindent
The left expression involves $S^{zz}_{\text{hol},[\text{LR}]}$, whose $Q$-expansion was written in \eqref{eq:localP2_SzzholLR_Qseries}, but it also involves the holomorphic limit of the propagator in the conifold frame, $S^{zz}_{\text{hol},[\text{c}]}$ (see \cite{cesv14}). It is interesting to note how the expression on the right of \eqref{eq:localP2_F10s} is actually equivalent, up to a factor of $3=n_0^{(1)}$, to the one we computed earlier for the resolved conifold.

Being part of the B-model transseries, one may feel tempted to use the mirror map and write these nonperturbative expressions in the A-model, hoping for regular generating functions of our would-be nonperturbative invariants. Unfortunately, as already explained earlier, their $Q$-expansions are explicitly non-regular
\bea
F^{(1)[\text{c}]}_0 &=& \frac{\rmi \left(-Q\right)^{\frac{3}{2}}}{4\pi} \left( \left( \log(Q) - \rmi\pi \right)^2-\pi^2 - 18 Q + \frac{135}{2} Q^2 + \cdots \right) \left( 1 - \frac{27}{2} Q + \frac{1539}{8} Q^2 + \cdots \right), \\
F^{(1)[\text{K}]}_0 &=& - \frac{3}{\pi} \log Q.
\eea
\noindent
This implies that the A-model transseries, obtained via mirror map from the B-model transseries, is not a regular generating function and one has to dig deeper in order to try to understand the counting problem associated to the conjectured quantities ``$N^{(n)}_{g,d}$''.

%%%%%%%%%%%%%%%%%%%%%%%%%%%%%%%%%%%%%%%%%%%%%%%%%%%%%%%%%%%%%%%%%
\subsubsection*{Analysis of Large-Degree Growth}
%%%%%%%%%%%%%%%%%%%%%%%%%%%%%%%%%%%%%%%%%%%%%%%%%%%%%%%%%%%%%%%%%

At fixed genus, the GW invariants grow exponentially in the degree as
\begin{equation}
N_{\boldsymbol{g},d} \sim c_{\boldsymbol{g}}\, d^{2\boldsymbol{g}-3}\, \rme^{d t_{\text{c}}} \left( \log d \right)^\delta,
\label{eq:extraLOGS}
\end{equation}
where $t_{\text{c}} := t(z=-1/27) = 2.90759\ldots - \rmi\pi$ \cite{agm93, kz99}. Figure~\ref{fig:localP2_large_degree_tc} shows a numerical verification of this value for $t_{\text{c}}$. The exponent $2\boldsymbol{g}-3$ of the degree $d$ may be verified numerically from the following large-$d$ sequence
\begin{equation}
f_d - 2 f_{d^2} \sim 2\boldsymbol{g}-3, \qquad \text{where} \qquad f_d := d \left( \rme^{t_{\text{c}}}\, \frac{N_{\boldsymbol{g},d+1}}{N_{\boldsymbol{g},d}} - 1 \right).
\label{eq:localP2_fcombination}
\end{equation}
\noindent
Due to the presence of the $d^2$ factor, and the limit upon available data, the results are not as good as those for $t_{\text{c}}$. Nonetheless, this exponent may also be cleanly verified numerically, as it is shown in figure~\ref{fig:localP2_large_degree_2g-3}. Finally, in similar fashion, we can determine the power of the logarithm $\log d$ from the sequence 
\begin{equation}
2^\delta \sim \frac{\rme^{d \left(d-1\right) t_{\text{c}}}}{d^{2\boldsymbol{g}-3}}\, \frac{N_{\boldsymbol{g},d^2}}{N_{\boldsymbol{g},d}}.
\label{eq:localP2_gammacombination}
\end{equation}
\noindent
This is done in figure~\ref{fig:localP2_large_degree_2g-2}, where it is shown that this exponent may be well fitted to the expected $\delta = 2\boldsymbol{g}-2$. Unfortunately, our available data does not allow us to numerically compute the genus-dependent pre-factor $c_{\boldsymbol{g}}$ with enough accuracy as to present it here. 

%%%%%%%%%%%%%%%%%%%%%%%%%%%%%%%%%%%%%%%%%%%%%%%%%%%%%%%%%%%%%%%%%
\begin{figure}[t!]
\centering
\includegraphics[width=0.6\textwidth]{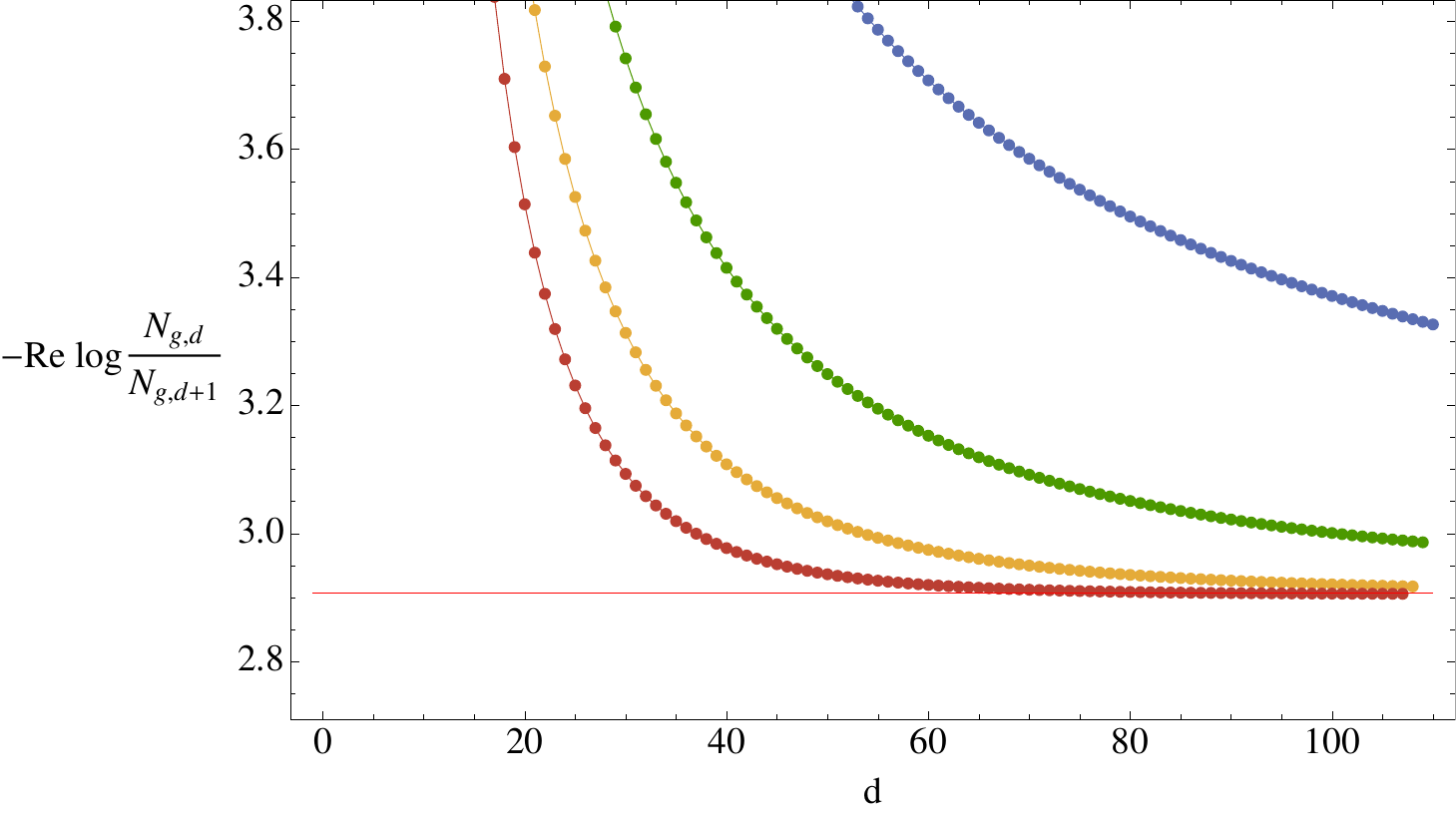}
\caption{Local $\BP^2$: The exponent $t_{\text{c}}$ in the growth of $N_{\boldsymbol{g},d}$ is captured from the ratio of two consecutive GW invariants, when the degree is large. We plot that ratio alongside three Richardson extrapolations, which are clearly converging faster towards the expected result (up to a numerical relative error of about $0.06\%$).}
\label{fig:localP2_large_degree_tc}
\end{figure}
%%%%%%%%%%%%%%%%%%%%%%%%%%%%%%%%%%%%%%%%%%%%%%%%%%%%%%%%%%%%%%%%%

%%%%%%%%%%%%%%%%%%%%%%%%%%%%%%%%%%%%%%%%%%%%%%%%%%%%%%%%%%%%%%%%%
\begin{figure}[t!]
\centering
\includegraphics[width=0.46\textwidth]{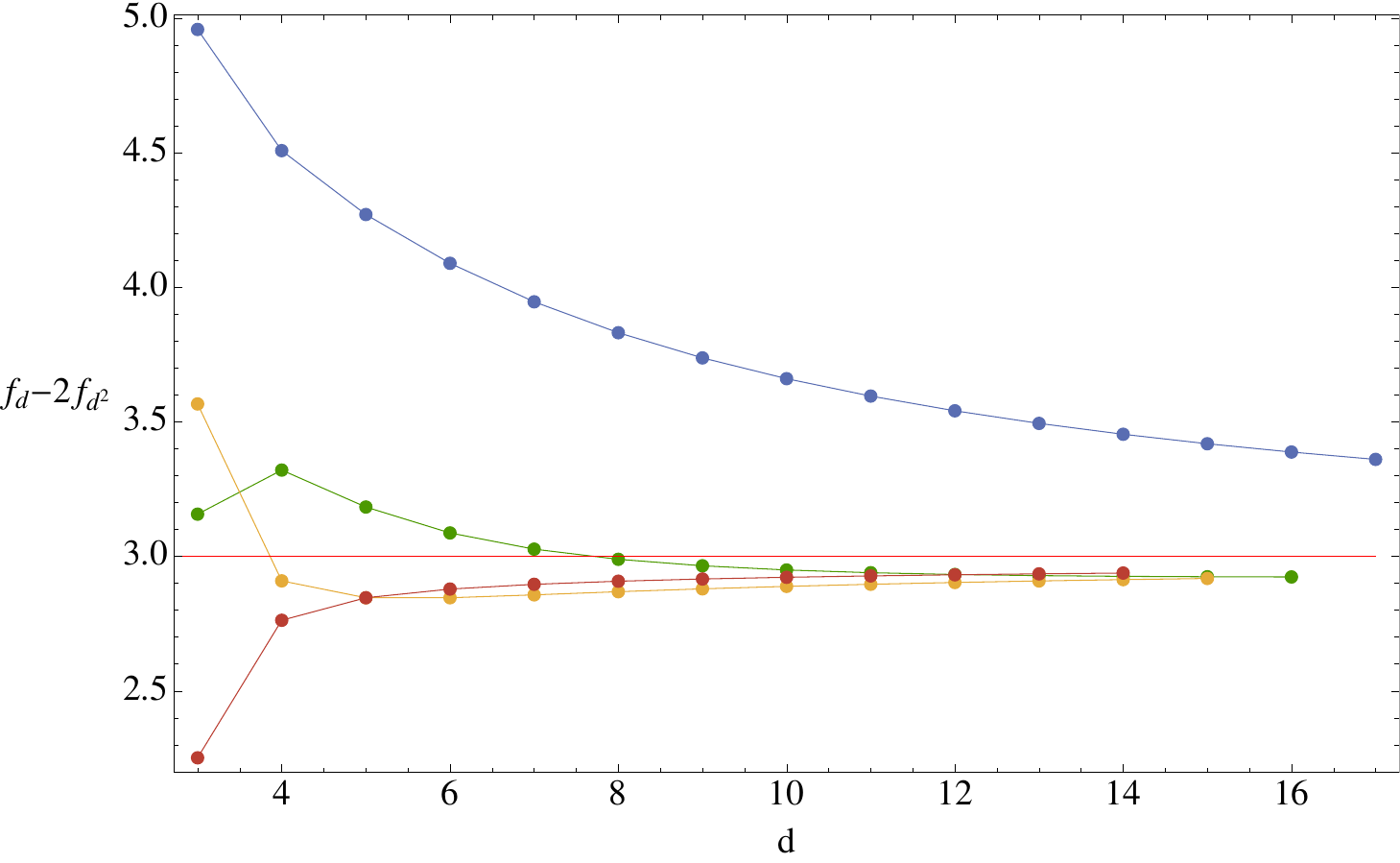}
\hspace{0.05\textwidth}
\includegraphics[width=0.46\textwidth]{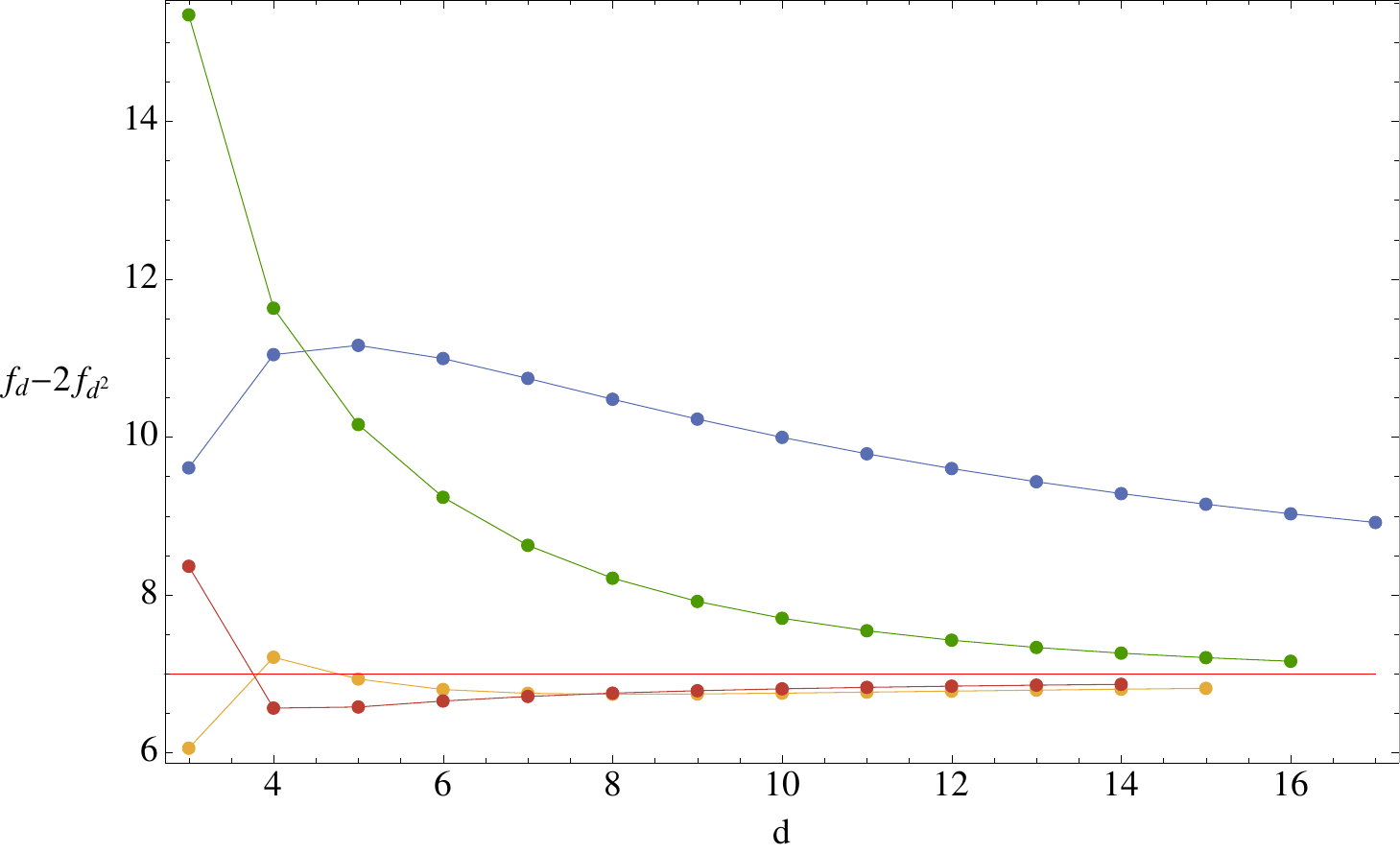}
\caption{Local $\BP^2$: The exponent $2\boldsymbol{g}-3$ is the leading large-order term in $f_d - 2 f_{d^2}$. We have data up to $d=325$ so that the horizontal axis can only reach $d = 17$. The plots illustrate the first few Richardson transforms for $g=3$ (left) and $g=5$ (right), converging faster towards the expected result (up to numerical relative errors of about $2\%$ in both cases). Notice how the presence of logarithms in \eqref{eq:extraLOGS} makes the convergence of Richardson transforms much slower.}
\label{fig:localP2_large_degree_2g-3}
\end{figure}
%%%%%%%%%%%%%%%%%%%%%%%%%%%%%%%%%%%%%%%%%%%%%%%%%%%%%%%%%%%%%%%%%

%%%%%%%%%%%%%%%%%%%%%%%%%%%%%%%%%%%%%%%%%%%%%%%%%%%%%%%%%%%%%%%%%
\begin{figure}[t!]
\centering
\includegraphics[width=0.46\textwidth]{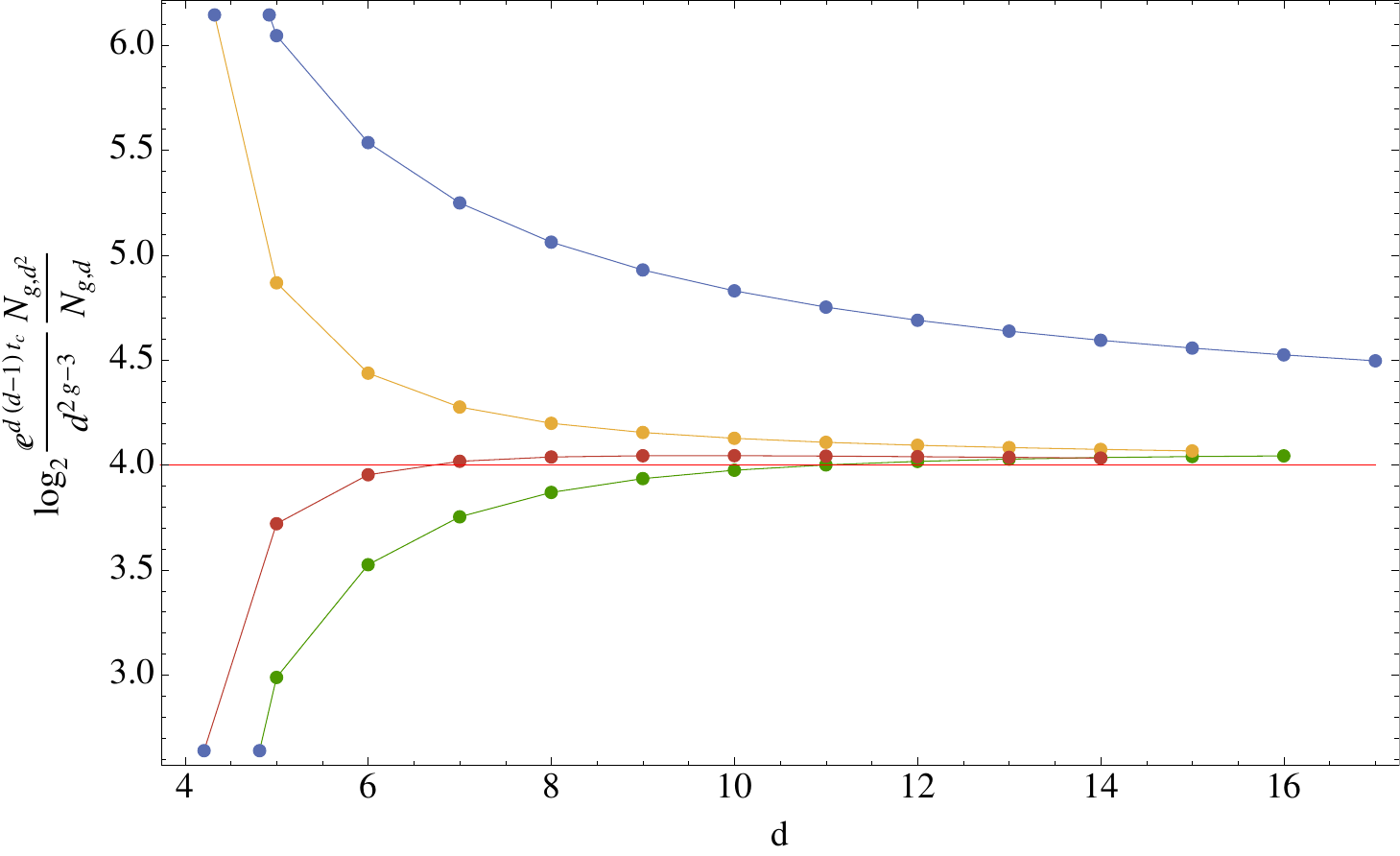}
\hspace{0.05\textwidth}
\includegraphics[width=0.46\textwidth]{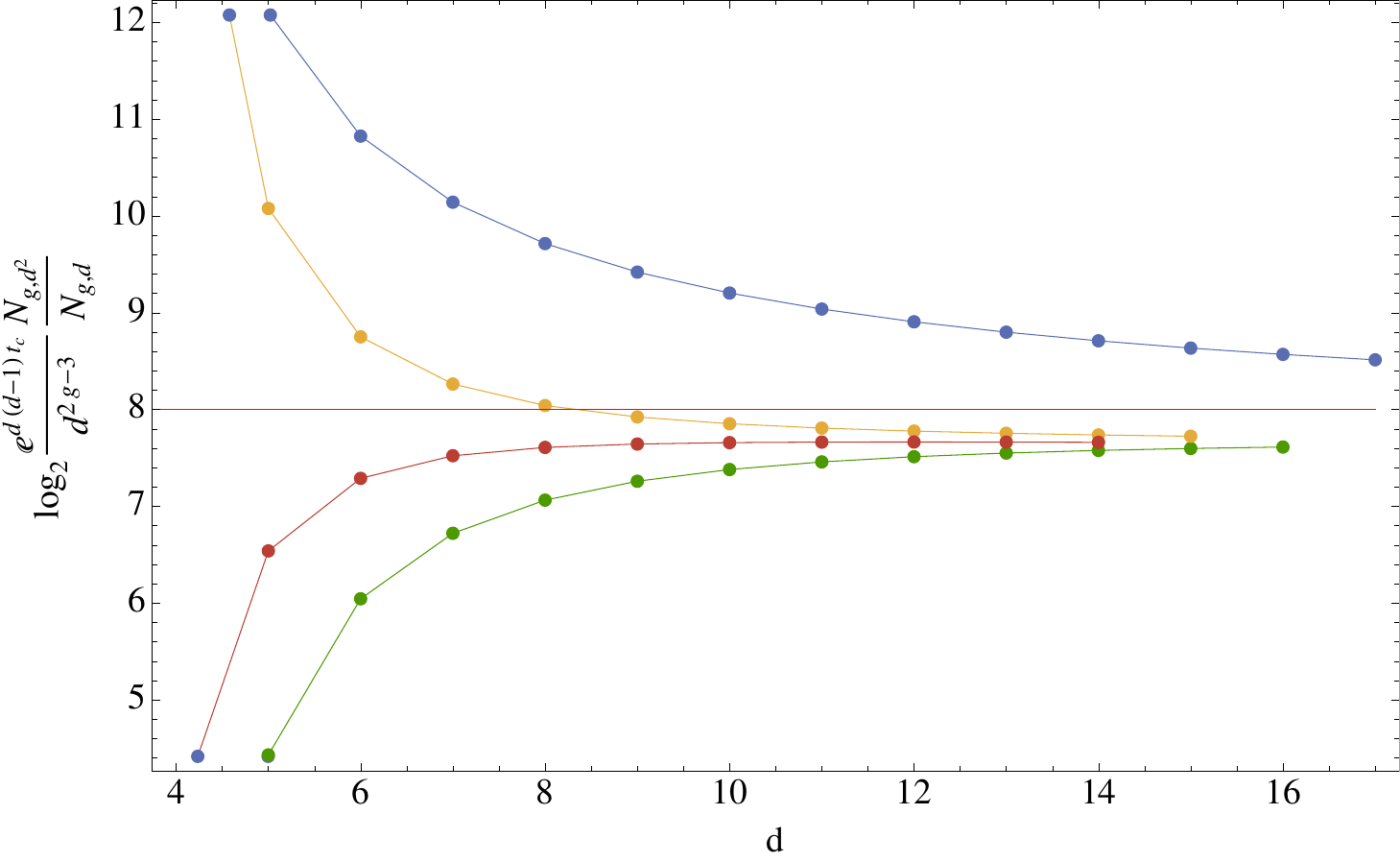}
\caption{Local $\BP^2$: The exponent $\delta$ of the logarithm $\log d$ is the leading term in the sequence \eqref{eq:localP2_gammacombination}. Having data up to $d=325$ implies the horizontal axis only reaches $d = 17$. We plot the first few Richardson transforms for $g=3$ (left) and $g=5$ (right), converging faster towards the expected result (up to small numerical relative errors of about $3\%$ and $4\%$, respectively).}
\label{fig:localP2_large_degree_2g-2}
\end{figure}
%%%%%%%%%%%%%%%%%%%%%%%%%%%%%%%%%%%%%%%%%%%%%%%%%%%%%%%%%%%%%%%%%

%%%%%%%%%%%%%%%%%%%%%%%%%%%%%%%%%%%%%%%%%%%%%%%%%%%%%%%%%%%%%%%%%
\subsubsection*{Analysis of Large-Genus Growth}
%%%%%%%%%%%%%%%%%%%%%%%%%%%%%%%%%%%%%%%%%%%%%%%%%%%%%%%%%%%%%%%%%

As explained in section~\ref{sec:stage}, the large-genus expansion is best expressed in terms of the coefficients $a_d$, $b_{d,n}$, and $c_d$, as in \eqref{eq:GW_abc}, which we reproduce in here
\begin{equation}
N^\tPTwo_{g,\boldsymbol{d}} = f^\tconi_g \left\{ \sum_{n|\boldsymbol{d}} a^\tPTwo_n \left( \frac{\boldsymbol{d}}{n} \right)^{2g-3} + \frac{2g}{B_{2g}}\, \frac{1}{\boldsymbol{d}} \left( c^\tPTwo_{\boldsymbol{d}}\, \delta_{g,1} + \sum_{n=1}^{G_\tPTwo(\boldsymbol{d})-1} b^\tPTwo_{\boldsymbol{d},n}\, n^{2g-2} \right) \right\},
\end{equation}
\noindent
and where, for this example, one explicitly has $G_\tPTwo(d)= (d-1)(d-2)/2$. A table with these first few coefficients is shown in appendix \ref{sec:appendix:localP2}. Recall that these are just convenient integer numbers which essentially capture the very same information as either GW or GV invariants. 

Some of these coefficients, $b_{d,n}$ with $n$ close to $G_\tPTwo(d)-1$, can be identified in closed form as
\begin{equation}
b^\tPTwo_{d,G_\tPTwo(d)-1-k} = p_{-3}(k)\, (-1)^d\, d \left( \left(d+1\right)\left(d+2\right)-2k \right), \qquad 0 \leq k \leq d-2,
\end{equation}
\noindent
where the $p_{-3}(k)$ are given by the generating function
\begin{equation}
\sum_{k=0}^{+\infty} p_{-3}(k)\, q^k = \prod_{m=1}^{+\infty} \frac{1}{\left(1-q^m\right)^3}.
\end{equation}
\noindent
For larger values of $k$ one can try to extend the above formula, at the cost of identifying similar coefficients to $p_{-3}(k)$. A conjectural partial formula for $b^\tPTwo_{d,n}$ is
\begin{equation}
b^\tPTwo_{d,G_\tPTwo(d)-1-k} \stackrel{?}{=} (-1)^d\, d\, \sum_{s=0}^{+\infty} \left( \alpha_{s,k-m_s d+n_s} \left(d+1-s\right)\left(d+2-s\right) - \beta_{s,k-m_s d+n_s} \right),
\end{equation}
\noindent
where
\begin{align}
m_0 = 0, \quad n_0 = 0, & \qquad \alpha_{0,n} = p_{-3}(n), \qquad \beta_{0,n} = 2n p_{-3}(n), \\
m_1 = 1, \quad n_1 = 1, & \qquad \sum_{n=0}^{+\infty} \alpha_{1,n}\, q^n = -3\, \frac{1+q+q^2}{1-q}\, \prod_{m=1}^{+\infty} \frac{1}{\left(1-q^m\right)^3}, \nonumber \\
& 
\hspace{-90pt}
-\beta_{1,n} = 0, 18, 144, 684, 2484, 7578, 20628, 51390, 119736, 263970, 556308, 1127880, 2212704, ?, \\
m_2 = 2, \quad n_2 = 4, & \qquad -\alpha_{2,n}= 6, 24, 72, 162, 315, ?, \qquad -\beta_{2,n} = 0, 36, 252, 1008, 3042, ?.
\end{align}
\noindent
This is as far as we were able to reach with the data we have available for local $\BP^2$.

%%%%%%%%%%%%%%%%%%%%%%%%%%%%%%%%%%%%%%%%%%%%%%%%%%%%%%%%%%%%%%%%%
\subsubsection*{Combined/Diagonal Large-Growth in Genus and Degree}
%%%%%%%%%%%%%%%%%%%%%%%%%%%%%%%%%%%%%%%%%%%%%%%%%%%%%%%%%%%%%%%%%

As discussed earlier in section~\ref{sec:stage}, and illustrated in figure~\ref{fig:leading_degrees}, local $\BP^2$ has two (different) growths of combined genus and degree. They are associated to the K\"ahler ($A_{\text{K}} = 2\pi t$) and conifold ($A_{\text{c}}$) instanton actions, and they are, respectively,
\begin{equation}
d = \frac{2g-3}{t} \qquad \text{and} \qquad d = a_0(Q) + a_1(Q)\, g.
\end{equation}
\noindent
Note that, as one varies $t$, the large-order growth of the free energies will be dominated by either $A_{\text{K}}$ or $A_{\text{c}}$, or a competition between both (see the analysis in \cite{cesv14}). The situation is slightly different with the GW large-order. Here, along \textit{any} diagonal one will find a factorial growth. However, from a resurgence standpoint, perhaps the most interesting diagonals are the ones which connect back to the resurgent structure of the free energies \cite{cesv13, cesv14}. For any chosen diagonal, this connection will exist every time there is a value of $t$ which realizes that chosen diagonal as one of the above (specific) slices. If such a value of $t$ exists, then the large-order growth of the enumerative invariants will be dominated by either $A_{\text{K}}$ or $A_{\text{c}}$ and the connection to the free energies is rather clean. If not, one will instead be upon a ``mixed'' diagonal with both $A_{\text{K}}$ and $A_{\text{c}}$ vying for dominance. Below we shall focus only upon the leading diagonals.

The first leading degree above was explored and justified analytically for the resolved conifold, and the main features which were found in that example remain in the present one. The second leading degree above depends on $t$ (or $Q$) through two \textit{unknown} functions, $a_0(Q)$ and $a_1(Q)$. At this stage, these functions may only be accessed via \textit{numerical} computations; and given limited data, with some significant limitations. In the following we shall summarize the resulting factorial growth of the GW invariants, along the leading diagonals of their $(g,d)$-table, and the relation of this growth with the resurgent structure of the topological-string free energy.

%%%%%%%%%%%%%%%%%%%%%%%%%%%%%%%%%%%%%%%%%%%%%%%%%%%%%%%%%%%%%%%%%
\subsubsection*{K\"ahler Leading Degree}
%%%%%%%%%%%%%%%%%%%%%%%%%%%%%%%%%%%%%%%%%%%%%%%%%%%%%%%%%%%%%%%%%

In this case, the only difference with respect to the resolved conifold turns out to be a simple multiplying factor, the GV invariant $n_0^{(1)} = 3$ of local $\BP^2$, in which case the analog of \eqref{eq:GW_coni_large_order} is now
\begin{equation}
\left. N^\tPTwo_{g,d}\, Q^d \right|_{g=\frac{t}{2}d+\sh} \sim \sum_{h=0}^{+\infty} \frac{\Gamma \left(2g-\frac{3}{2}-h\right)}{A_{\text{K}}^{2g-\frac{3}{2}-h}}\, \frac{n_0^{(1)}\, t^{\frac{3}{2}-h}}{2^{2h+1}\, \pi^{h+2}}\, \polyname_h(\sh).
\label{eq:GW_localP2_large_order_Kahler}
\end{equation}
\noindent
The polynomials $\polyname_h(\sh)$ are precisely the \textit{same} as in \eqref{eq:CP_h_polynomials}, and the integer $\sh$ is introduced to make both $g$ and $d$ integer; see the discussion around equation \eqref{eq:def_p}. 

Computational tests on the validity of \eqref{eq:GW_localP2_large_order_Kahler} are shown in figures~\ref{fig:ALL_Kahler_Action} and~\ref{fig:ALL_LOOPS}; with figure~\ref{fig:ALL_Kahler_Action} testing the K\"ahler instanton action and figure~\ref{fig:ALL_LOOPS} testing the (universal) validity of the polynomials $\polyname_h(\sh)$ for $h=0$ through $5$. The precise nature of these computational tests is exactly the same as we did earlier for the resolved conifold, and we refer to that discussion for further details.

%%%%%%%%%%%%%%%%%%%%%%%%%%%%%%%%%%%%%%%%%%%%%%%%%%%%%%%%%%%%%%%%%
\begin{figure}[t!]
\centering
\includegraphics[width=0.45\textwidth]{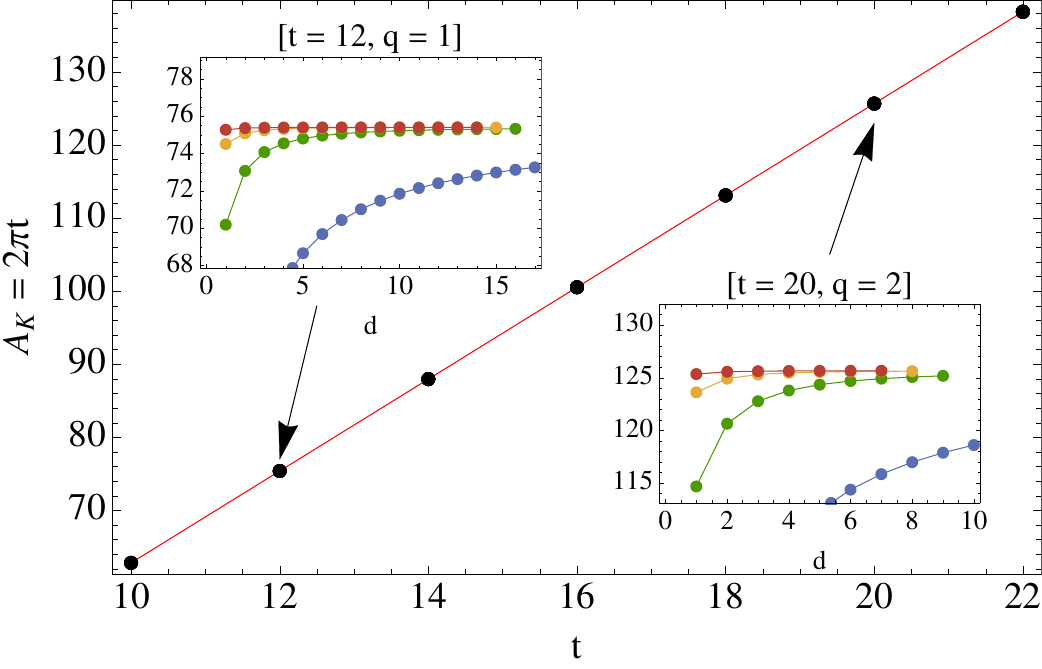}
\hspace{0.05\textwidth}
\includegraphics[width=0.45\textwidth]{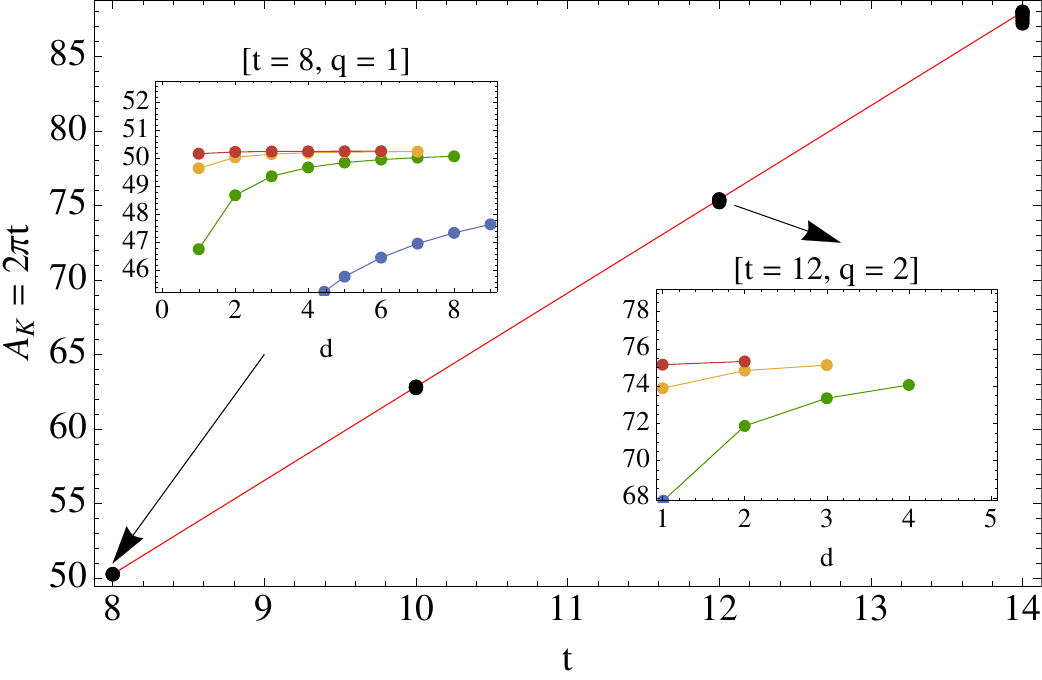}
\includegraphics[width=0.45\textwidth]{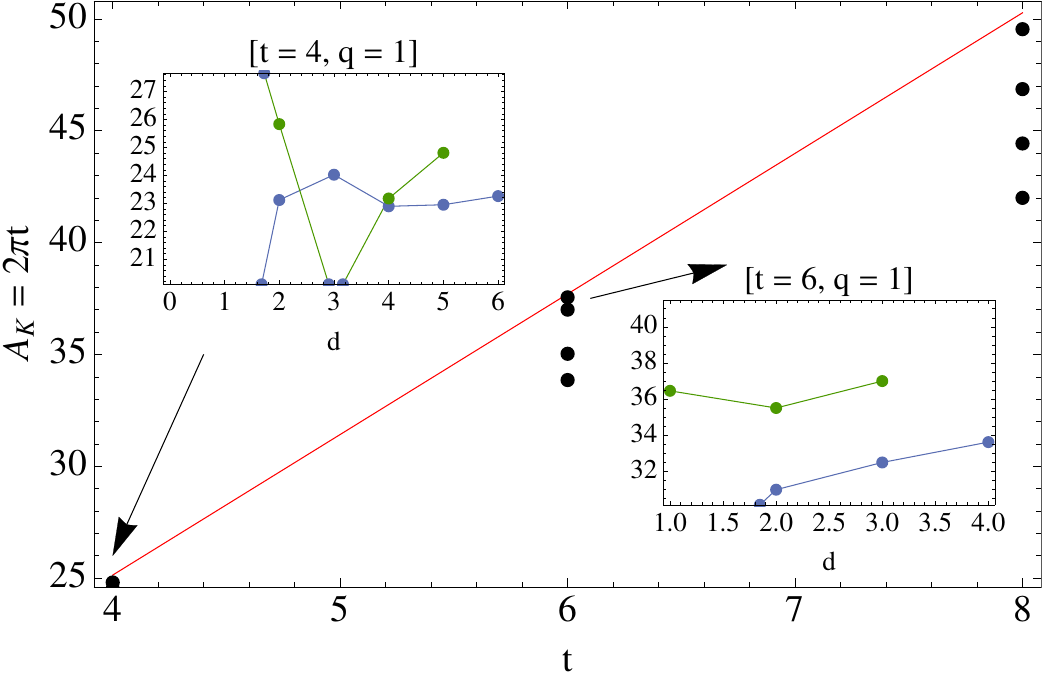}
\hspace{0.05\textwidth}
\includegraphics[width=0.45\textwidth]{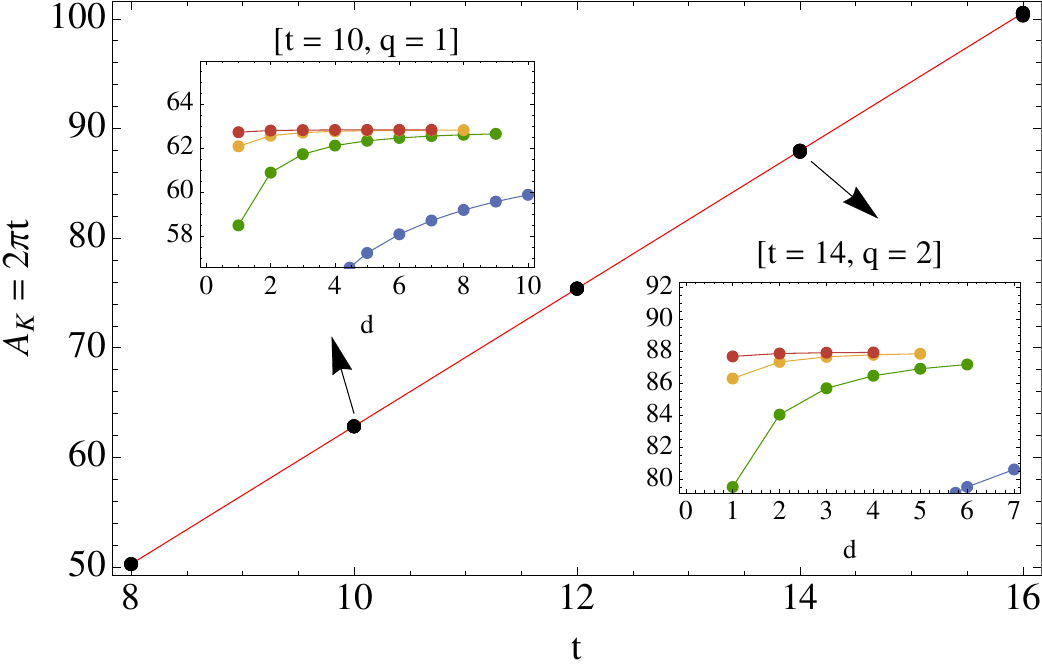}
\caption{Tests of the instanton action $A_{\text{K}} = 2\pi t$, from the K\"ahler saddle-point, for local $\BP^2$, ABJM, the quintic, and the local curve $X_3$ (from left to right and top to bottom). The inclosed plots show the convergence for a couple of values of $t$ and $\sh$. Note that for the quintic we do not have enough data to guarantee reaching a limit where all black dots overlap.}
\label{fig:ALL_Kahler_Action}
\end{figure}
%%%%%%%%%%%%%%%%%%%%%%%%%%%%%%%%%%%%%%%%%%%%%%%%%%%%%%%%%%%%%%%%%

%%%%%%%%%%%%%%%%%%%%%%%%%%%%%%%%%%%%%%%%%%%%%%%%%%%%%%%%%%%%%%%%%
\begin{figure}[t!]
\centering
\includegraphics[width=0.7\textwidth]{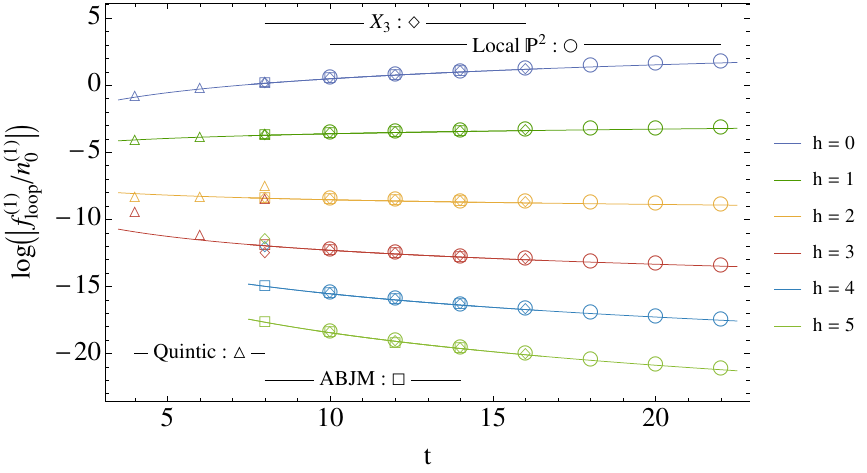}
\caption{Numerical checks of the loop-corrections $f^{(1)}_h = \frac{n^{(1)}_0 t^{\frac{3}{2}-h}}{2^{2h+1} \pi^{h+2}}\, \CP_h(\sh)$, for $h=0,\ldots,5$ and $\sh=1$, for our several examples. We plot the logarithm of the ratio $f^{(1)}_0/n^{(1)}_0$ to have universal quantities which all fit within the same graph.}
\label{fig:ALL_LOOPS}
\end{figure}
%%%%%%%%%%%%%%%%%%%%%%%%%%%%%%%%%%%%%%%%%%%%%%%%%%%%%%%%%%%%%%%%%

This formula \eqref{eq:GW_localP2_large_order_Kahler}, when restricted to the first approximation $h=0$, reproduces the prediction from the saddle-point approximation which was explained around equation \eqref{eq:GW_saddle_point_growth} (and leading up to it). Indeed, for the case of K\"ahler leading degree, $a_0 = -3/t$, $a_1 = 2/t$, $a_2 = t$, and $F^{(1)}_0 = \frac{3 A_{\text{K}}}{2\pi^2}$. Using these values in equation \eqref{eq:GW_saddle_point_growth} (with $\beta = 1$) we precisely reproduce \eqref{eq:GW_localP2_large_order_Kahler} truncated to $h=0$. 
A numerical check of the values of $a_0$ and $a_1$ is shown in the upper plots of figure \ref{fig:localP2_1overa0a1}, for which we can nicely fit
\begin{alignat}{2}
a_0(Q)^{-1} &= \left(-0.65 \pm 0.09\right) + \left(-0.274 \pm 0.008\right) t, \qquad && r^2 = 0.945, \\
a_1(Q)^{-1} &= \left(0.038 \pm 0.005\right) + \left(0.4967 \pm 0.0004\right) t, \qquad && r^2 = 0.99995.
\end{alignat}
\noindent
The shift in $a_0(Q)$ is not very reliable, but the slope in $a_1$ is quite close to the expected value. 

This asymptotics arises from the $a_{d=1}$ contribution in the $abc$-expansion of the GW invariants in \eqref{eq:GW_abc}. Since $a_{d=1} = n_0^{(1)}$ one will always find the resolved-conifold asymptotics multiplied by this factor. 

%%%%%%%%%%%%%%%%%%%%%%%%%%%%%%%%%%%%%%%%%%%%%%%%%%%%%%%%%%%%%%%%%
\begin{figure}[t!]
\centering
\includegraphics[width=0.45\textwidth]{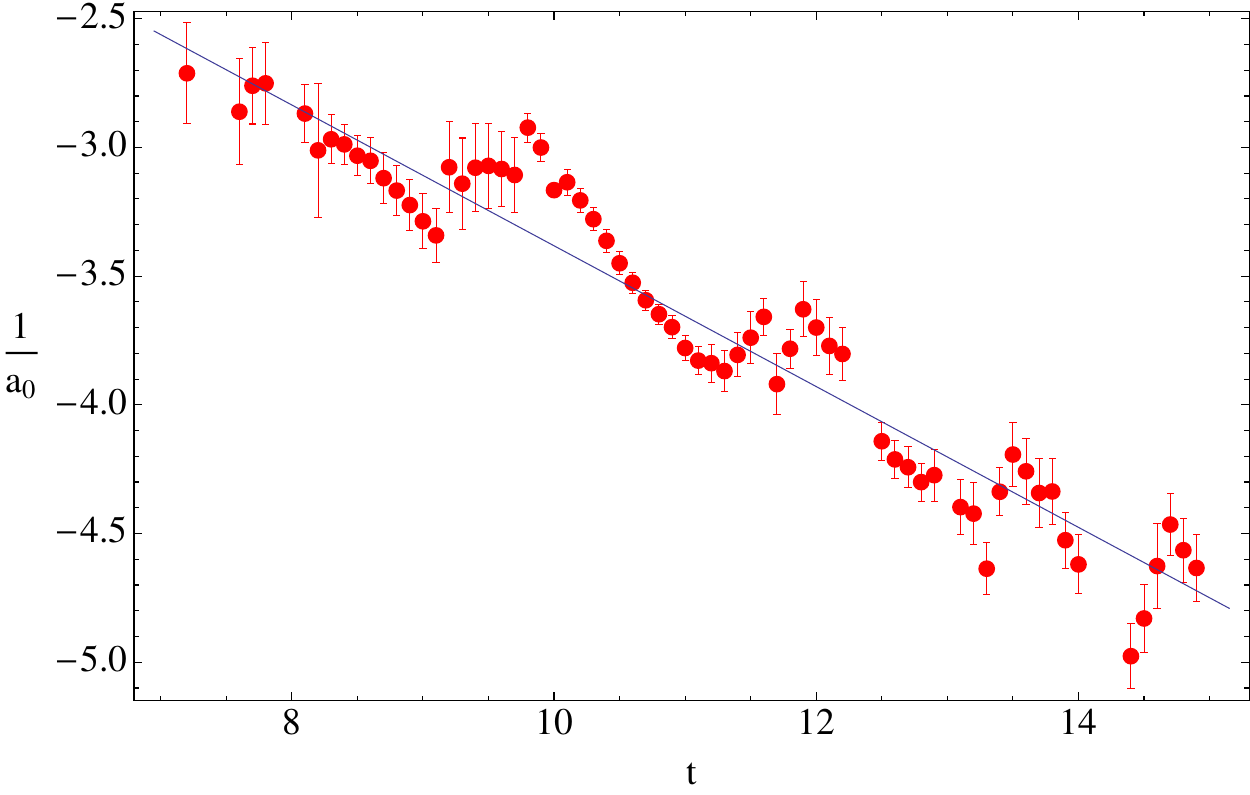}
\hspace{0.05\textwidth}
\includegraphics[width=0.45\textwidth]{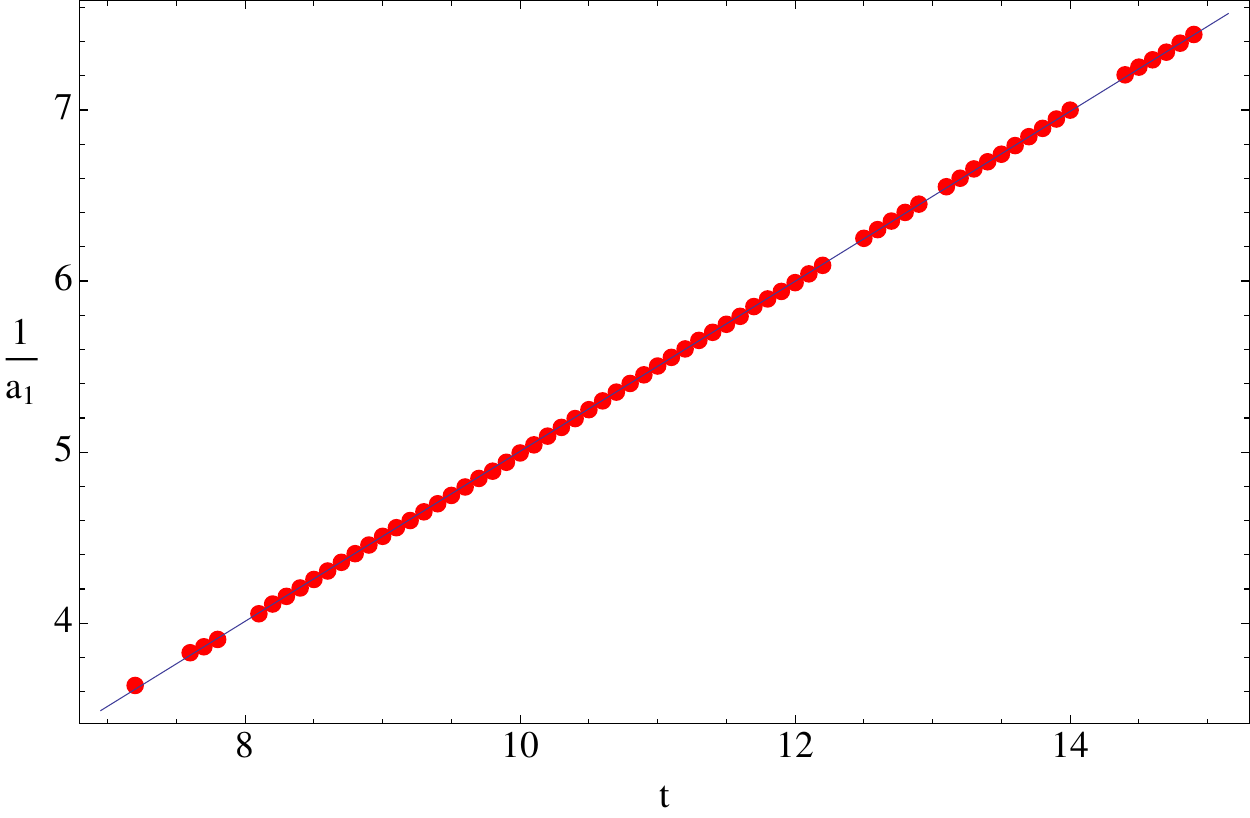}
\includegraphics[width=0.45\textwidth]{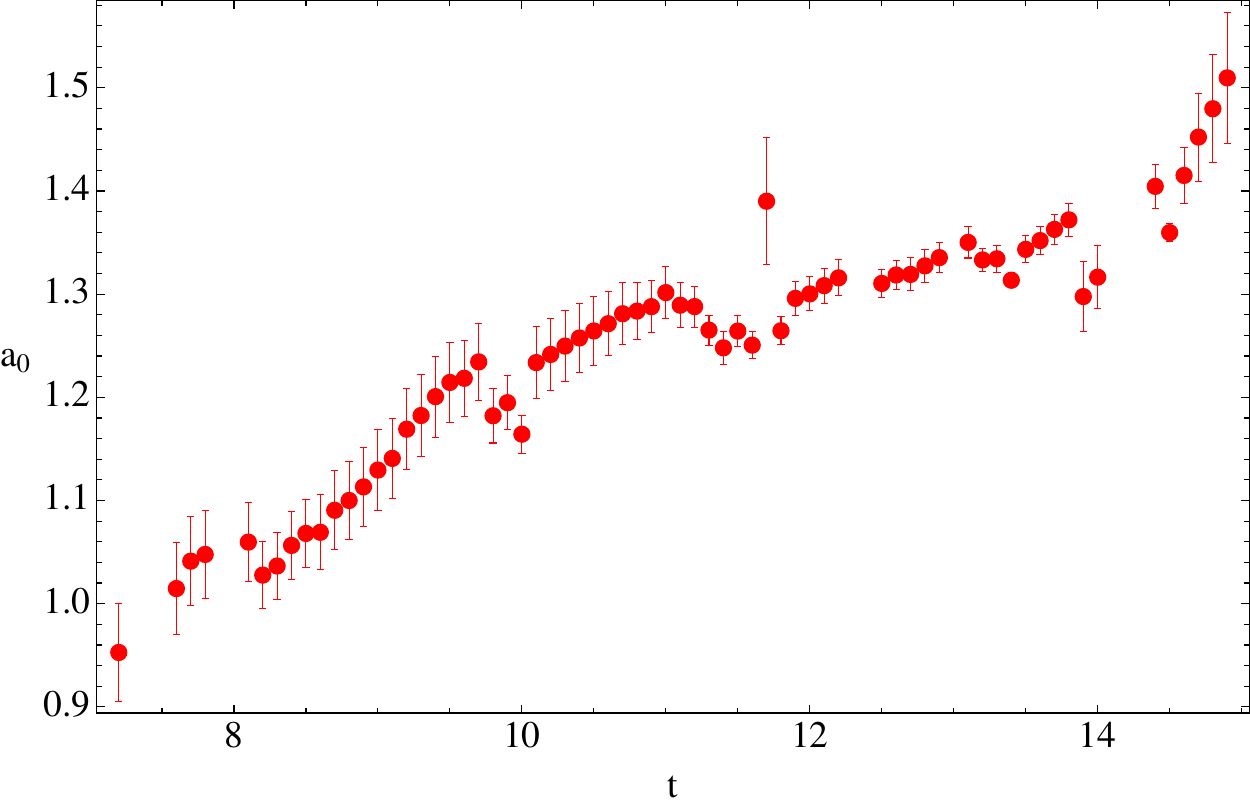}
\hspace{0.05\textwidth}
\includegraphics[width=0.45\textwidth]{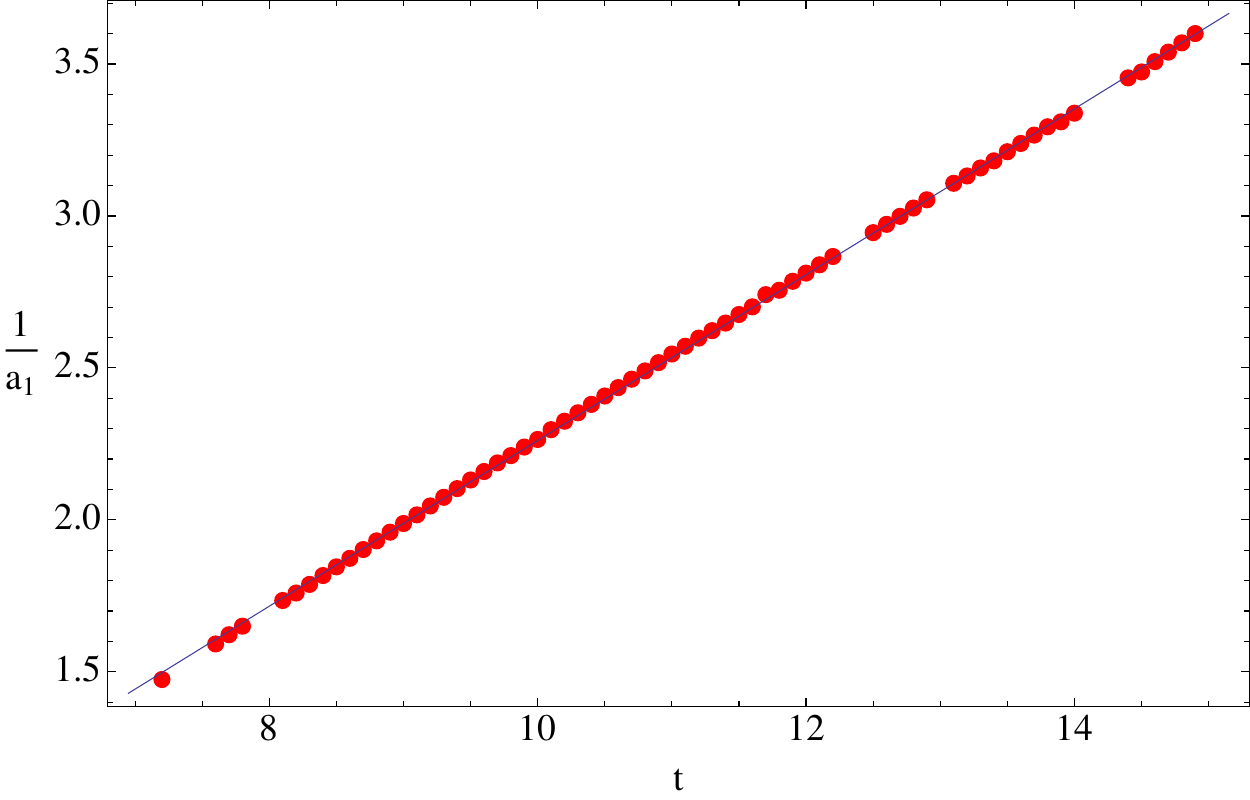}
\caption{Local $\BP^2$: Numerical calculation of $a_0(Q)$ and $a_1(Q)$ associated to the K\"ahler instanton action (the two upper plots) and to the conifold instanton action (the two lower plots). We are showing tests for their inverses whenever the dependence seems linear, although we were not able to confirm this analytically in the case of the conifold action. (We are only plotting reasonably trustworthy fits, thus the apparent holes in the plots.)}
\label{fig:localP2_1overa0a1}
\end{figure}
%%%%%%%%%%%%%%%%%%%%%%%%%%%%%%%%%%%%%%%%%%%%%%%%%%%%%%%%%%%%%%%%%

%%%%%%%%%%%%%%%%%%%%%%%%%%%%%%%%%%%%%%%%%%%%%%%%%%%%%%%%%%%%%%%%%
\subsubsection*{Conifold Leading Degree}
%%%%%%%%%%%%%%%%%%%%%%%%%%%%%%%%%%%%%%%%%%%%%%%%%%%%%%%%%%%%%%%%%

The second dominant degree, $d = a_0(Q) + a_1(Q)\, g$, is harder to analyze as everything must now be approached numerically; from the computation of $a_0(Q)$ and $a_1(Q)$ to the asymptotics.

The numerical fit to $1/a_1(Q)$ is shown in the lower-right plot of figure~\ref{fig:localP2_1overa0a1}. It is obtained from first fitting straight lines $d = \alpha\, g + \beta$ for different (fixed) values of $t$. Then fitting these results against a linear dependence in $t$ we have obtained
\be
a_1(Q)^{-1} = \left(-0.466 \pm 0.005\right) + \left(0.2728 \pm 0.0005\right) t, \qquad r^2 = 0.9998.
\ee
\noindent
On what concerns $a_0(Q)$, its numerical calculation is shown in the lower-left plot of figure~\ref{fig:localP2_1overa0a1}, but there is no obvious fit to do here (numerically, the dependence of $1/a_0(Q)$ does not seem to be linear in $t$, yielding a poor $r^2=0.849$, and this will become even more evident in following examples). At this moment we cannot provide an analytical interpretation for these numbers, or even guarantee that the fit to a straight line is justified since the interval in $t$ we have considered might be too small to be significant. Nonetheless, we do report the results.  

Because we now lack the precision we had along the K\"ahler leading degree, we cannot provide a systematic exploration of the GW asymptotics with conifold leading degree. We can, however, identify particular values of $t$ for which the exploration becomes simpler. One such point is found when $a_1(Q) = 1$, or $t \approx 5.6993\ldots$. In this case we explore the growth of $N_{g,g+\Delta}$ for some integer $\Delta \in \BZ$ (implicitly associated to $a_0(Q)$). The numerical exploration of this diagonal slice in the GW-table yields the result
\begin{equation}
N^\tPTwo_{g,g+\Delta} \sim \frac{\Gamma \left(2g-\frac{3}{2}\right)}{A_0^{2g-\frac{3}{2}}}\, \rme^{\alpha_0 + \alpha_1\, \Delta} \left(-1\right)^{\Delta+1},
  \label{eq:localP2_GW_main_diagonal}
\end{equation}
\noindent
where $A_0 \approx 0.655 995 \ldots$, and $\alpha_0$ and $\alpha_1$ are pure numbers which cannot be computed with much precision. The interesting point is that the value of $A_0$ can be matched to the saddle-point prediction involving the \textit{conifold} action,
\begin{equation}
A_1(Q)\, \sqrt{Q} = 0.655995043\ldots,
\end{equation}
\noindent
where the $\sqrt{Q}$ comes from including the factor $Q^d$ that multiplies $N^\tPTwo_{g,g+\Delta}$ in the saddle-point expression. On the other hand, the one-loop coefficient in \eqref{eq:localP2_GW_main_diagonal} should correspond to
\begin{equation}
\left( \frac{a_2}{\pi a_1 A_1} \right)^{1/2} F^{(1)[\text{c}]}_0,
\end{equation}
\noindent
but $a_2(Q)$ is directly related to the second derivative at the saddle point and thus it cannot be computed from first principles.

%%%%%%%%%%%%%%%%%%%%%%%%%%%%%%%%%%%%%%%%%%%%%%%%%%%%%%%%%%%%%%%%%
\subsection{The Example of Local $\BP^1 \times \BP^1$}\label{sec:P1xP1}
%%%%%%%%%%%%%%%%%%%%%%%%%%%%%%%%%%%%%%%%%%%%%%%%%%%%%%%%%%%%%%%%%

Let us next address another (toric) local surface, the non-compact CY threefold known as local $\BP^1 \times \BP^1$, which is the total space of the line bundle $\CO (-2,-2) \to \BP^1 \times \BP^1$. Generically, local $\BP^1 \times \BP^1$ has two complex structure moduli, $z_1$ and $z_2$, implying that the mirror map is similarly twofold, $Q_1 = \rme^{-t_1} = \CO (z_1)$ and $Q_2 = \rme^{-t_2} = \CO (z_2)$. In order to have reasonable large-order data for the resurgence analysis, in what follows we shall restrict to a slice of this variety where $z_1=z_2$.

%%%%%%%%%%%%%%%%%%%%%%%%%%%%%%%%%%%%%%%%%%%%%%%%%%%%%%%%%%%%%%%%%
\subsubsection*{Free Energies and Gromov--Witten Invariants}
%%%%%%%%%%%%%%%%%%%%%%%%%%%%%%%%%%%%%%%%%%%%%%%%%%%%%%%%%%%%%%%%%

Instead of working in the full two-dimensional moduli space, we shall restrict to the (simpler) one-dimensional \textit{diagonal slice} where the sizes of both $\BP^1$'s in the local $\BP^1\times\BP^1$ geometry are set to be \textit{equal}. One is thus left with a single modulus. The resulting such theory is closely related to a rather well-known gauge theory, called ABJM gauge theory \cite{abjm08} (see also \cite{mp09}), and we shall use this name in the following to denote this diagonal slice of local $\BP^1\times\BP^1$.

Of course the general local $\BP^1\times\BP^1$ geometry has two Picard--Fuchs operators; annihilating periods. One immediate simplification of the diagonal slice is to reduce this number to just one, 
\begin{equation}
\left\{ \left( z \p_z \right)^4 - 4z \left( 4 \left( z \p_z \right)^3 + 4 \left( z \p_z \right)^2 + z \p_z \right) \right\} f(z) = 0.
\end{equation}
\noindent
From its solutions, we can identify the mirror map and the genus-zero free energy,
\begin{align}
- t &= \log z + 4z + 18z^2 + \frac{400}{3}z^3 + \cdots = \log Q, \\ 
F^{(0)}_0 &= c_3 t^3 + c_2 t^2 + c_1 t - 4 Q - \frac{9}{2} Q^2 - \frac{328}{27} Q^3 + \cdots.
\end{align}
\noindent
After dealing with the genus-one free energy, $F^{(0)}_1$, one can proceed and use the holomorphic anomaly equations to compute higher-genus free energies\footnote{This was done in \cite{hkr08} using the language of propagators, and addressing the full local $\BP^1\times\BP^1$ geometry; and in \cite{dmp10} using modular forms, and while restricting to the ABJM diagonal slice.}, from which the GW invariants are eventually read. For example, in the language of modular forms,
\begin{align}
F^{(0)}_2 &= \frac{5 E_2^3}{5184 c d^2} - \frac{E_2^2}{576 d^2} + \frac{E_2 \left(c^2-c d+d^2\right)}{864 c d^2} + \frac{-16 c^3+15 c^2 d-21 c d^2+2 d^3}{51840 c d^2} = \nonumber \\
&= - \frac{Q}{60} - \frac{Q^2}{20} - \frac{Q^3}{10} + \cdots,
\end{align}
\noindent
where $E_2(\tau)$ is the second Eisenstein series, $c = \vartheta_3^4(\tau)$ and $d= \vartheta_4^4(\tau)$ are powers of theta functions, and the modular parameter $\tau$ is a function of $z$; see \cite{dmp10} for full details. Also, in this language $F^{(0)}_1 = \log \eta(\tau)$ with $\eta(\tau)$ the Dedekind eta-function. In appendix~\ref{sec:appendix:localP1xP1} we list the first few GW invariants, and figure~\ref{fig:GWtotalP1P1} schematically represents the ones we have computed and will work with.

%%%%%%%%%%%%%%%%%%%%%%%%%%%%%%%%%%%%%%%%%%%%%%%%%%%%%%%%%%%%%%%%%
\begin{figure}[t!]
\begin{center}
\raisebox{0.5cm}{\includegraphics[width=0.5\textwidth]{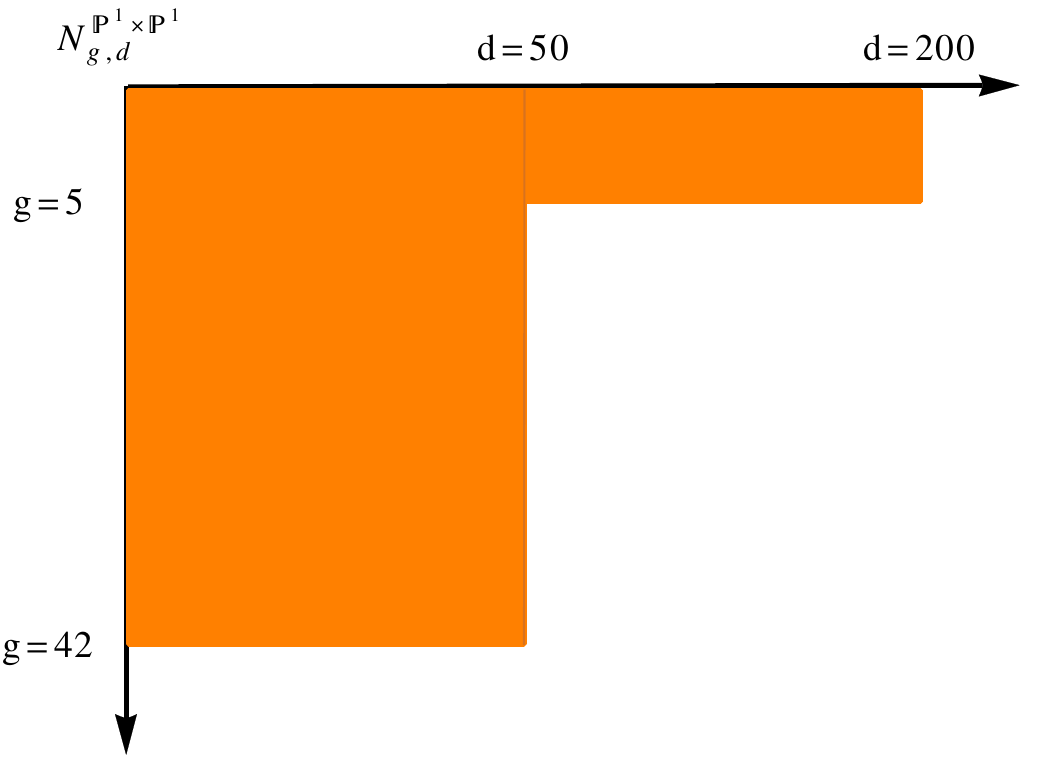}}
\end{center}
\vspace{-1\baselineskip}
\caption{Maximum degree and genus of the GW invariants we computed for ABJM.}
\label{fig:GWtotalP1P1}
\end{figure}
%%%%%%%%%%%%%%%%%%%%%%%%%%%%%%%%%%%%%%%%%%%%%%%%%%%%%%%%%%%%%%%%%

The instanton actions for ABJM were extensively discussed in \cite{dmp11} and are associated to special points in moduli space: large K\"ahler structure ($z=\infty$) yielding $A_{\text{K}} = 2\pi t$, conifold point ($z=z_{\text{c}}$) yielding $A_{\text{c}}$, and orbifold point ($z=0$) yielding $A_{\text{o}}$. These three actions are actually \textit{linearly dependent} with integer coefficients. This implies that the ABJM transseries (whose construction is still an open problem for future research) might either involve only two of these actions (selected upon some criteria of relevance), or it might involve all three of them (in which case one would obtain a resonant transseries as in, \textit{e.g.}, \cite{gikm10, asv11, sv13}).

%%%%%%%%%%%%%%%%%%%%%%%%%%%%%%%%%%%%%%%%%%%%%%%%%%%%%%%%%%%%%%%%%
\subsubsection*{Analysis of Large-Degree Growth}
%%%%%%%%%%%%%%%%%%%%%%%%%%%%%%%%%%%%%%%%%%%%%%%%%%%%%%%%%%%%%%%%%

The large degree, fixed genus, growth is completely analogous to that of local $\BP^2$, 
\be
N_{\boldsymbol{g},d} \sim c_{\boldsymbol{g}}\, d^{2\boldsymbol{g}-3}\, \rme^{d t_{\text{c}}} \left( \log d \right)^\delta.
\ee
\noindent
The value of the K\"ahler parameter at the conifold point is now $t_{\text{c}} := t(z=1/16) = 8K/\pi = 2.33248723\ldots$ \cite{dmp10}, where $K = \sum_{n=0}^{+\infty} (-1)^n \left(2n+1\right)^{-2}$ is the Catalan constant. The exponent $\delta$ numerically matches to the expected $2\boldsymbol{g}-2$. This $g$-dependence of the exponents may be tested using the same large-$d$ sequences as for local $\BP^2$, \textit{i.e.}, the combinations \eqref{eq:localP2_fcombination} and \eqref{eq:localP2_gammacombination}. These numerical results are illustrated in figures~\ref{fig:ABJM_large_degree_tc}, \ref{fig:ABJM_large_degree_2g-3} and \ref{fig:ABJM_large_degree_2g-2}.

%%%%%%%%%%%%%%%%%%%%%%%%%%%%%%%%%%%%%%%%%%%%%%%%%%%%%%%%%%%%%%%%%
\begin{figure}[t!]
\centering
\includegraphics[width=0.6\textwidth]{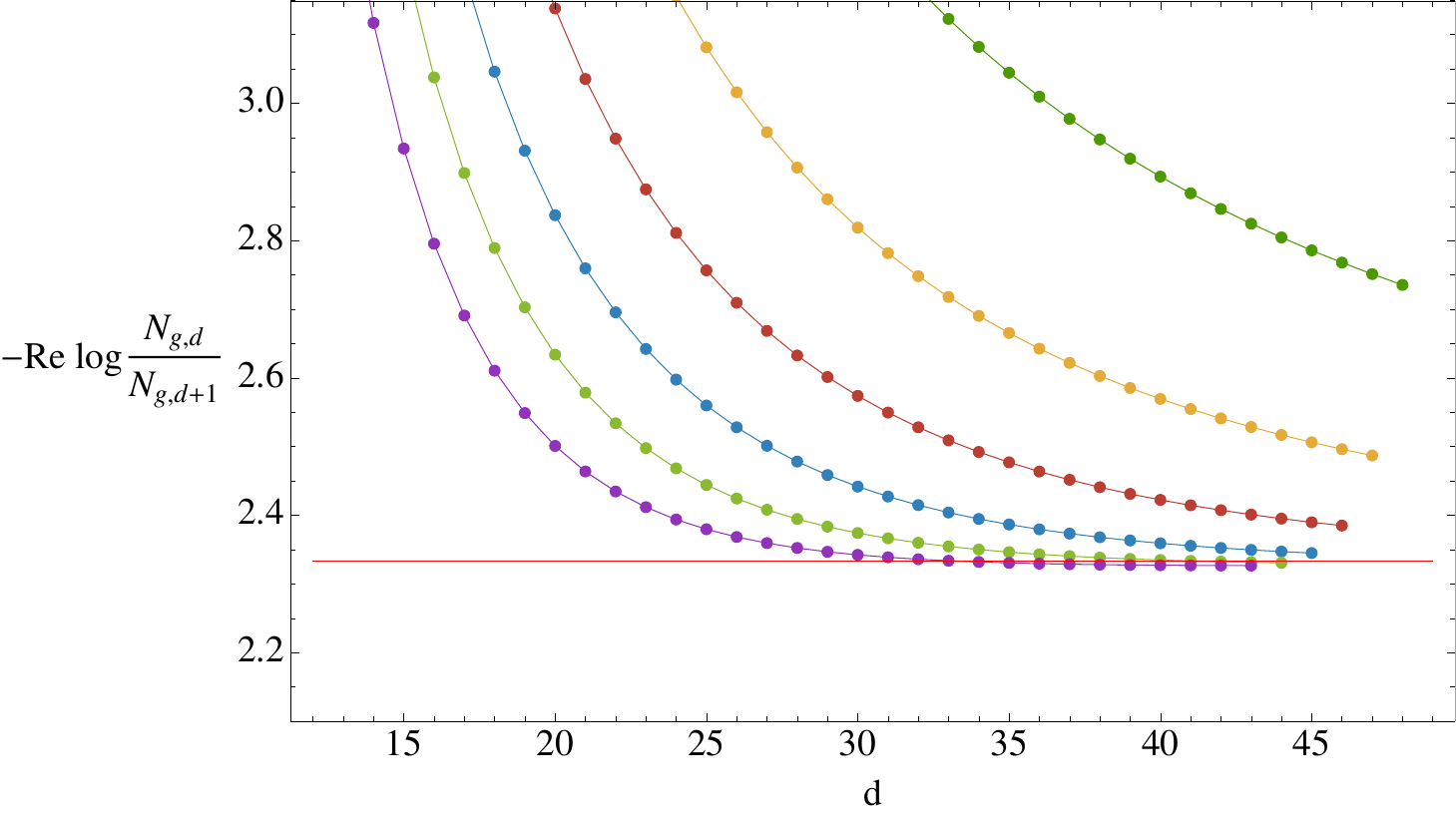}
\caption{ABJM: The exponent $t_{\text{c}}$ in the growth of $N_{\boldsymbol{g},d}$ is captured from the ratio of two consecutive GW invariants, when the degree is large. We plot that ratio alongside six Richardson extrapolations, which are clearly converging faster towards the expected result (up to a numerical relative error of about $0.3\%$).}
\label{fig:ABJM_large_degree_tc}
\end{figure}
%%%%%%%%%%%%%%%%%%%%%%%%%%%%%%%%%%%%%%%%%%%%%%%%%%%%%%%%%%%%%%%%%

%%%%%%%%%%%%%%%%%%%%%%%%%%%%%%%%%%%%%%%%%%%%%%%%%%%%%%%%%%%%%%%%%
\begin{figure}[t!]
\centering
\includegraphics[width=0.46\textwidth]{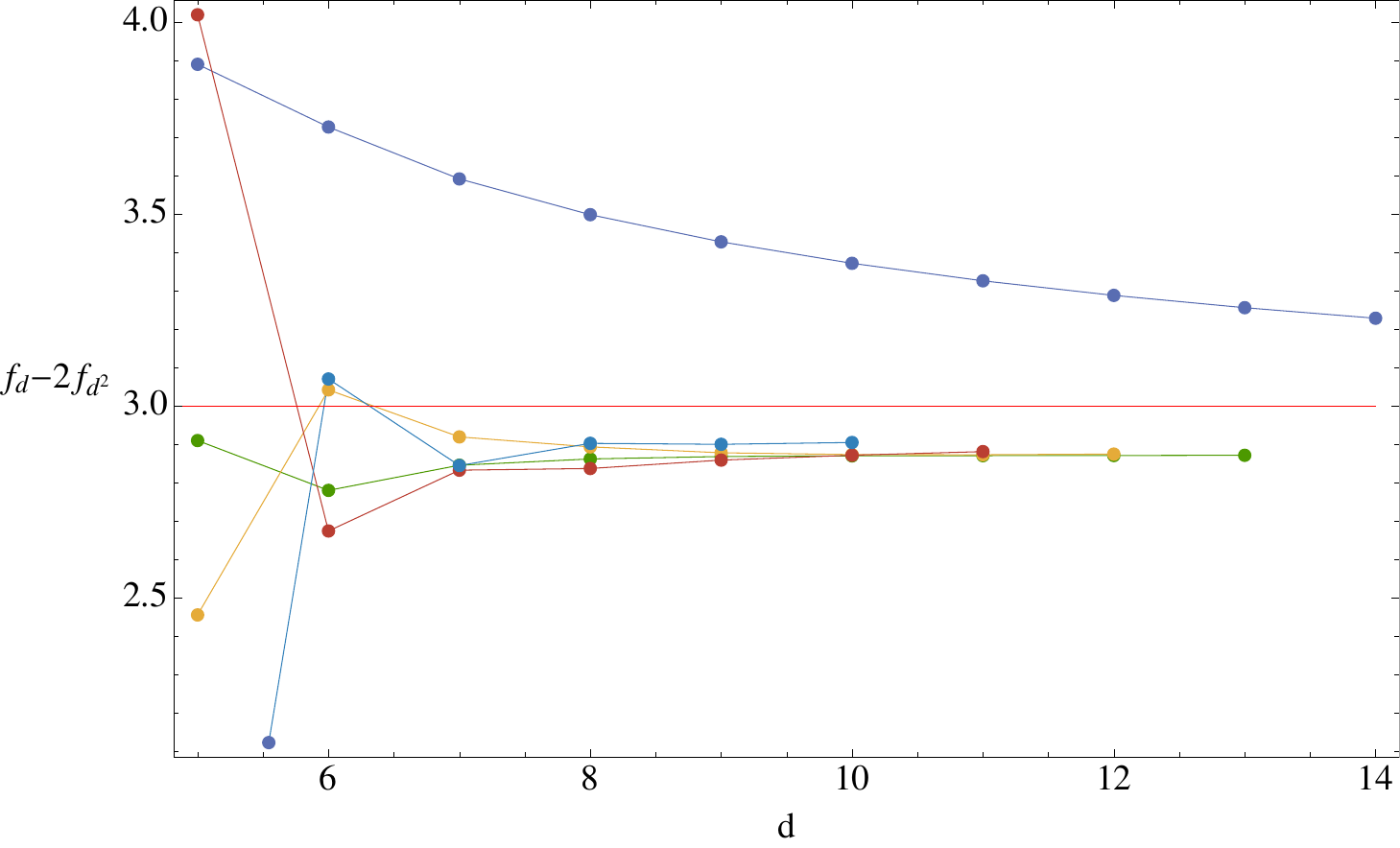}
\hspace{0.05\textwidth}
\includegraphics[width=0.46\textwidth]{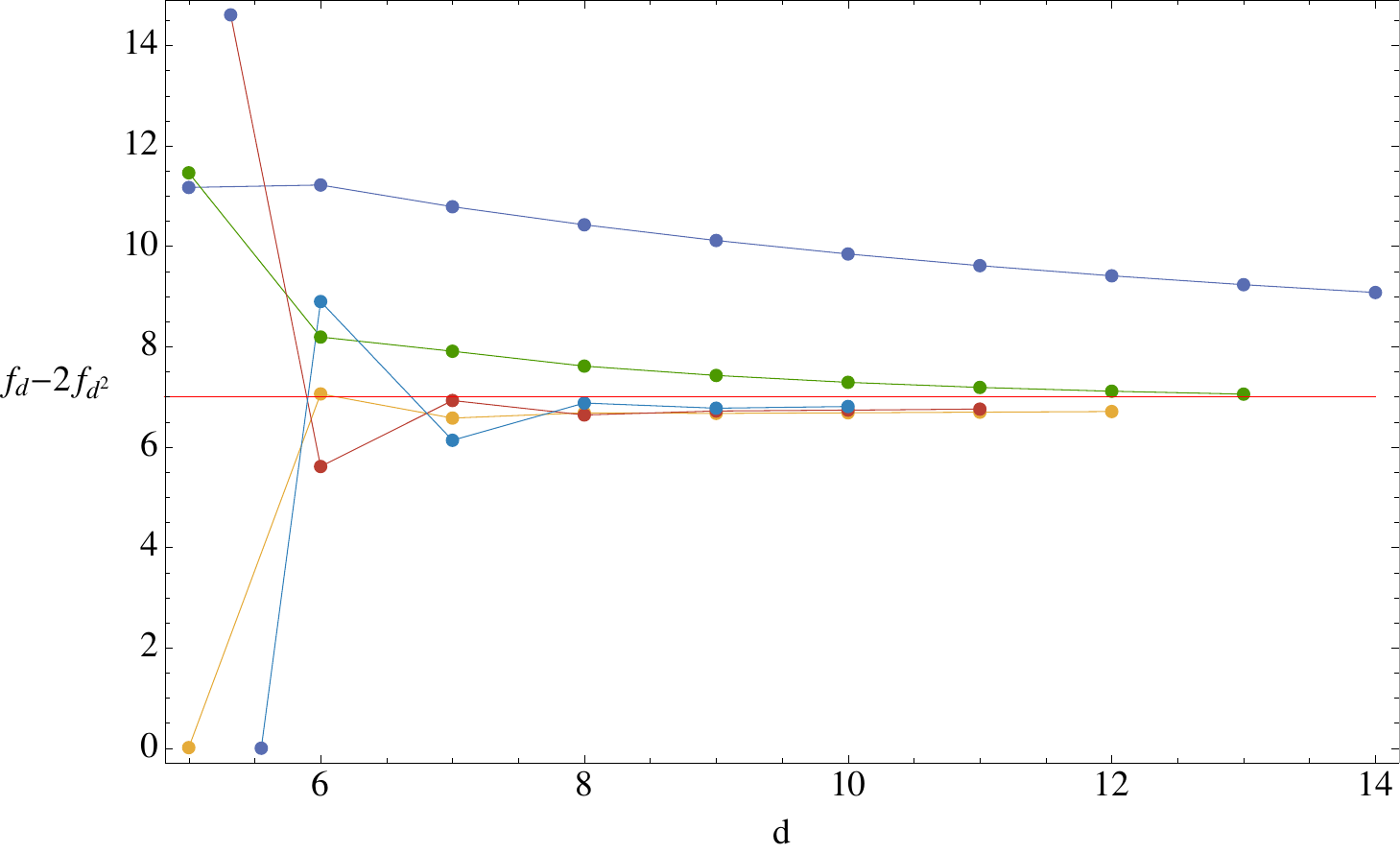}
\caption{ABJM: The exponent $2\boldsymbol{g}-3$ is the leading large-order term in $f_d - 2 f_{d^2}$. We have data up to $d=200$ so that the horizontal axis can only reach $d = 14$. The plots illustrate the first few Richardson transforms for $g=3$ (left) and $g=5$ (right), converging faster towards the expected result (up to numerical relative errors of about $3\%$ in both cases).}
\label{fig:ABJM_large_degree_2g-3}
\end{figure}
%%%%%%%%%%%%%%%%%%%%%%%%%%%%%%%%%%%%%%%%%%%%%%%%%%%%%%%%%%%%%%%%%

%%%%%%%%%%%%%%%%%%%%%%%%%%%%%%%%%%%%%%%%%%%%%%%%%%%%%%%%%%%%%%%%%
\begin{figure}[t!]
\centering
\includegraphics[width=0.46\textwidth]{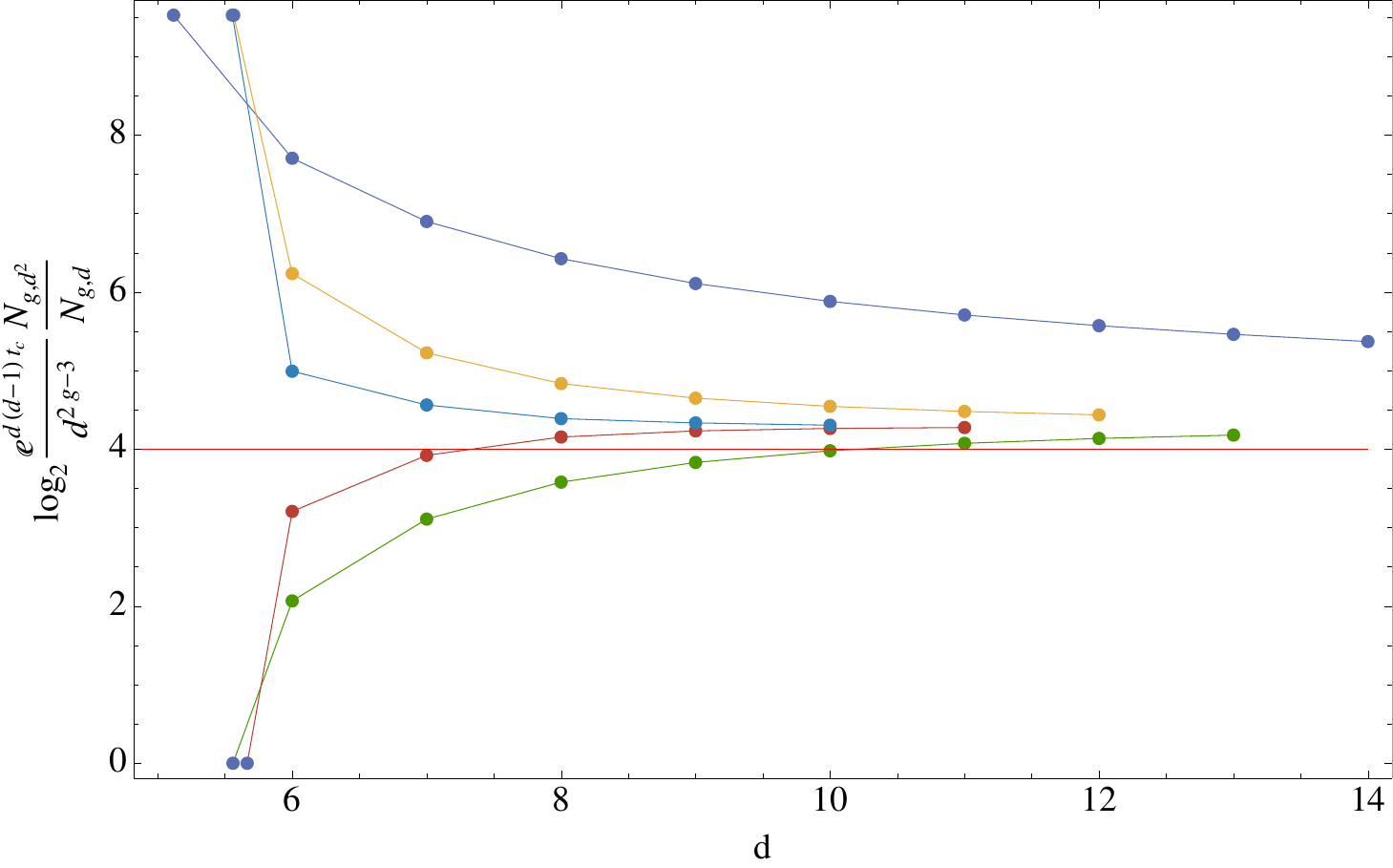}
\hspace{0.05\textwidth}
\includegraphics[width=0.46\textwidth]{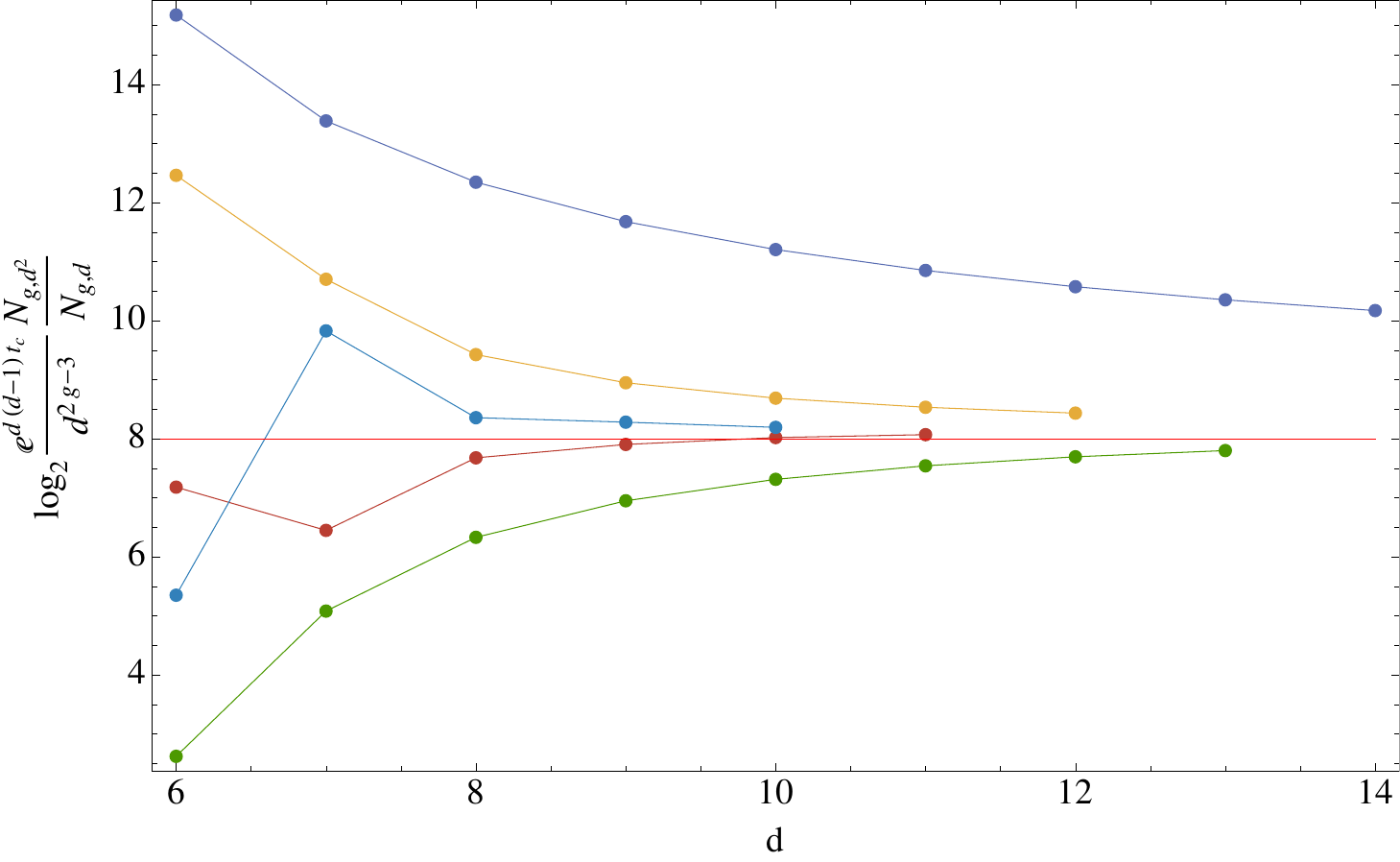}
\caption{ABJM: The exponent $\delta$ of the logarithm $\log d$ is the leading term in the sequence \eqref{eq:localP2_gammacombination}. Having data up to $d=200$ implies the horizontal axis only reaches $d = 14$. We plot the first few Richardson transforms for $g=3$ (left) and $g=5$ (right), converging faster towards the expected result (up to small numerical relative errors of about $8\%$ and $3\%$, respectively).}
\label{fig:ABJM_large_degree_2g-2}
\end{figure}
%%%%%%%%%%%%%%%%%%%%%%%%%%%%%%%%%%%%%%%%%%%%%%%%%%%%%%%%%%%%%%%%%

%%%%%%%%%%%%%%%%%%%%%%%%%%%%%%%%%%%%%%%%%%%%%%%%%%%%%%%%%%%%%%%%%
\subsubsection*{Analysis of Large-Genus Growth}
%%%%%%%%%%%%%%%%%%%%%%%%%%%%%%%%%%%%%%%%%%%%%%%%%%%%%%%%%%%%%%%%%

Again, the strategy is essentially the same as for the example of local $\BP^2$. As before, the GW invariants can be expanded in terms of $abc$-coefficients according to equation \eqref{eq:GW_abc}, but where in the present example one explicitly has $G_{\text{ABJM}} (d) = \floor{d(d-4)/4}+1$. A table with these first few coefficients is shown in appendix~\ref{sec:appendix:localP1xP1}.

The first few coefficients seem to answer to the closed-form formula
\begin{equation}
b^\tABJM_{d,\floor*{\frac{d(d-4)}{4}}-k} = p_{d,-4}(k)\, 4d \left( \floor*{\frac{\left(d+2\right)^2}{4}}-k\right),
\end{equation}
\noindent
where
\begin{equation}
p_{d,-4}(k) =
\begin{cases}
p_{\text{e},-4}(k)/2 & \text{even $d$}, \\
p_{\text{o},-4}(k)   & \text{odd $d$},
\end{cases}
\end{equation}
\noindent
and
\begin{align}
\sum_{k=0}^{+\infty} p_{\text{e},-4}(k)\, q^k &\stackrel{?}{=} \left(1+2q+2q^4\right) \prod_{m=1}^{+\infty} \frac{1}{\left(1-q^m\right)^4}, \\
\sum_{k=0}^{+\infty} p_{\text{o},-4}(k)\, q^k &\stackrel{?}{=} \left(1+q^2+q^6\right) \prod_{m=1}^{+\infty} \frac{1}{\left(1-q^m\right)^4}. 
\end{align}
\noindent
But with the (limited) available data we cannot confirm that these formulae are complete.

%%%%%%%%%%%%%%%%%%%%%%%%%%%%%%%%%%%%%%%%%%%%%%%%%%%%%%%%%%%%%%%%%
\subsubsection*{Combined/Diagonal Large-Growth in Genus and Degree}
%%%%%%%%%%%%%%%%%%%%%%%%%%%%%%%%%%%%%%%%%%%%%%%%%%%%%%%%%%%%%%%%%

The combined growth in genus and degree is similar to the one for local $\BP^2$. The two leading combinations, again arising from K\"ahler and conifold instanton actions, will allow us to connect the factorial growth of the perturbative free energies with the factorial growth of GW invariants. This is illustrated in figure~\ref{fig:ABJM_leading_degrees}. As already happened before, the growth associated to the K\"ahler action, $d = (2g-3)/t$, is well understood since the example of the resolved conifold, whereas the one associated to the conifold action, $d = a_0(Q) + a_1(Q)\, g$, can only be probed numerically. 

%%%%%%%%%%%%%%%%%%%%%%%%%%%%%%%%%%%%%%%%%%%%%%%%%%%%%%%%%%%%%%%%%
\begin{figure}[t!]
\centering
\includegraphics[width=0.46\textwidth]{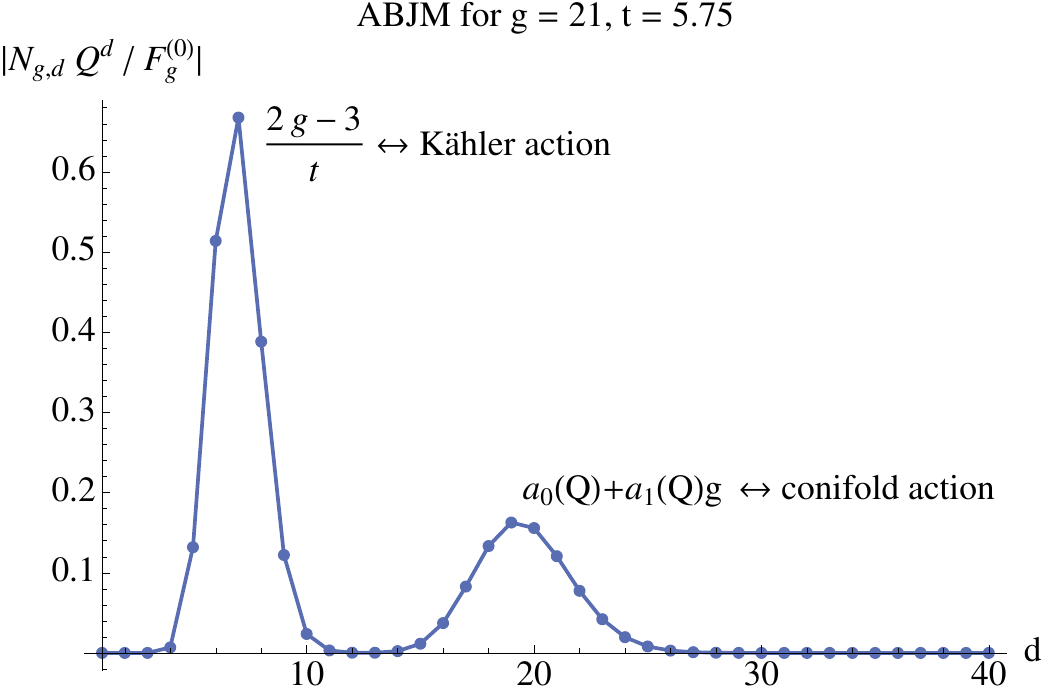}
\caption{ABJM: Graphical representation of which GW invariants contribute the most to a free energy $F^{(0)}_g(Q)$, for fixed values of $g$ and $Q=\rme^{-t}$. As for local $\BP^2$ in figure~\ref{fig:leading_degrees}, ABJM has saddle points corresponding to both K\"ahler and conifold actions. The values of $g$ and $t$ in the plot were chosen as to clearly see both saddles in the same figure.}
\label{fig:ABJM_leading_degrees}
\end{figure}
%%%%%%%%%%%%%%%%%%%%%%%%%%%%%%%%%%%%%%%%%%%%%%%%%%%%%%%%%%%%%%%%%

%%%%%%%%%%%%%%%%%%%%%%%%%%%%%%%%%%%%%%%%%%%%%%%%%%%%%%%%%%%%%%%%%
\subsubsection*{K\"ahler Leading Degree}
%%%%%%%%%%%%%%%%%%%%%%%%%%%%%%%%%%%%%%%%%%%%%%%%%%%%%%%%%%%%%%%%%

This growth is completely determined, at least up to the first exponentially-subleading instanton corrections, by the first GV invariant of ABJM which in this case is $n_0^{(1)} = -4$. Thus we have
\begin{equation}
\left. N^\tABJM_{g,d}\, Q^d\right|_{g=\frac{t}{2}d+\sh} \sim \sum_{h=0}^{+\infty} \frac{\Gamma \left(2g-\frac{3}{2}-h\right)}{A_{\text{K}}^{2g-\frac{3}{2}-h}}\, \frac{n_0^{(1)}\, t^{\frac{3}{2}-h}}{2^{2h+1}\, \pi^{h+2}}\, \polyname_h(\sh),
\label{eq:GW_ABJM_large_order_Kahler}
\end{equation}
\noindent
where $A_{\text{K}} = 2\pi t$ and the definition of $\polyname_h(\sh)$ is given in \eqref{eq:CP_h_polynomials}, \textit{i.e.}, they are precisely the same polynomials which have already appeared for resolved conifold and local $\BP^2$. Computational tests on the validity of \eqref{eq:GW_ABJM_large_order_Kahler} are shown in figures~\ref{fig:ALL_Kahler_Action} and~\ref{fig:ALL_LOOPS}, with the exact same discussion as for resolved conifold and local $\BP^2$. In fact, all the very same comments we made for local $\BP^2$ in section~\ref{sec:P2} also apply now. In particular, the numerical check of $a_0 = -3/t$ and $a_1 = 2/t$ is indeed confirmed as
\begin{alignat}{2}
a_0(Q)^{-1} &= \left(0.01 \pm 0.07\right) + \left(-0.33 \pm 0.01\right) t, \qquad && r^2 = 0.970, \\
a_1(Q)^{-1} &= \left(-0.003 \pm 0.005\right) + \left(0.5000 \pm 0.0009\right) t, \qquad && r^2 = 0.99992.
\end{alignat}
\noindent
This check is shown in the upper plots of figure~\ref{fig:ABJM_1overa0a1}.

%%%%%%%%%%%%%%%%%%%%%%%%%%%%%%%%%%%%%%%%%%%%%%%%%%%%%%%%%%%%%%%%%
\begin{figure}[t!]
\centering
\includegraphics[width=0.45\textwidth]{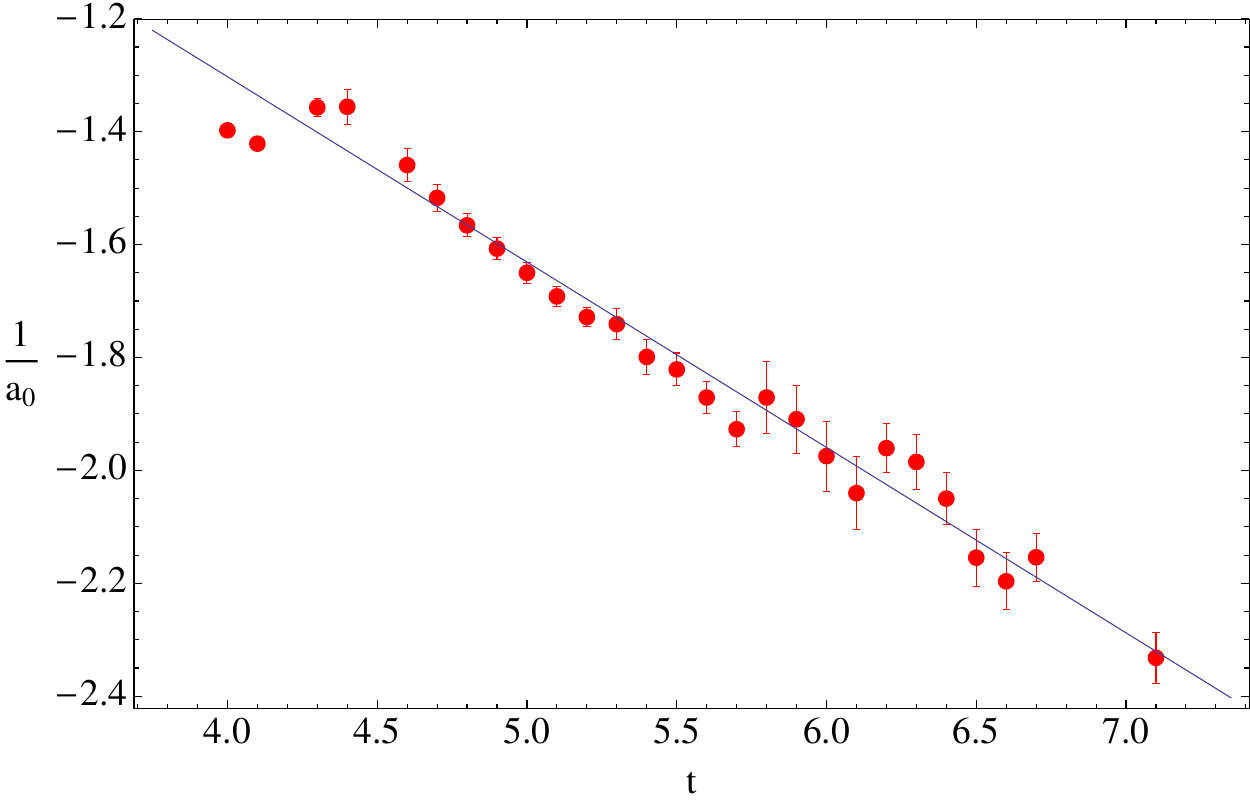}
\hspace{0.05\textwidth}
\includegraphics[width=0.45\textwidth]{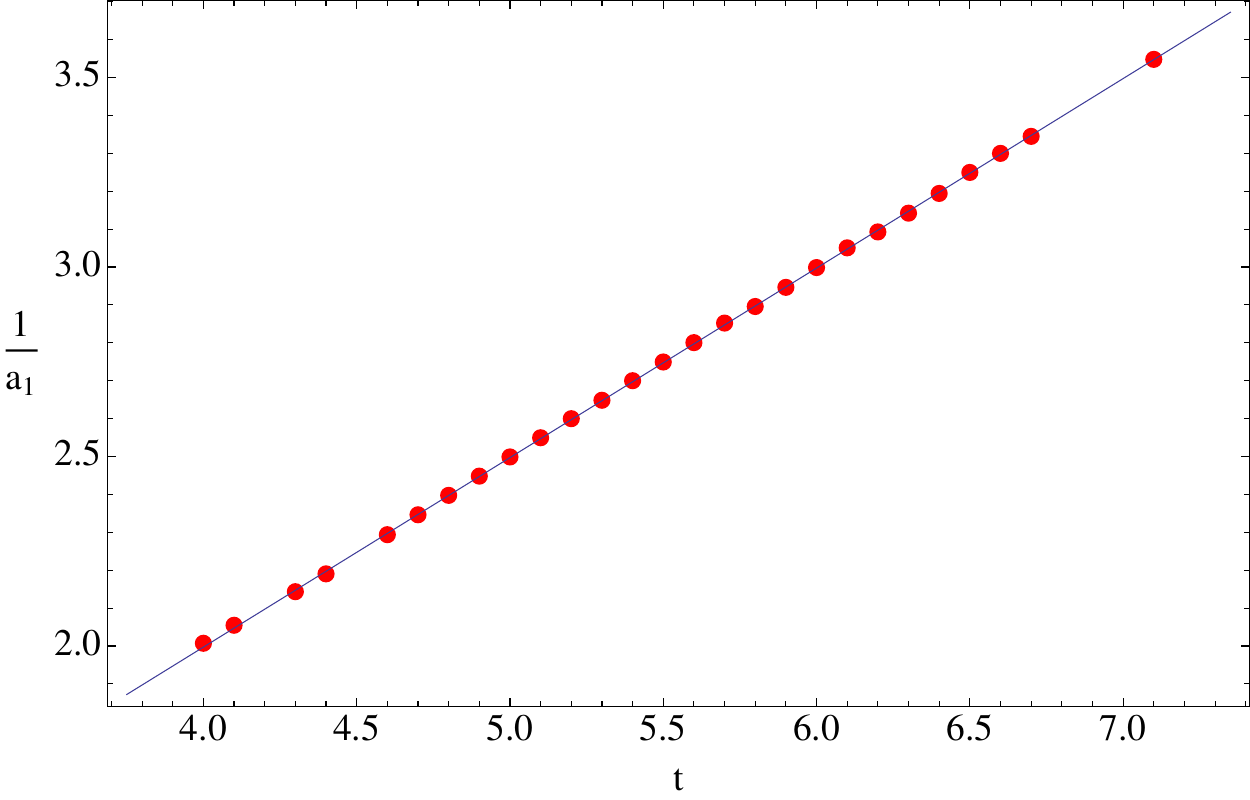}
\includegraphics[width=0.45\textwidth]{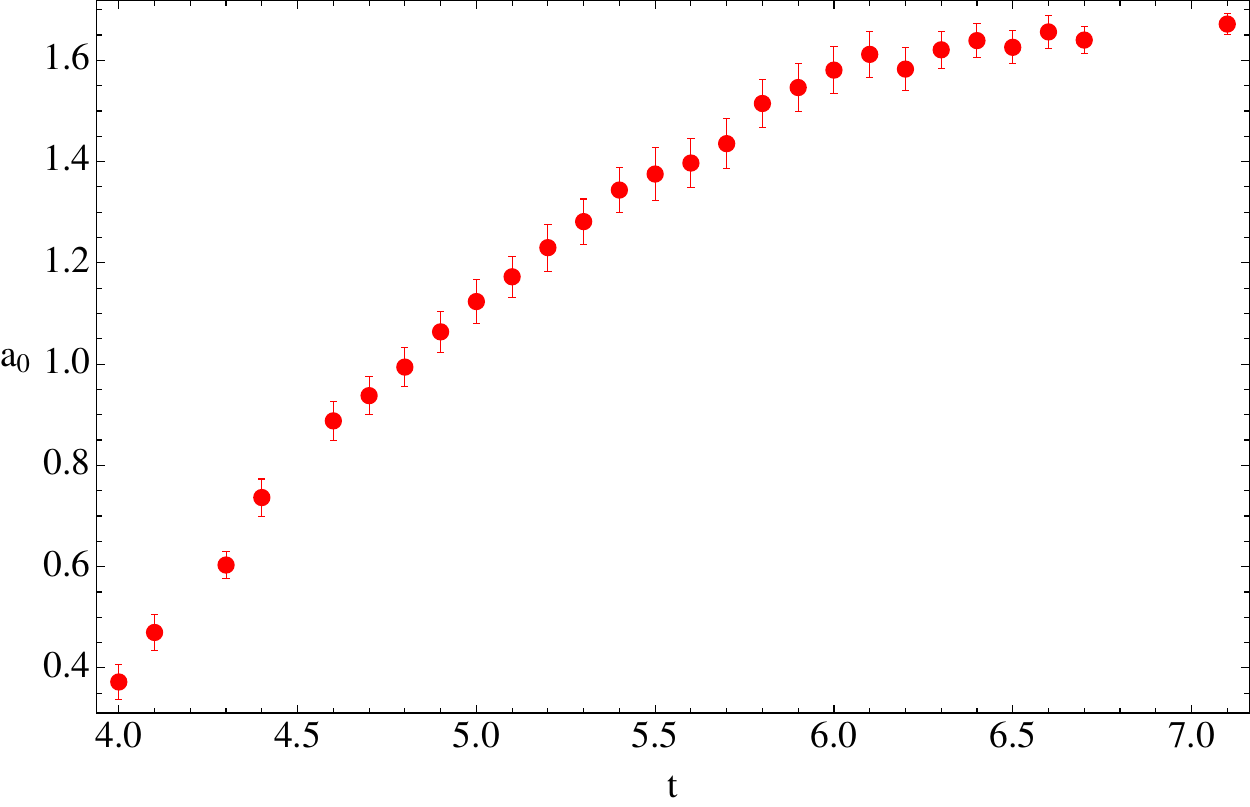}
\hspace{0.05\textwidth}
\includegraphics[width=0.45\textwidth]{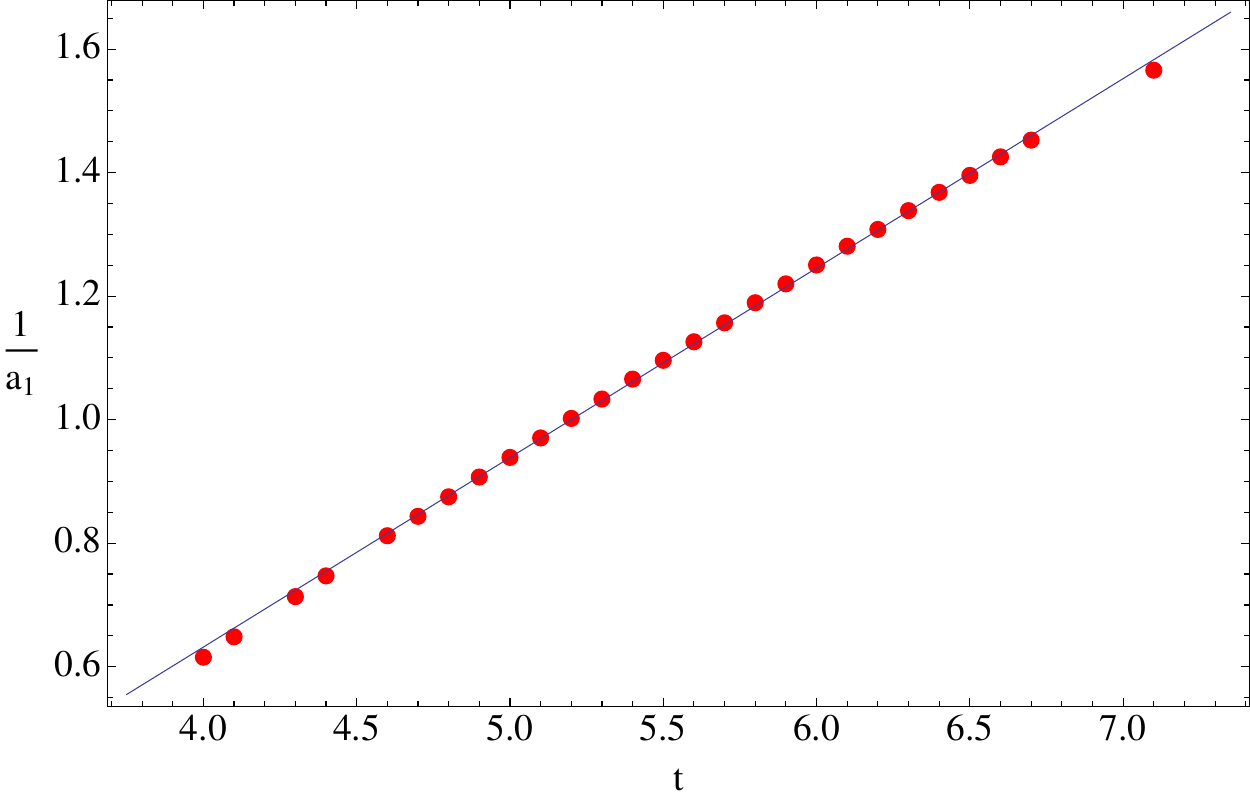}
\caption{ABJM: Numerical calculation of $a_0(Q)$ and $a_1(Q)$ associated to the K\"ahler instanton action (the two upper plots) and to the conifold instanton action (the two lower plots). We are showing test for their inverses whenever the dependence seems linear, although we were not able to confirm this analytically in the case of the conifold action.}
\label{fig:ABJM_1overa0a1}
\end{figure}
%%%%%%%%%%%%%%%%%%%%%%%%%%%%%%%%%%%%%%%%%%%%%%%%%%%%%%%%%%%%%%%%%

%%%%%%%%%%%%%%%%%%%%%%%%%%%%%%%%%%%%%%%%%%%%%%%%%%%%%%%%%%%%%%%%%
\subsubsection*{Conifold Leading Degree}
%%%%%%%%%%%%%%%%%%%%%%%%%%%%%%%%%%%%%%%%%%%%%%%%%%%%%%%%%%%%%%%%%

In this case, the analysis can only be carried out numerically, due to lack of theoretical knowledge of where $a_0(Q)$ and $a_1(Q)$ come from. The strategy is essentially the one already started with local $\BP^2$, and we find
\be
a_1(Q)^{-1} = \left(-0.595 \pm 0.009\right) + \left(0.307 \pm 0.002\right) t, \qquad r^2 = 0.9993.
\ee
\noindent
This fit is shown in the lower-right plot of figure~\ref{fig:ABJM_1overa0a1}. The lower-left plot of this figure shows the numerical calculation of $a_0(Q)$, again with no obvious fit to do here.

%%%%%%%%%%%%%%%%%%%%%%%%%%%%%%%%%%%%%%%%%%%%%%%%%%%%%%%%%%%%%%%%%
\subsection{The Example of the Local Curve $X_p$}\label{sec:curve}
%%%%%%%%%%%%%%%%%%%%%%%%%%%%%%%%%%%%%%%%%%%%%%%%%%%%%%%%%%%%%%%%%

Our next example deals with local curves. The non-compact CY threefolds to be considered are the total spaces of the rank-two holomorphic vector bundles $X_p \simeq \CO (p-2) \oplus \CO (-p) \to \BP^1$, with $p$ an integer (but due to the invariance  $p-2 \leftrightarrow -p$, one may choose $p \in \BN$). When $p=1$ one finds the resolved conifold, $\CO (-1) \oplus \CO (-1) \to \BP^1$ (addressed earlier), and when $p=2$ one finds the Dijkgraaf--Vafa geometries $\CO (0) \oplus \CO (-2) \to \BP^1$ relating to hermitian matrix models \cite{dv02}.

By making use of the topological vertex machinery \cite{akmv03} one may actually compute  high genus GW invariants for the local curve directly in the A-model. We shall nonetheless begin with some comments pertaining to the B-model free energy, following  \cite{cgmps06}.

%%%%%%%%%%%%%%%%%%%%%%%%%%%%%%%%%%%%%%%%%%%%%%%%%%%%%%%%%%%%%%%%%
\subsubsection*{Free Energies and Gromov--Witten Invariants}
%%%%%%%%%%%%%%%%%%%%%%%%%%%%%%%%%%%%%%%%%%%%%%%%%%%%%%%%%%%%%%%%%

The B-model free energy has the general structure \cite{cgmps06}
\be 
F_g^\tLC{p}( w) = \frac{1}{\left(w-w_{\text{c}}\right)^{5(g-1)}}\, \sum_{n=1}^{5(g-1)} a_{g,n}(p) \left(w-1\right)^n,
\label{FgBmodel}
\ee
\noindent
where the coefficients $a_{g,n}$ are of the form
\be 
a_{g,n} = \frac{b_{g,n}(p)}{(p-1)^k},
\ee
\noindent
with $k$ a positive integer and $b_{g,n}(p)$ a polynomial in $p$. They are not known in general and have to be fixed with GW invariants up to degree $d=5(g-1)$ (we present some of these coefficients in appendix~\ref{sec:appendix_local}). The modulus $w$ is related to the K\"ahler parameter $t$ through the mirror map
\be 
Q \equiv \rme^{-t} = w^{(p-1)^2-1} - w^{(p-1)^2},
\label{mirror_Xp}
\ee
\noindent
where the critical point is at
\be 
w_{\text{c}} = \frac{p(p-2)}{(p-1)^2},
\ee
\noindent
which translates to
\be
t_{\text{c}} = \log \left( \left( p \left(p-2\right) \right)^{p \left(2-p\right)} \left(p-1\right)^{2\left(p-1\right)^2} \right).
\ee
\noindent
It is interesting to notice that, unlike the previous geometries, all these formulae are now \textit{exact}. Further notice that the double-scaled theory at the critical point is now in the universality class of 2d gravity (the free energy being related to the Painlev\'e I equation) which is a distinct universality class from the previous $c=1$ examples \cite{cgmps06}.

As mentioned earlier, one can compute the partition function, and thus the free energy, as a sum over integer partitions directly in the A-model using the topological vertex \cite{akmv03}. We shall not get into any details, which may be found in \cite{cgmps06}, and simply quote the end results. We have computed\footnote{On a technical aside, let us mention that the main obstacle in such A-model computations is the growth in degree, since it implies considering an exponentially-growing number of partitions and ever larger expressions to put together. The expansion of the free energy in powers of $g_{\text{s}}$ is also time and resource-consuming, but this computation can be improved if we compute the GW invariants numerically. The only requisite is that the numerical precision should be high enough in order to extract rational numbers out of decimal expansions.} GW invariants $N_{g,d}^\tLC{p}$ with fixed $p=3,4,5$ and in appendix~\ref{sec:appendix_local} we list a few such invariants. Figure~\ref{fig:data} schematically represents all the invariants we did compute and will work with herein. In the rest of this section we will mostly omit the $p$-dependence of the GW invariants for shortness, but our results for the different types of growth will always be for general $p$ unless explicitly stated otherwise. We should also point out that we are including an extra sign in our GW invariants \cite{gv98a, gv98c}
\be 
N_{g,d}^\tLC{p} \rightarrow (-1)^{g-1}\, N_{g,d}^\tLC{p}.
\ee
\noindent
This is essentially required in order to produce integer GV invariants.

%%%%%%%%%%%%%%%%%%%%%%%%%%%%%%%%%%%%%%%%%%%%%%%%%%%%%%%%%%%%%%%%%
\begin{figure}[t!]
\begin{center}
\raisebox{0.5cm}{\includegraphics[width=0.5\textwidth]{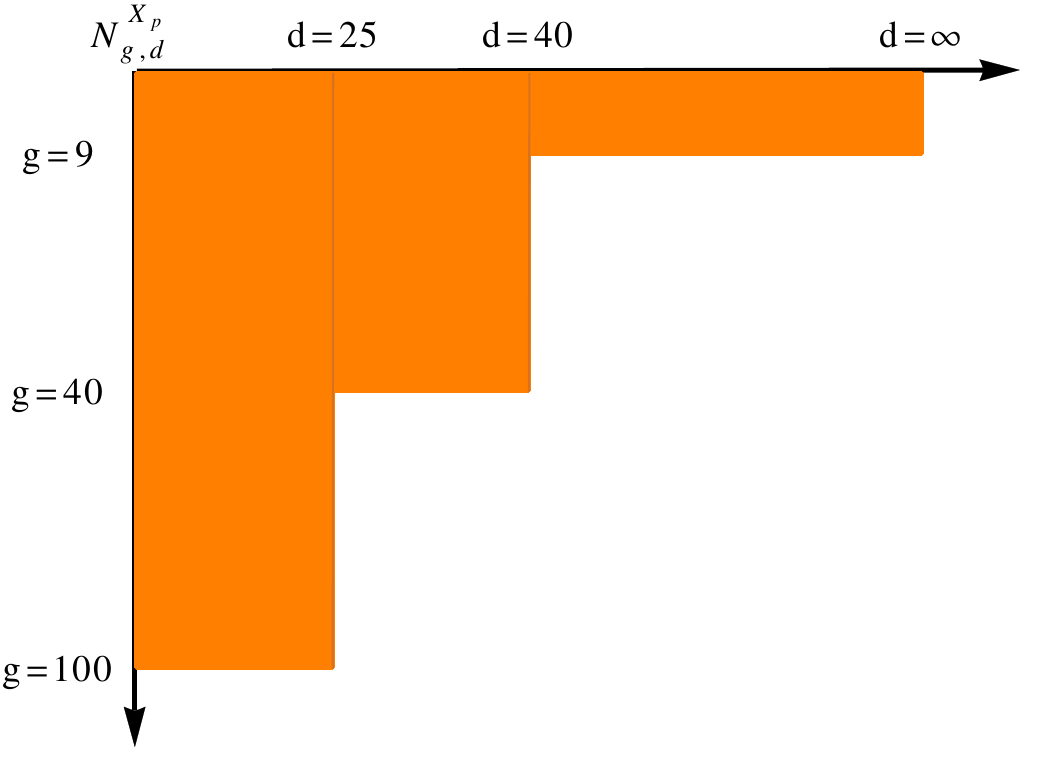}}
\end{center}
\vspace{-1\baselineskip}
\caption{Maximum degree and genus of the GW invariants we computed for the local curves $X_p$ with $p=3,4,5$. For $g\leq9$ we have all the required data to fix \eqref{FgBmodel} and thus can compute GW invariants for any degree.}
\label{fig:data}
\end{figure}
%%%%%%%%%%%%%%%%%%%%%%%%%%%%%%%%%%%%%%%%%%%%%%%%%%%%%%%%%%%%%%%%%

For $g\leq9$ we have enough data to completely fix the coefficients in \eqref{FgBmodel}, which means the GW invariants can then be computed to arbitrarily high degree. It is also worth mentioning that for $g=0, 1$ there are explicit formulae for the GW invariants \cite{cgmps06}
\bea
N_{0,d}^\tLC{p} (p) &=& - \frac{\left( d\, f-1 \right)!}{d!\, d^2\, \left( d \left( f-1 \right) \right)!}, \\
N_{1,d}^\tLC{p} (p) &=& \frac{1}{24 d} \sum_{n=0}^{d-1} \frac{f^{d-n}}{n!} \prod_{k=1}^{n} \left( d \left( f-1 \right) + k - 1 \right) - \frac{1}{24} \frac{\left( d\, f - 1 \right)!}{d!\, \left( d \left( f-1 \right) \right)!} \left( f+2 \right),
\eea 
\noindent
where we have set $f \equiv \left( p-1 \right)^2$, and where higher-genus closed-form generalizations are not known. One general thing which is known \cite{cgmps06} is that if written for arbitrary $p$, as $N_{g,d}^\tLC{p} (p)$, the GW invariants are polynomials of degree $2g+2d-2$ in $p$, with rational coefficients,
\be
N_{g,d}^\tLC{p} (p) = \sum_{n=0}^{2g+2d-2} N_{g,d,n}\, p^n.
\ee

%%%%%%%%%%%%%%%%%%%%%%%%%%%%%%%%%%%%%%%%%%%%%%%%%%%%%%%%%%%%%%%%%
\subsubsection*{Analysis of Large-Degree Growth}
%%%%%%%%%%%%%%%%%%%%%%%%%%%%%%%%%%%%%%%%%%%%%%%%%%%%%%%%%%%%%%%%%

The analysis of the fixed-genus, large-degree growth of GW invariants in this example is best achieved within the B-model formulation \eqref{FgBmodel}. The A-model (topological vertex) calculation is only efficient up to about degree $d=40$ and, as we shall see, the large-degree convergence of the GW invariants is very slow. Using data up to genus $g=40$ (for $p=3,4$) we have fixed the coefficients $a_{g,n}$ up to $g=9$ and then can compute $N_{g\leq9,d}$ up to very high degree.

To see how this works, rewrite \eqref{FgBmodel} as
\be 
F_g^\tLC{p}(t) = \sum_{k=0}^{5(g-1)} \alpha_{g,k} \left(w-w_{\text{c}}\right)^{k-5(g-1)} \qquad \text{with} \qquad \alpha_{g,k} = \sum_{n=1}^{5(g-1)} \binom{n}{k}\, a_{g,n} \left(w_{\text{c}}-1\right)^{n-k}.
\label{FgBmodel2}
\ee
\noindent
A Lagrange inversion turns \eqref{FgBmodel2} into a $Q$-expansion (related to $w$ via the mirror map \eqref{mirror_Xp}), from where GW invariants are easily extracted. We will skip these details and refer the reader to Appendix~A of \cite{cgmps06} for the definition of Lagrange inversion as well as instructive examples. In the end, one finds the explicit result
\be 
N_{\boldsymbol{g},d}^\tLC{p} = \frac{(-1)^{d-1}}{d}\, \sum_{k=0}^{5(\boldsymbol{g}-1)} \alpha_{\boldsymbol{g},k}\, \left( 5 \left(g-1\right) - k \right)\, f^{d + 5 \left( \boldsymbol{g}-1 \right) - k}\, P_{d-1}^{(u,v)} \left(\frac{f-2}{f}\right),
\label{GW_lagrange}
\ee
\noindent
where
\be 
u = k - d - 5 \left(\boldsymbol{g}-1\right), \qquad v = d \left(f-1\right) + 5\left(\boldsymbol{g}-1\right)-k,
\ee
\noindent
and where the $P_n^{(a,b)}(z)$ are Jacobi polynomials. Finding the large-degree behavior of the GW invariants now reduces to the corresponding large-degree behavior of the Jacobi polynomials. We still had to approach this behavior numerically, essentially because the degree $d$ appears in three different places. Furthermore, the aforementioned slow convergence of the GW invariants will be made clear in the following, as the large-degree expansion turns out to be a power-series expansion in $\sqrt{d}$ for which our standard techniques of Richardson extrapolation are not very useful. However, restricting the study to a grid of perfect squares, \textit{i.e.}, $d=\ell^2$, we then get back the very fast convergence via Richardson transforms, from where one can then comfortably find rational numbers out of decimal expansions.

Consider the following combination
\be 
P_{d,\boldsymbol{g},k} (f) \equiv f^{d + 5\left(\boldsymbol{g}-1\right)-k}\, \left( 5 \left(\boldsymbol{g}-1\right) - k \right)\, P_{d-1}^{(u,v)} \left(\frac{f-2}{f}\right),
\ee
\noindent
for which we find a large-degree expansion of the form
\be 
P_{d,\boldsymbol{g},k} (f) = (-1)^{d-1}\, \rme^{d t_{\text{c}}}\, d^{\frac{5}{2} \left(\boldsymbol{g}-1\right) - \frac{k}{2}} \sum_{n=0}^{+\infty} c_{\boldsymbol{g},k}^{(n)}\, d^{-\frac{n}{2}} \simeq (-1)^{d-1}\, \rme^{d t_{\text{c}}}\, d^{\frac{5}{2} \left(\boldsymbol{g}-1\right) - \frac{k}{2}} \left( c_{\boldsymbol{g},k}^{(0)} + \frac{c_{\boldsymbol{g},k}^{(1)}}{\sqrt{d}} + \cdots \right).
\ee
\noindent
Using $\widehat{g} \equiv 5(g-1)$ for shortness, the first coefficients are
\be
c_{g,k}^{(0)} = \frac{\rme^{\frac{1}{2} \left(\widehat{g}-k\right) t_{\text{c}}}\, \mathcal{A}^{k-\widehat{g}}}{\Gamma\left(\frac{1}{2} \left( \widehat{g}-k \right) \right)}, \qquad c_{g,k}^{(1)} = \frac{\sqrt{2}}{3}\, \frac{\rme^{\frac{1}{2} \left(\widehat{g}-k\right) t_{\text{c}}}\, \mathcal{A}^{k-\widehat{g}}}{\Gamma\left(\frac{1}{2} \left( \widehat{g}-k-1 \right) \right)}\, \frac{f-2}{\sqrt{f \left(f-1\right)}} \left(\widehat{g}-k\right), 
\ee
\noindent
where we have defined
\be 
\mathcal{A} = \sqrt{2}\, \frac{w_{\text{c}}^{1-(p-1)^2/2}}{p-1}.
\ee
\noindent
In general, they will have the structure
\be 
c_{g,k}^{(j)} = \frac{\rme^{\frac{1}{2} \left(\widehat{g}-k\right) t_{\text{c}}}\, \mathcal{A}^{k-\widehat{g}}}{\Gamma\left( \frac{1}{2} \left( \widehat{g}-k-j \right) \right)}\, \frac{1}{\left( f \left( f-1 \right) \right)^{\frac{j}{2}}}\, \sum_{j_0=1}^{j} \widehat{c}_{j_0}^{(j)}(f) \left(\widehat{g}-k\right)^{j_0},
\ee
\noindent
with $\widehat{c}_{j_0}^{(j)}(f)$ a polynomial in $f$ of degree $j$. We refer to appendix~\ref{sec:appendix_local} for a few such explicit results at the lowest orders. From these results it immediately follows that the leading term in $P_{d,g,0}$ reproduces the known behavior of the GW invariants \eqref{larged_general} (with $\alpha=\beta=0$ and $\gamma=-1/2$), and the coefficients can be shown to match the solution of Painlev\'e I, in the appropriate double-scaling limit. This leading behavior then has corrections, suppressed by powers of $d^{-1/2}$.

The GW invariants for arbitrary $p$ (or $f$) thus have the following large-degree expansion\footnote{Note that one interesting feature of this $\gamma=-1/2$ universality class, and as compared to the general structure in \eqref{larged_general}, is that there are no logarithmic contributions to the large-degree asymptotics.}
\bea 
N_{\boldsymbol{g},d} &\sim& \rme^{d t_{\text{c}}}\, d^{\frac{5}{2} \left(\boldsymbol{g}-1\right) - 1} \left( c_{\boldsymbol{g},0}^{(0)}\, \alpha_{\boldsymbol{g},0} + \frac{c_{\boldsymbol{g},1}^{(0)}\, \alpha_{\boldsymbol{g},1} + c_{\boldsymbol{g},0}^{(1)}\, \alpha_{\boldsymbol{g},0}}{\sqrt{d}} + \cdots \right) = \\
&=& \rme^{d t_{\text{c}}}\, d^{\frac{5}{2} \left(\boldsymbol{g}-1\right) - 1}\, \sum_{j=0}^{+\infty}\, \sum_{j'=0}^{\text{Min} \left(j, \,5(\boldsymbol{g}-1)\right)} c_{\boldsymbol{g},j'}^{(j-j')}\, \alpha_{\boldsymbol{g},j'}\, d^{-\frac{j}{2}}.
\label{GWlarged}
\eea
\noindent
This result is illustrated and tested in figure~\ref{fig:larged}. On its left plot we consider the ratio $N_{\boldsymbol{g},d}^{(\text{pred})}/N_{\boldsymbol{g},d}$, where the asymptotic prediction in the numerator consists of using the expansion \eqref{GWlarged} up to the subleading correction $j=0$ (blue), $j=2$ (green), $j=4$ (yellow) and $j=6$ (red), up to degree $d=100$ and fixed genus $g=3$. At degree $d=100$ the leading-order term is of the right order of magnitude, but it is still off by about $\sim 75\%$. Upon inclusion of subleading terms, the ratio then starts approaching $1$, faster and faster. On the right plot of figure~\ref{fig:larged} we show the number of decimal places of agreement between the GW invariants and their large-degree expansion, with the same color coding. This time we move only along perfect squares and work with fixed genus $g=6$. At $d=10000$ the leading term is still only good enough for the first decimal place, but then the agreement improves significantly once we start adding subleading corrections.

%%%%%%%%%%%%%%%%%%%%%%%%%%%%%%%%%%%%%%%%%%%%%%%%%%%%%%%%%%%%%%%%%
\begin{figure}[t!]
\begin{center}
\raisebox{0.5cm}{\includegraphics[width=0.49\textwidth]{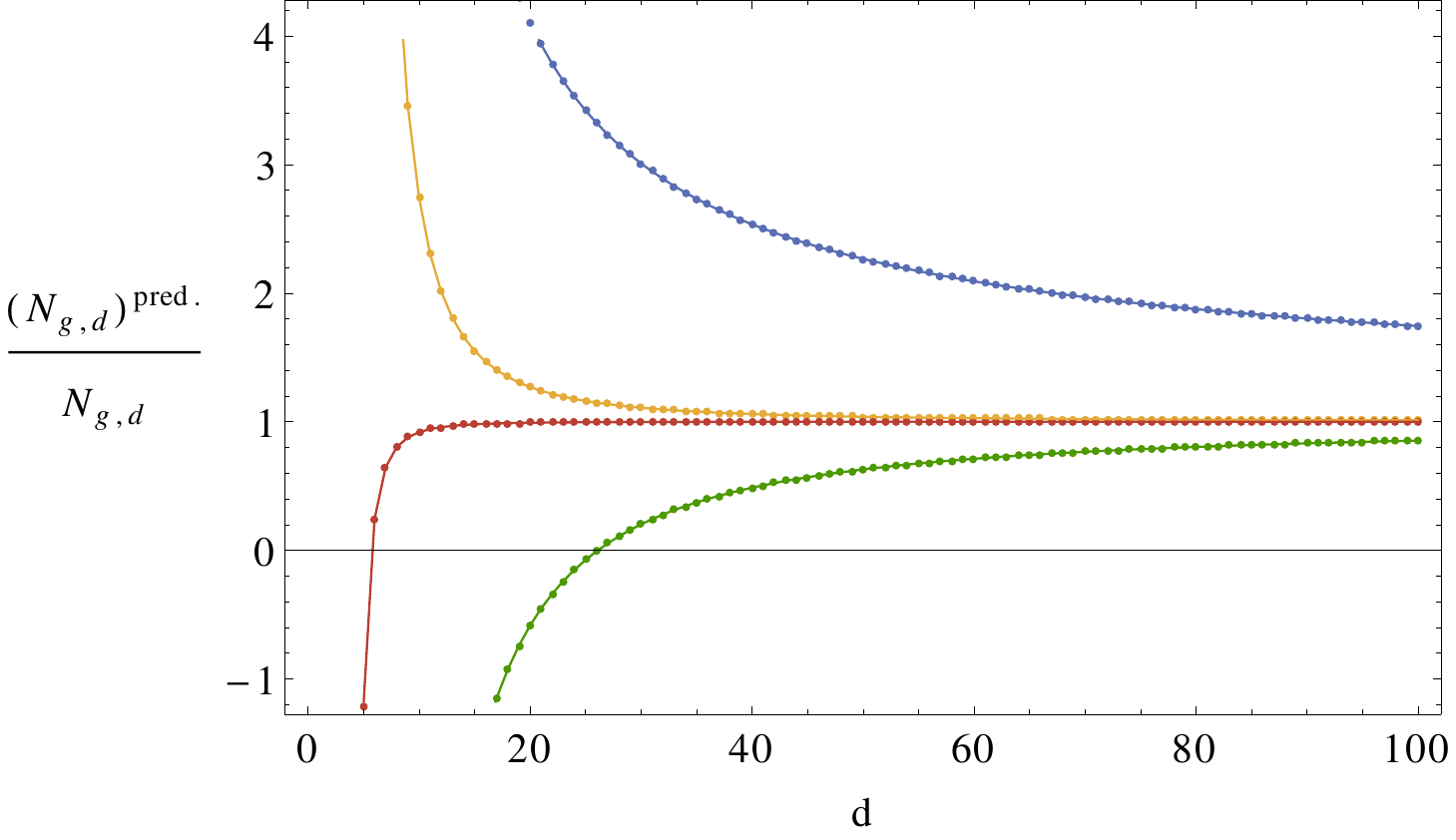}}
\hspace{0.02\textwidth}
\includegraphics[width=0.47\textwidth]{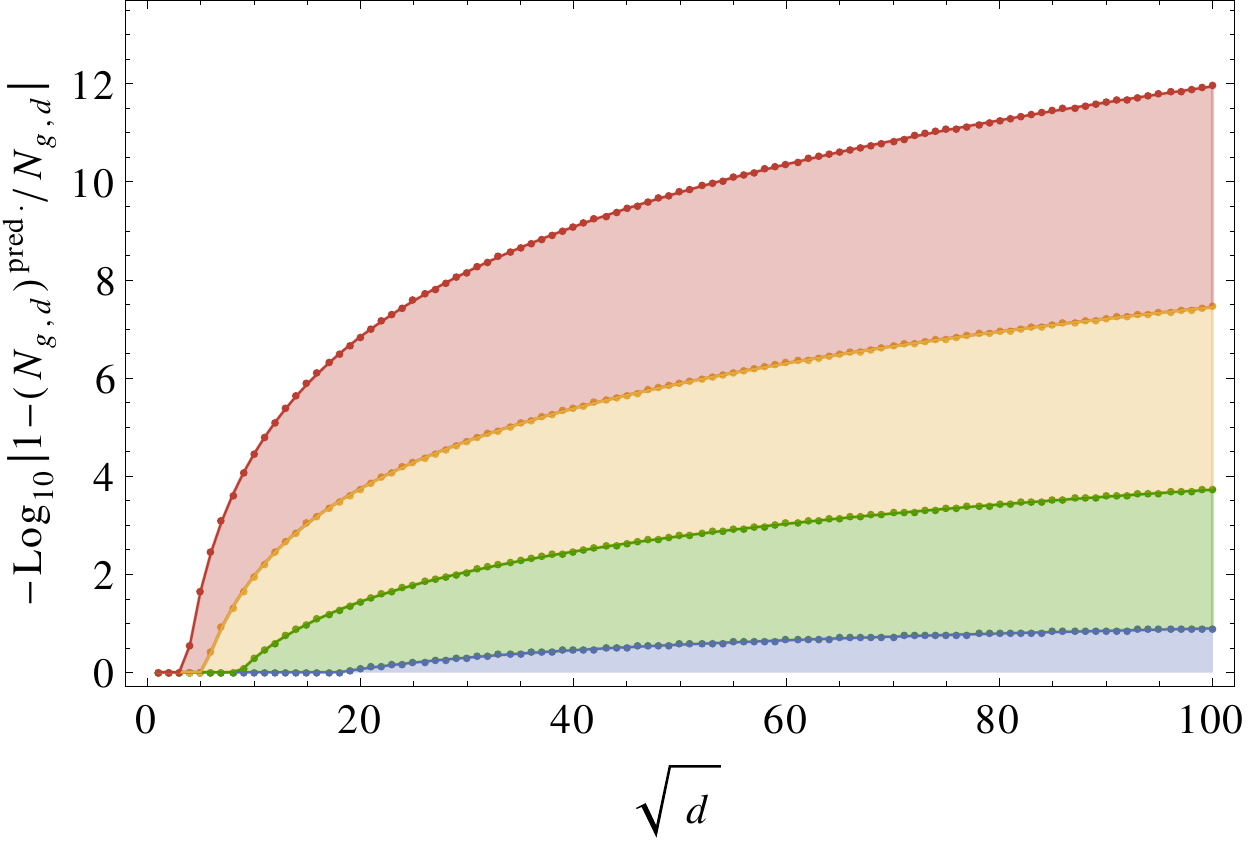} 
\end{center}
\vspace{-1\baselineskip}
\caption{Local curve: Left: Comparison between $p=3$ GW invariants and their asymptotic prediction in \eqref{GWlarged}, at fixed genus $g=3$ and degree $d\leq 100$, and up to subleading correction $j=0$ (blue), $j=2$ (green), $j=4$ (yellow) and $j=6$ (red). Right: Number of decimal places of agreement between the analytical $p=3$ GW invariants and their asymptotic prediction, for fixed genus $g=6$ (and using the same color code).}
\label{fig:larged}
\end{figure}
%%%%%%%%%%%%%%%%%%%%%%%%%%%%%%%%%%%%%%%%%%%%%%%%%%%%%%%%%%%%%%%%%

%%%%%%%%%%%%%%%%%%%%%%%%%%%%%%%%%%%%%%%%%%%%%%%%%%%%%%%%%%%%%%%%%
\subsubsection*{Analysis of Large-Genus Growth}
%%%%%%%%%%%%%%%%%%%%%%%%%%%%%%%%%%%%%%%%%%%%%%%%%%%%%%%%%%%%%%%%%

As discussed earlier, the large-genus fixed-degree growth of the GW invariants may be read directly from the \textit{abc}-formula \eqref{eq:GW_abc}. In the case of the local curve, it simply reads
\be
N^\tLC{p}_{g,\boldsymbol{d}} = f^\tconi_g \left\{ \sum_{n|\boldsymbol{d}} a^\tLC{p}_n \left( \frac{\boldsymbol{d}}{n} \right)^{2g-3} + \frac{2g}{B_{2g}}\, \frac{1}{\boldsymbol{d}} \left( c^\tLC{p}_{\boldsymbol{d}}\, \delta_{g,1} + \sum_{n=1}^{G_\tLC{p}(\boldsymbol{d})-1} b^\tLC{p}_{\boldsymbol{d},n}\, n^{2g-2} \right) \right\},
\ee
\noindent
where now $G_\tLC{p}(d) = (d-1) \left( \left(p-2\right) d-2 \right)/2$. A table with the first few $abc$-coefficients may be found in appendix~\ref{sec:appendix_local}. To show an example, let us write down the first couple of terms for $p=3$
\be 
N_{g,\boldsymbol{d}}^\tLC{3} \sim \frac{2 \left(2g-1\right)}{\boldsymbol{d}^3} \left( \frac{\boldsymbol{d}}{2\pi} \right)^{2g} \left\{ 1 + (-1)^{\boldsymbol{d}-1}\, 7^{\frac{1+(-1)^{\boldsymbol{d}}}{2}}\, \frac{1}{2^{2g}} + 55^{(1+2\boldsymbol{d}^2) \text{ mod } 3}\, \frac{1}{3^{2g}} + \mathcal{O}(4^{-2g}) \right\}.
\ee
\noindent
This expression is valid for any degree, unlike in previous examples.

%%%%%%%%%%%%%%%%%%%%%%%%%%%%%%%%%%%%%%%%%%%%%%%%%%%%%%%%%%%%%%%%%
\subsubsection*{Combined/Diagonal Large-Growth in Genus and Degree}
%%%%%%%%%%%%%%%%%%%%%%%%%%%%%%%%%%%%%%%%%%%%%%%%%%%%%%%%%%%%%%%%%

The combined ``diagonal'' growth will turn out to be similar to the previous local-surface examples (and in fact will lead to some sort of \textit{large-order universality} for topological strings in different double-scaled universality classes). In fact, also in the local-curve case we cannot (analytically) pinpoint the nonperturbative structure of the GW invariants in a general situation where $d = a_0(Q) + a_1(Q)\, g$. As usual, the one exception happens at large-radius $t \rightarrow +\infty$, where the contribution of the GW invariants to the free energy is strongly peaked around $d = (2g-3)/t$. These K\"ahler and critical-point peaks are illustrated in figure~\ref{fig:localcurve_leading_degrees}.

%%%%%%%%%%%%%%%%%%%%%%%%%%%%%%%%%%%%%%%%%%%%%%%%%%%%%%%%%%%%%%%%%
\begin{figure}[t!]
\centering
\includegraphics[width=0.46\textwidth]{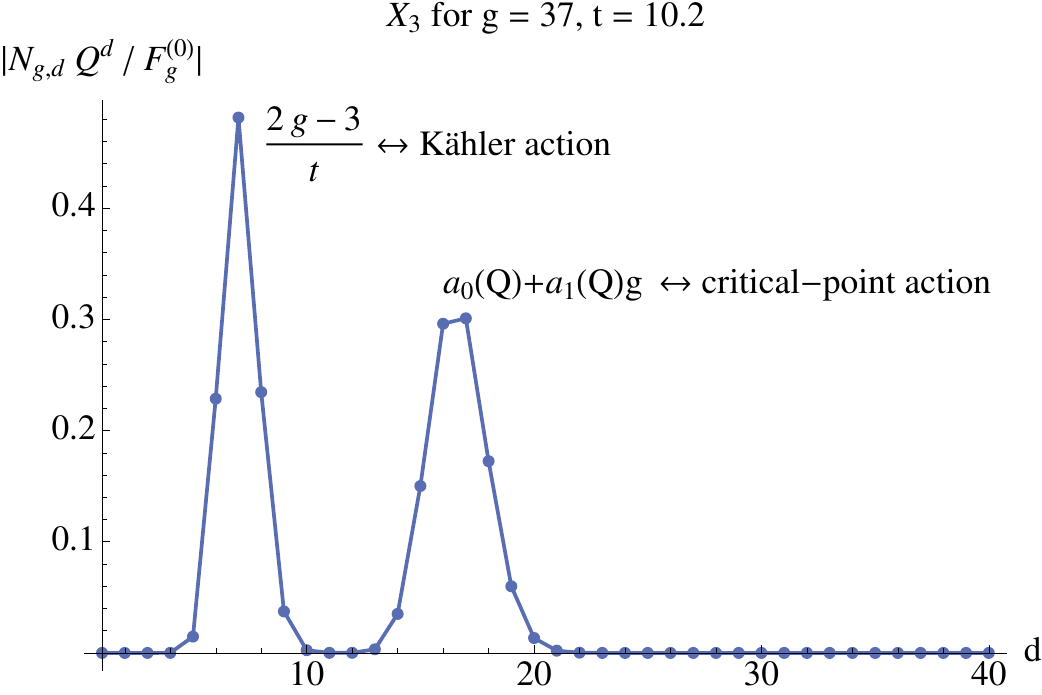}
\caption{Local curve: Graphical representation of which GW invariants contribute the most to a free energy $F^{(0)}_g(Q)$, for fixed values of $g$ and $Q=\rme^{-t}$, and with $p=3$. As for the earlier examples of local $\BP^2$ in figure~\ref{fig:leading_degrees} and ABJM in figure~\ref{fig:ABJM_leading_degrees}, also the local curve has saddle points corresponding to both K\"ahler and critical-point actions. The values of $g$ and $t$ in the plot were chosen as to clearly see both saddles in the same figure.}
\label{fig:localcurve_leading_degrees}
\end{figure}
%%%%%%%%%%%%%%%%%%%%%%%%%%%%%%%%%%%%%%%%%%%%%%%%%%%%%%%%%%%%%%%%%

The nonperturbative structure of the local-curve free-energy was addressed in \cite{cgmps06, m06, msw07}, where instantons associated to spectral-curve B-cycles were found to control the large-order behavior of the free energy. But, as we shall see, as one moves towards larger and larger values of $t$ this picture changes. Let us first address this question within the free energy itself, before translating to the GW invariants. For the remainder of this section we work with an approximated free energy
\be 
F_g^\tLC{p}(t) \approx (F^\tLC{p})^{\star}_g(t) := \sum_{d=1}^{d_{\text{max}}(g)} N_{g,d}^\tLC{p}\, \rme^{-d t},
\label{Fg_sum}
\ee
\noindent
where $d_{\text{max}}(g)$ is the highest degree for which we have computed $N_{g,d}^\tLC{p}$ (our data is represented in figure~\ref{fig:data}). We should stress that, with the leading contributions arising from near $d = \frac{2g-3}{t}$, we can always be sure that no significant contributions were left unaccounted for, and, in the end, the high accuracy of the large-order predictions will confirm that there are no issues with our approximation \eqref{Fg_sum}. What we find at large-order resembles the resolved conifold \eqref{largeo_coni}, in that the leading factorial growth is governed by a Gaussian-like action $A=2\pi t$ with a one-instanton sector that truncates at two-loops\footnote{This is also true for all higher instanton sectors.}. In addition, there is a tower of other contributions that amounts to replacing the K\"ahler modulus $t$ with $t_n = t + 2\pi\rmi\, n$, $n \in \mathbb{Z}_{\ne 0}$, or, equivalently, with shifted instanton actions $A_n = A + 4\pi^2\rmi\, n$. The result is 
\bea 
F_g^\tLC{p}(t) &\sim& \frac{\Gamma \left(2g-1\right)}{\pi \left(2\pi t\right)^{2g-1}} \left( t + \frac{t}{2g-2} \right) + \label{largeoF} \\
&&
+ \sum_{m=1}^{+\infty} \frac{\Gamma \left(2g-1\right)}{\pi \left|A_m\right|^{2g-1}} \left\{ 2 \left|t_m\right| \cos \left( \left(2g-2\right) \theta_m \right) + \frac{2 \left|t_m\right|}{2g-2}\, \cos \left( \left(2g-2\right) \theta_m \right) \right\}, \nonumber
\eea
\noindent
where we have also defined $\theta_m := \arg A_m = \arctan \frac{2\pi n}{t}$. It may be instructive to rewrite the above large-order relation in a more standard resurgence language, similar to \eqref{eq:largeorderF0gF10}. The difference is that now we have infinitely-many instanton actions $A_m$, where expansions around each sector truncate at two loops. In this way we rewrite \eqref{largeoF} as
\be
F_g^{X_p,(0)}(t) \sim \frac{S_1}{2\pi\rmi}\, \sum_{m \in \BZ} \frac{\Gamma \left(2g-1\right)}{A_m^{2g-1}} \left\{ F^{X_p,(1)}_{m,1}(t_m) + \frac{F^{X_p,(1)}_{m,2}(t_m)\, A_m}{2g-2} \right\} + \text{2-instanton corrections}.
\ee
\noindent
One immediately identifies the Stokes constant $S_1 = 2\rmi$, and the one-instanton, one- and two-loop coefficients
\be 
F^{X_p,(1)}_{m,1}(t_m) = t_m, \qquad F^{X_p,(1)}_{m,2}(t_m) = \frac{1}{2\pi}.
\ee

Tests of the large-order prediction \eqref{largeoF} are shown in figure~\ref{fig:resurgent_tower}, where it proves convenient to normalize the local-curve free energy against the Gaussian contribution, \textit{i.e.}, we use $\CF_g \equiv F_g - F^{\text{G}}_g$. The large-order growth of this normalized free-energy then corresponds to the second line in \eqref{largeoF}. In the left plot we show the normalized free energies for $t=100$ (multiplied by an appropriate factor to make all numbers of $\mathcal{O}(1)$; the gray dots), alongside the sum in \eqref{largeoF} up to $m=2$ (green), $m=4$ (yellow) and $m=6$ (red). One can clearly see that the agreement with the data gets better and better by including more terms in the tower of corrections in \eqref{largeoF}. In the right plot we further illustrate this by showing how small the error is for $t=24$. We use the normalized free-energy and define the error as $\left| 1 - \CF^{\text{pred},m}_g/\CF^{(0)}_g \right|$, where $\CF^{\text{pred},m}_g$ just corresponds to taking the second line of \eqref{largeoF} up to a given maximum (the colors are the same as the ones used on the left). We see that for $m=6$, at $g=100$, the error is of the order $10^{-58} \%$.

%%%%%%%%%%%%%%%%%%%%%%%%%%%%%%%%%%%%%%%%%%%%%%%%%%%%%%%%%%%%%%%%%
\begin{figure}[t!]
\begin{center}
\includegraphics[width=0.49\textwidth]{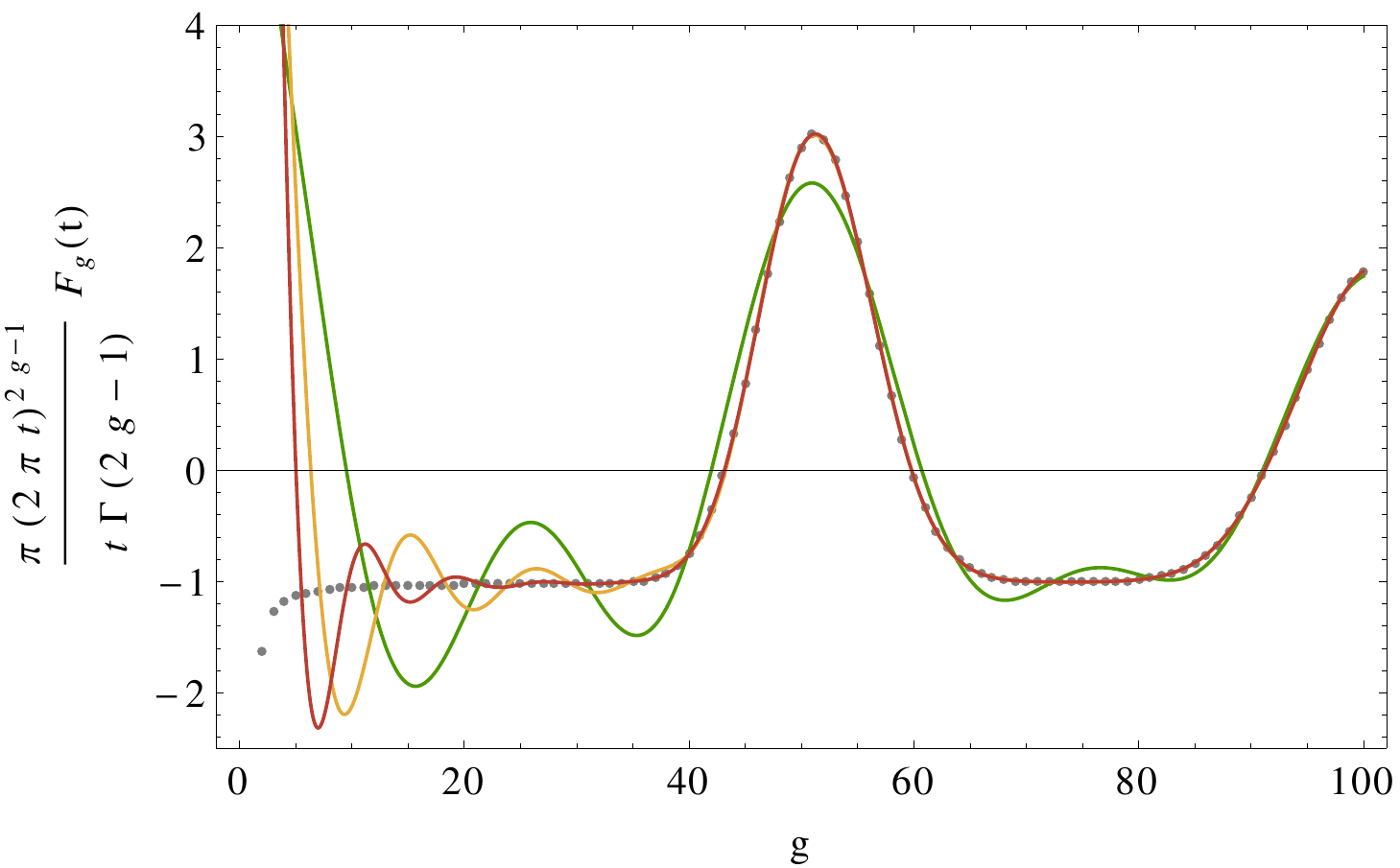}
\hspace{0.02\textwidth}
\includegraphics[width=0.47\textwidth]{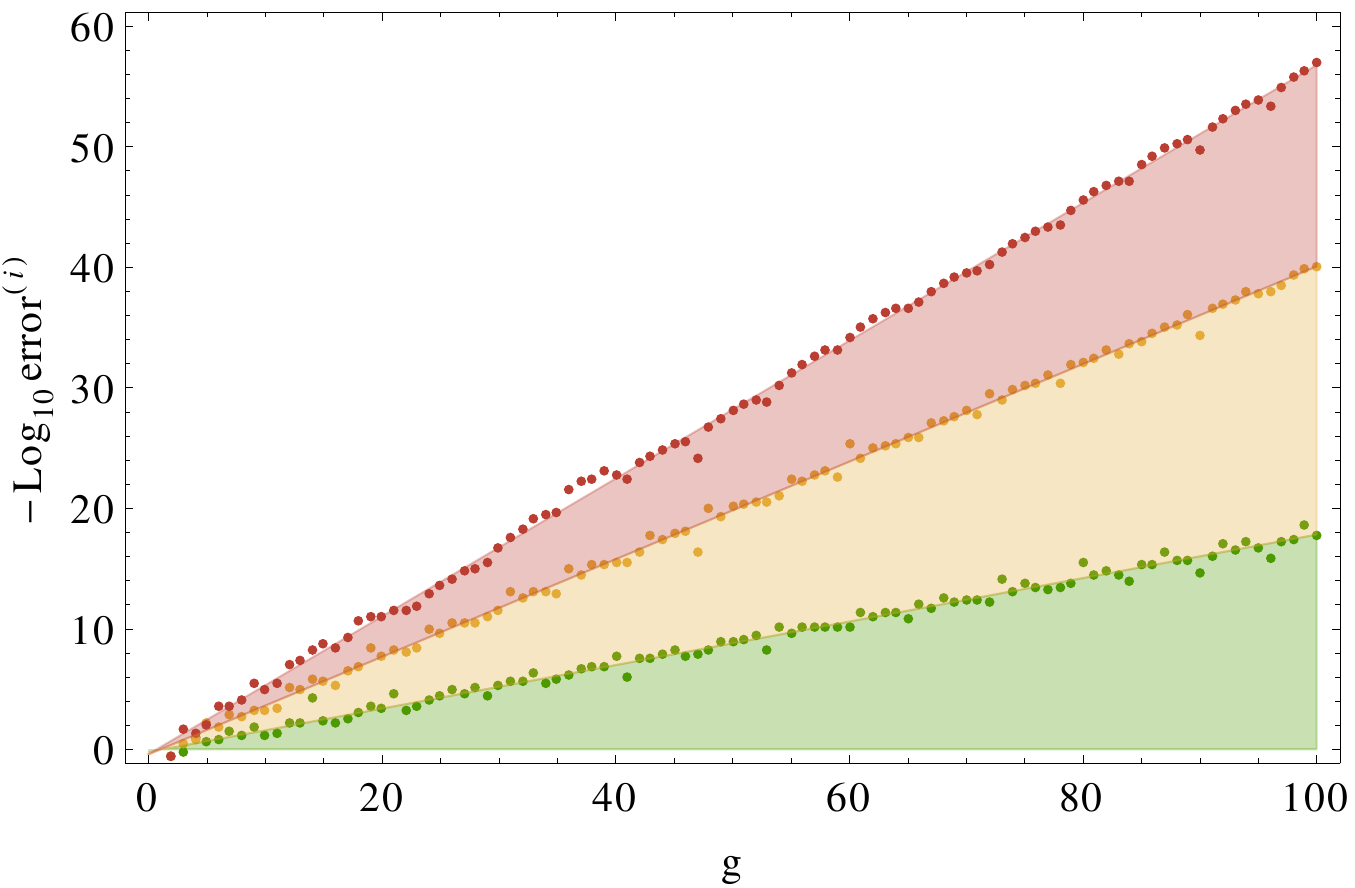} 
\end{center}
\vspace{-1\baselineskip}
\caption{Local curve: Left: Comparison between the $p=3$ free energy, with an appropriate pre-factor (the gray dots), and the predictions coming from \eqref{largeoF} up to $m=2$ (green), $m=4$ (yellow) and $m=6$ (red); showing a quicker convergence the more terms are included. Right: Logarithm of the error in the predictions (with the same colors), this time against the normalized free energy. For $m=6$, at genus $g=100$, the error is roughly of the order of one part in $10^{60}$.}
\label{fig:resurgent_tower}
\end{figure}
%%%%%%%%%%%%%%%%%%%%%%%%%%%%%%%%%%%%%%%%%%%%%%%%%%%%%%%%%%%%%%%%%

Let us comment on the relation to the large-order behavior found in \cite{m06, msw07}. There should be a value of the K\"ahler parameter for which there is an effective competition between the B-cycle action found in \cite{m06, msw07} and the Gaussian-like tower described above. Unfortunately, our data does not allow us to directly look at this interplay, for as one moves towards smaller $t$ the contribution to the free energies is no longer dominated just by the invariants close to $d = \frac{2g-3}{t}$. There will also be other relevant contributions at higher degree, which we do not have enough data to account for. Nonetheless, do notice that an exchange in large-order dominance should be precisely related to this emergence of relevant contributions beyond the large-radius ``peak''.

Given the above large-order behavior of the free energy, we may next deduce its consequences towards the ``diagonal'' growth of GW invariants. After accounting for the appropriate Gaussian correction, it turns out that the GW large-order is of the exact same type as in earlier examples (in this case, $n_0^{(1)}=(-1)^{p-1}$),
\be 
\left. N_{g,d}^\tLC{p}\, Q^d \right|_{g=\frac{t}{2}d+\sh} \sim \sum_{h=0}^{+\infty} \frac{\Gamma \left(2g-\frac{3}{2}-h\right)}{\left( 2\pi t \right)^{2g-\frac{3}{2}-h}}\, \frac{n_0^{(1)}\, t^{\frac{3}{2}-h}}{2^{2h+1}\, \pi^{h+2}}\, \polyname_h(\sh),
\ee
\noindent
where the $\polyname_h(\sh)$ are precisely the polynomials which were introduced in \eqref{eq:GW_coni_large_order}, and which also appeared in the similar large-order results for local $\mathbb{P}^2$ and local $\BP^1 \times \BP^1$. Computational tests on the validity of this expression are shown in figures~\ref{fig:ALL_Kahler_Action} and~\ref{fig:ALL_LOOPS}, with their details and discussion being the same as before. One thus finds that even for theories in different (critical) universality classes, there is some sort of \textit{universal large-order behavior} taking place at large radius\footnote{It would be interesting to compare this to the B-model large-radius universality recently uncovered in \cite{ayz15, c15}.}. This is also very clear in the plots in figures~\ref{fig:ALL_Kahler_Action} and~\ref{fig:ALL_LOOPS}. Furthermore, in the case of the local curve $X_p$, this large-order behavior turns out to be \textit{independent} of $p$ (up to a sign).

%%%%%%%%%%%%%%%%%%%%%%%%%%%%%%%%%%%%%%%%%%%%%%%%%%%%%%%%%%%%%%%%%
\subsection{The Example of Hurwitz Theory}\label{sec:Hur}
%%%%%%%%%%%%%%%%%%%%%%%%%%%%%%%%%%%%%%%%%%%%%%%%%%%%%%%%%%%%%%%%%

Let us now address a slightly more algebraic example, that of Hurwitz theory. Generically, it addresses branched covers of algebraic curves, but herein we restrict to so-called \textit{simple} Hurwitz numbers, denoted by $H_{g,d}^\tPOne\left(1^d\right)$, which count the number of degree-$d$ disconnected coverings of $\mathbb{P}^1$ by a genus-$g$ Riemann surface. These numbers have a combinatorial definition in terms of Young tableaux, but---in line with what we have been doing---they also have a string-theoretic origin. Indeed, Hurwitz theory may be thought of as a topological string theory, as it can be obtained by a particular limit of the A-model on the local curve $X_p$ \cite{cgmps06}. This limit consists in taking
\be 
p \rightarrow +\infty, \qquad t \rightarrow +\infty, \qquad g_{\text{s}} \rightarrow 0,
\ee
\noindent
while the combinations
\be 
g_\tH \equiv p\, g_{\text{s}}, \qquad \rme^{-t_\tH} \equiv (-1)^p\, p^2 \rme^{-t},
\ee
\noindent
are held fixed. A number of results can then be straightforwardly obtained by applying this limit to our results in the previous section.

%%%%%%%%%%%%%%%%%%%%%%%%%%%%%%%%%%%%%%%%%%%%%%%%%%%%%%%%%%%%%%%%%
\subsubsection*{Free Energies and Gromov--Witten Invariants}
%%%%%%%%%%%%%%%%%%%%%%%%%%%%%%%%%%%%%%%%%%%%%%%%%%%%%%%%%%%%%%%%%

The Hurwitz free-energy was shown to satisfy a Toda-like equation in \cite{p99},
\be 
\exp \Big\{ F^\tH (t_\tH-g_\tH) + F^\tH (t_\tH) + F^\tH (t_\tH+g_\tH) \Big\} = g_\tH^2\, \rme^{t_\tH}\, \partial_{t_\tH}^2 F^\tH (t_\tH).
\label{diffHur}
\ee
\noindent
However, regarding Hurwitz theory as a limiting local-curve in the sense explained above, implies one may compute the genus-$g$ free energy directly in the B-model as
\be 
F_g^{\tH} = \frac{1}{\left(1-\chi\right)^{5(g-1)}}\, \sum_{n=1}^{3g-3} a_{g,n}^\tH\, \chi^n.
\label{FgHur}
\ee
\noindent
Here, the new variable $\chi$ is related to the local-curve B-model modulus $w$ as
\be 
w - 1 = -\frac{\chi}{p^2},
\ee
\noindent
and in the Hurwitz limit the mirror map becomes
\be 
\rme^{-t_\tH} = \chi\, \rme^{-\chi}.
\label{mirror_Hur}
\ee
\noindent
In \eqref{FgHur} one needs to use the appropriate limit of the $a_{g,n}(p)$ coefficients from \eqref{FgBmodel}, defined as
\be 
a_{g,n}^\tH = \lim_{p \rightarrow +\infty} p^{8 \left(g-1\right)-2n} \left(-1\right)^n a_{g,n}.
\label{coeffsHur}
\ee
\noindent
From explicit results in appendix~\ref{sec:appendix_local} one can see that only some of the coefficients contribute in the limit. The coefficients $a_{g,n}^\tH$ also turn out to be related to the perturbative free energies of 2d gravity as, under the appropriate double-scaling limit, the difference equation \eqref{diffHur} reduces to the Painlev\'e I equation. A large-order analysis of the Hurwitz free energy was performed in \cite{msw07} finding that large-order effects were, as expected, governed by the $p \rightarrow +\infty$ limit of the B-cycle instanton-action that controlled the large-order effects of the local curve.

In the A-model formulation, GW invariants are defined as usual. Furthermore, without surprise, they may also be obtained from the limit
\be 
N_{g,d}^\tH = \frac{H_{g,d}^\tPOne\left(1^d\right)^{\bullet}}{\left(2g + 2d -2\right)!} = \lim_{p\to +\infty} p^{2-2g-2d}\, N_{g,d}^\tLC{p}.
\label{GW_Hur}
\ee
\noindent
Here we have written the GW invariants in terms of \textit{connected}, simple Hurwitz numbers $H_{g,d}^\tPOne\left(1^d\right)^{\bullet}$. In practice, we computed the $N_{g,d}^\tH$ along the same lines as we computed the $N_{g,d}^\tLC{p}$, \textit{i.e.}, starting from the partition function, then computing the free energy at fixed degree, and finally expanding in powers of $g_\tH$. We have computed \eqref{GW_Hur} up to the totals schematically shown in figure~\ref{fig:dataH} (also see appendix~\ref{sec:appendix_Hur_d}). For $g\leq 17$, data up to $d=50$ is enough to fix all the coefficients in \eqref{FgHur}, and we may thus compute GW invariants for any degree. This data will not be crucial in the following, since several results just follow from the ``finite $p$'' case addressed before. Nonetheless, the ability to generate more data also allows us to make predictions to higher orders.

%%%%%%%%%%%%%%%%%%%%%%%%%%%%%%%%%%%%%%%%%%%%%%%%%%%%%%%%%%%%%%%%%
\begin{figure}[t!]
\begin{center}
\raisebox{0.5cm}{\includegraphics[width=0.5\textwidth]{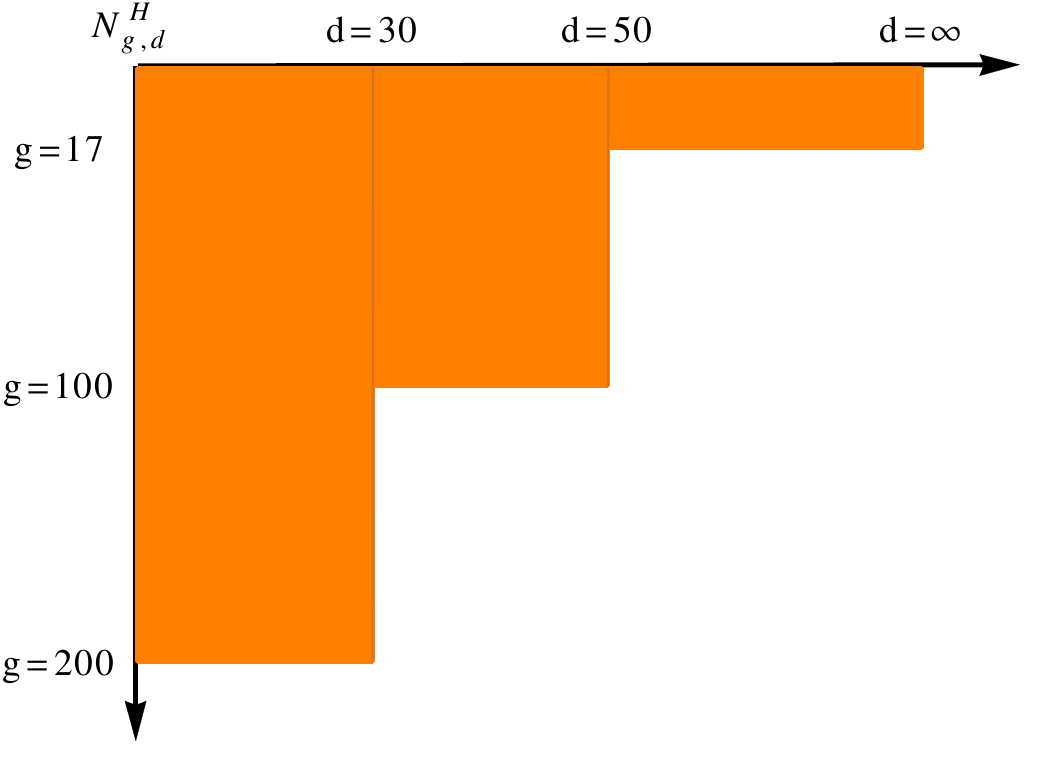}}
\end{center}
\vspace{-1\baselineskip}
\caption{Maximum degree and genus of the GW invariants we computed for Hurwitz theory. These are related to simple Hurwitz numbers via \eqref{GW_Hur}. For $g\leq17$ we have all the required data to fix \eqref{FgHur} and thus can compute GW invariants for any degree.}
\label{fig:dataH}
\end{figure}
%%%%%%%%%%%%%%%%%%%%%%%%%%%%%%%%%%%%%%%%%%%%%%%%%%%%%%%%%%%%%%%%%

%%%%%%%%%%%%%%%%%%%%%%%%%%%%%%%%%%%%%%%%%%%%%%%%%%%%%%%%%%%%%%%%%
\subsubsection*{Analysis of Large-Degree Growth}
%%%%%%%%%%%%%%%%%%%%%%%%%%%%%%%%%%%%%%%%%%%%%%%%%%%%%%%%%%%%%%%%%

The asymptotic\footnote{Reference \cite{op01} also studies asymptotics of Hurwitz numbers $H_{g,\mu}$, with $\mu$ a partition with $\ell$ parts $\mu_1, \ldots, \mu_\ell$. However, the asymptotics considered in \cite{op01} are in the limit $\lim_{N\to+\infty} H_{g,N\mu}$. This is conceptually different from our large-degree expansion of simple Hurwitz numbers: our case corresponds to a partition $(1,\ldots,1)$ with $d$ entries (the number which is growing), while in the results of \citep{op01} the length of the partition is always kept fixed.} growth of Hurwitz numbers at large degree $d$, with fixed genus, may be extracted from what we found earlier for the local curve. In particular, one can apply the limit \eqref{GW_Hur} directly to \eqref{GWlarged}, so that after introducing\footnote{In this limit we can equally take $f \sim p^2 \rightarrow +\infty$.}
\bea
\widetilde{c}^{(j)}_{g,k} &=& \frac{2^{\frac{1}{2} \left( k - 5\left(g-1\right) \right)}}{\Gamma \left(\frac{1}{2} \left(5\left(g-1\right)-k-j\right) \right)}\, \sum_{j_0=1}^{j} \widehat{c}_{j_0}^{\tH,(j)} \left( 5 \left(g-1\right)-k \right)^{j_0}, \quad \widehat{c}_{j_0}^{\tH,(j)} = \lim_{f \rightarrow +\infty} \frac{\widehat{c}_{j_0}^{(j)}(f)}{\left( f \left( f-1 \right) \right)^{\frac{j}{2}}} , \nonumber \\
\alpha^\tH_{g,k} &=& \lim_{f \rightarrow +\infty} f^{4 \left(g-1\right) - k}\, \alpha_{g,k} (f),
\eea
\noindent
we immediately arrive at
\be 
N_{\boldsymbol{g},d}^\tH \sim \rme^d\, d^{\frac{5}{2} \left(\boldsymbol{g}-1\right) - 1}\,  \sum_{j=0}^{+\infty}\, \sum_{j'=0}^{\text{Min} \left(j, \,3(\boldsymbol{g}-1)\right)} \widetilde{c}^{(j-j')}_{\boldsymbol{g},j'}\, \alpha^\tH_{\boldsymbol{g},j'}\, d^{-\frac{j}{2}}.
\label{Hur_larged}
\ee

Another route to this result would be to directly write GW invariants for Hurwitz theory, as we did in \eqref{GW_lagrange} for the local curve (one would now have to use the Hurwitz mirror map \eqref{mirror_Hur}). In this case, one obtains the GW invariants as\footnote{We use the $\alpha^\tH_{g,k}$ coefficients for convenience; they are related to the $a_{g,i}^\tH$ in \eqref{coeffsHur} via $(-1)^k \alpha^\tH_{g,k} = \sum_{i=1}^{3g-3} \binom{i}{k} a_{g,i}^\tH$. They could just as well be fixed by using the mirror map and the Toda equation \eqref{diffHur}.}
\be 
N_{\boldsymbol{g},d}^\tH = \frac{(-1)^{d}}{d}\, \sum_{k=0}^{3\boldsymbol{g}-3} \alpha^\tH_{\boldsymbol{g},k} \left( 5 \left(\boldsymbol{g}-1\right)-k \right) L_{d-1}^{k-d-5\left(\boldsymbol{g}-1\right)}(d),
\label{GW_lagrange_H}
\ee
\noindent
where the $L_{m}^{a}(z)$ are the associated Laguerre polynomials. Since the sum in \eqref{GW_lagrange_H} now runs over fewer values, it becomes easier to fix the necessary coefficients and generate GW invariants to arbitrarily large degree. Associated to the fact that the coefficients $\widetilde{c}^{(j)}_{g,k}$ no longer depend on an extra parameter, we can find the large-degree expansion \eqref{Hur_larged} to very high order with little effort. We present some of these results in appendix~\ref{sec:appendix_Hur_d}.

%%%%%%%%%%%%%%%%%%%%%%%%%%%%%%%%%%%%%%%%%%%%%%%%%%%%%%%%%%%%%%%%%
\subsubsection*{Analysis of Large-Genus Growth}
%%%%%%%%%%%%%%%%%%%%%%%%%%%%%%%%%%%%%%%%%%%%%%%%%%%%%%%%%%%%%%%%%

Hurwitz theory does not have an $abc$-formula, because the would-be GV invariants are no longer integers. Nonetheless, one can proceed empirically, using numerics and Richardson extrapolation, in order to find the growth of Hurwitz numbers for large genus, while at fixed degree. At low degree, the large-genus expansions actually \textit{truncate}. For instance, for $d=2,3,4$ we find
\bea 
H_{g,2}^\tPOne \left(1^d\right)^{\bullet} &=& \frac{1}{2}, \\
H_{g,3}^\tPOne \left(1^d\right)^{\bullet} &=& \frac{3^{2g-2}}{2}, \\
H_{g,4}^\tPOne \left(1^d\right)^{\bullet} &=& \frac{1}{2} \left( 2^{2g+2}-1 \right) \left(3^{2g+4}-1 \right).
\eea
\noindent
That is no longer the case for degree $d \geq 5$, where we now find
\bea 
H_{g,\boldsymbol{d}}^\tPOne \left(1^d\right)^{\bullet} &=& \frac{2}{\left(\boldsymbol{d}!\right)^2}\left(\frac{\boldsymbol{d} \left( \boldsymbol{d}-1 \right)}{2}\right)^{2\boldsymbol{d}+2g-2} - \frac{2}{\left( \left(\boldsymbol{d}-1\right)! \right)^2} \left( \frac{\left(\boldsymbol{d}-1\right) \left(\boldsymbol{d}-2\right)}{2} \right)^{2\boldsymbol{d}+2g-2} + 
\label{large_g_H} \\
&&
\hspace{-30pt}
+ \frac{2}{\boldsymbol{d}^2 \left( \left(\boldsymbol{d}-2\right)! \right)^2} \left( \frac{\boldsymbol{d} \left(\boldsymbol{d}-3\right)}{2} \right)^{2\boldsymbol{d}+2g-2} - \frac{1}{2 \left( \left(\boldsymbol{d}-2\right)! \right)^2} \left( \frac{\left(\boldsymbol{d}^2-5d+8\right)}{2} \right)^{2\boldsymbol{d}+2g-2} + \cdots. \nonumber
\eea
\noindent
These results can also be easily derived by computing the free energy directly for low degree. For the purpose of illustration, let us show how the agreement between the exact $H_{g,d}^\tPOne\left(1^d\right)^{\bullet}$ (for $d=6$ and $g=100$) and its prediction from \eqref{large_g_H} improves, as we include more terms. Note that this is an integer number with $241$ digits, but we only display the first $82$. One has:
\bea
H_{100,6}^\tPOne\left(1^d\right)^{\bullet} &= \mathsmaller{36773029021136586120108822348086934417891861531447353197011119061184878815704795302}\ldots \nonumber \\
\text{1-term} &= \mathsmaller{\textcolor{pert}{36773029021136586120108822348086934}556780750396609834115336352765962460171085765176}\ldots \nonumber \\
\text{2-terms} &= \mathsmaller{\textcolor{pert}{36773029021136586120108822348086934}\textcolor{1inst}{4178918615}07720945226447463877073571282196876287}\ldots \nonumber \\
\text{3-terms} &= \mathsmaller{\textcolor{pert}{36773029021136586120108822348086934}\textcolor{1inst}{4178918615}\textcolor{2inst}{3144735319701111906118}7444102489809778}\ldots \nonumber \\
\text{4-terms} &= \mathsmaller{\textcolor{pert}{36773029021136586120108822348086934}\textcolor{1inst}{4178918615}\textcolor{2inst}{3144735319701111906118}\textcolor{3inst}{4878815704795}232}\ldots \nonumber
\eea

%%%%%%%%%%%%%%%%%%%%%%%%%%%%%%%%%%%%%%%%%%%%%%%%%%%%%%%%%%%%%%%%%
\subsubsection*{Combined/Diagonal Large-Growth in Genus and Degree}
%%%%%%%%%%%%%%%%%%%%%%%%%%%%%%%%%%%%%%%%%%%%%%%%%%%%%%%%%%%%%%%%%

Uncovering the combined growth in genus and degree for Hurwitz theory is a harder problem than in previous examples. The main reason being that in this example the ``large-radius peak'', where GW invariants near $d=(2g-3)/t$ give the main contribution to the free energy, no longer seems to exist. We did still numerically find another ``critical-point peak'', as shown in figure~\ref{fig:hurwitz_leading_degrees}, but with the data we have available we were not able to find a linear relation, nor to uncover what the large-order behavior should be, analytically.

%%%%%%%%%%%%%%%%%%%%%%%%%%%%%%%%%%%%%%%%%%%%%%%%%%%%%%%%%%%%%%%%%
\begin{figure}[t!]
\centering
\includegraphics[width=0.46\textwidth]{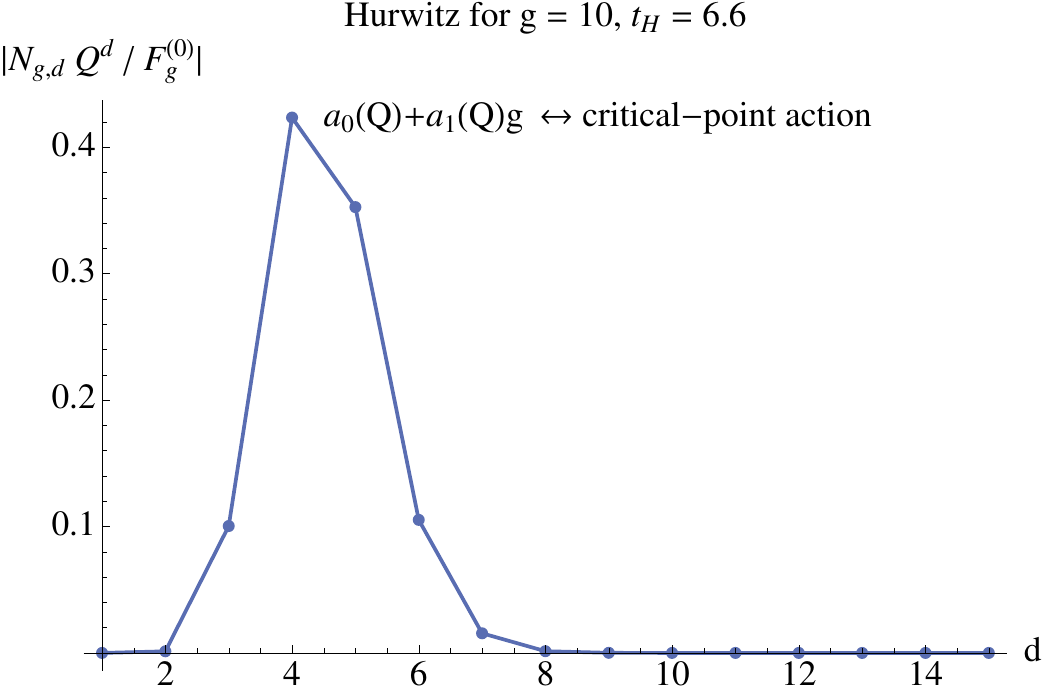}
\caption{Hurwitz: Graphical representation of which GW invariants contribute the most to a free energy $F^{(0)}_g(Q)$, for fixed values of $g$ and $Q=\rme^{-t_{\text{H}}}$. This time around we find a single saddle-point seemingly corresponding to the critical-point action.}
\label{fig:hurwitz_leading_degrees}
\end{figure}
%%%%%%%%%%%%%%%%%%%%%%%%%%%%%%%%%%%%%%%%%%%%%%%%%%%%%%%%%%%%%%%%%

%%%%%%%%%%%%%%%%%%%%%%%%%%%%%%%%%%%%%%%%%%%%%%%%%%%%%%%%%%%%%%%%%
\subsection{The Example of the Compact Quintic}
%%%%%%%%%%%%%%%%%%%%%%%%%%%%%%%%%%%%%%%%%%%%%%%%%%%%%%%%%%%%%%%%%

For our final example, we shall consider a compact geometry, in comparison to the non-compact local geometries we have been addressing up to now. This is actually the first example in which mirror symmetry was explicitly worked out and GW invariants systematically computed \cite{cdgp91}, the quintic CY threefold. The mirror of the quintic is described by the equation 
\be
\sum_{i=1}^5 x_i^5 - \frac{1}{z}\, \prod_{i=1}^5 x_i = 0,
\ee
\noindent
where $z$ captures the complex structure of the CY manifold. We will follow the notation in \cite{hkq06}.

%%%%%%%%%%%%%%%%%%%%%%%%%%%%%%%%%%%%%%%%%%%%%%%%%%%%%%%%%%%%%%%%%
\subsubsection*{Free Energies and Gromov--Witten Invariants}
%%%%%%%%%%%%%%%%%%%%%%%%%%%%%%%%%%%%%%%%%%%%%%%%%%%%%%%%%%%%%%%%%

As there is a single modulus, there is also a single Picard--Fuchs equation for the periods of this geometry, namely
\begin{equation}
\left\{ \left( z \partial_z \right)^4 - 5 z \left( 5 z \partial_z + 1 \right)\left( 5 z \partial_z + 2 \right)\left( 5 z \partial_z + 3 \right)\left( 5 z \partial_z + 4 \right) \right\} f(z) = 0.
\end{equation}
\noindent
From its solutions, we find the mirror map and the genus-zero free energy,
\begin{align}
-t &= \log z + 770 z + 717825 z^2 + \frac{3225308000}{3} z^3 + \cdots \equiv \log Q, \\
F^{(0)}_0 &= c_3 t^3 + c_2 t^2 + c_1 t + 2875 Q + \frac{4876875}{8} Q^2 + \frac{8564575000}{27} Q^3 + \cdots.
\end{align}
\noindent
Akin to what happened for the local geometries, here the higher-genus free energies can also be described in compact form in terms of a few generators (or propagators). Each of them has a holomorphic expansion around the large-radius point ($Q=0$), from which one can read the GW invariants. For example,
\begin{equation}
F^{(0)}_2 = \frac{575}{48} Q + \frac{5125}{2} Q^2 + \frac{7930375}{6} Q^3 + \cdots.
\end{equation}
\noindent
In this work we use the free energies which were computed in \cite{hkq06}, and which are available online\footnote{\texttt{http://uw.physics.wisc.edu/$\sim$strings/aklemm/highergenusdata/}}. In appendix~\ref{sec:appendix:quintic} we list a sample of the first GW invariants, and figure~\ref{fig:GWtotalQUINTIC} schematically represents the ones we used in our upcoming analysis. Note that we now have significantly less data than for the earlier non-compact examples, implying we will not have as many results.

%%%%%%%%%%%%%%%%%%%%%%%%%%%%%%%%%%%%%%%%%%%%%%%%%%%%%%%%%%%%%%%%%
\begin{figure}[t!]
\begin{center}
\raisebox{0.5cm}{\includegraphics[width=0.5\textwidth]{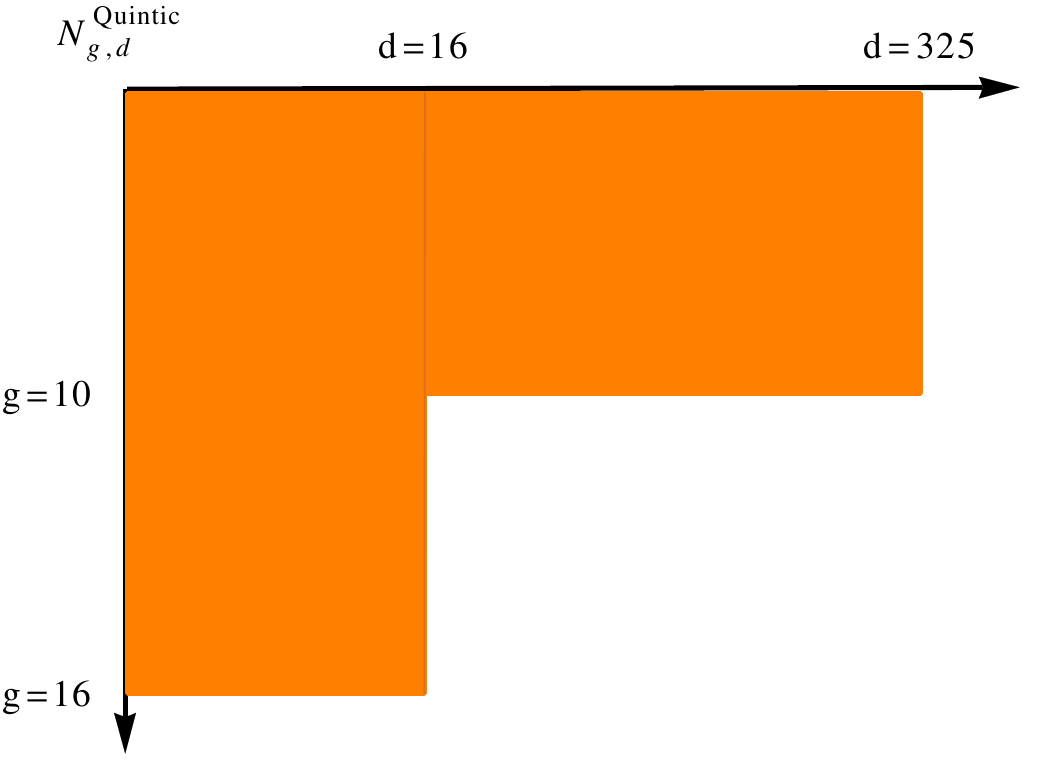}}
\end{center}
\vspace{-1\baselineskip}
\caption{Maximum degree and genus of the GW invariants computed for the quintic.}
\label{fig:GWtotalQUINTIC}
\end{figure}
%%%%%%%%%%%%%%%%%%%%%%%%%%%%%%%%%%%%%%%%%%%%%%%%%%%%%%%%%%%%%%%%%

Another important point is that there are essentially no studies of nonperturbative sectors for the quintic, mainly due to a lack of data to drive the analysis. However, it is natural to guess that there is an instanton action associated to the conifold point (located at $z = 5^{-5}$), $A_{\text{c}} = 2\pi T_{\text{c}}$, which is proportional to the flat coordinate $T_{\text{c}}$ and vanishing at the conifold point. A test of this instanton action is shown in figure~\ref{fig:quintic_coni_action_num}, precisely confirming that this is indeed the case. We should also expect a K\"ahler instanton action, but we do not have enough free energies available to report definite results.

%%%%%%%%%%%%%%%%%%%%%%%%%%%%%%%%%%%%%%%%%%%%%%%%%%%%%%%%%%%%%%%%%
\begin{figure}[t!]
\begin{center}
\includegraphics[width=0.5\textwidth]{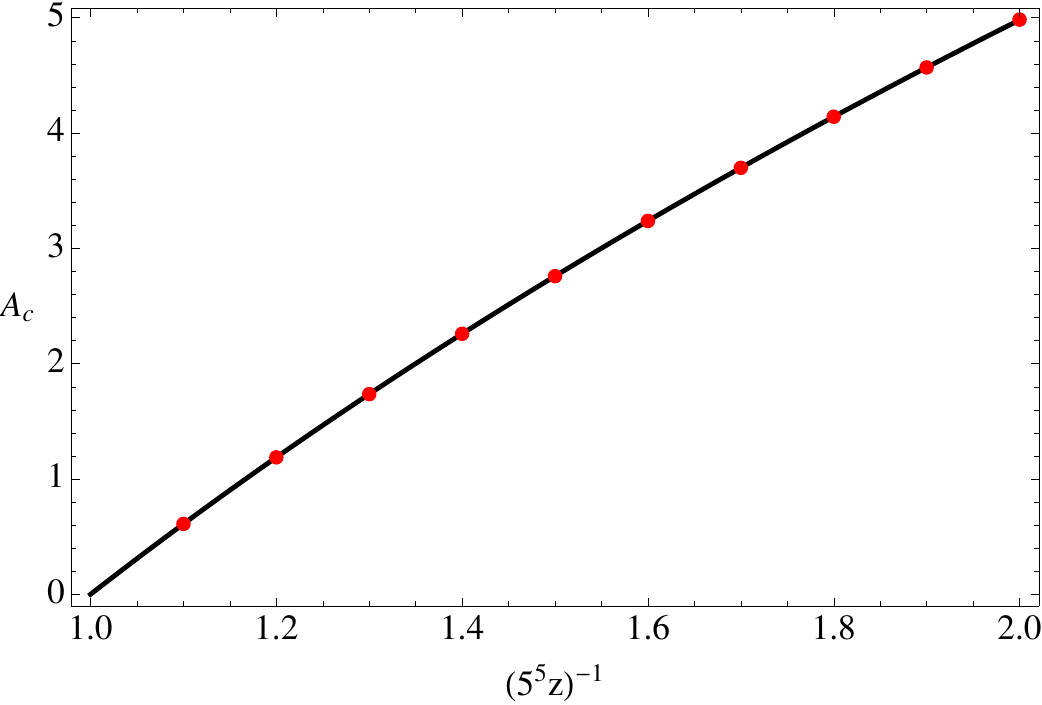}
\end{center}
\caption{Quintic: Near the conifold point, the instanton action $A_{\text{c}} = 2\pi T_{\text{c}}$ controls the factorial growth of the free energies. In the figure we plot the analytic dependence of $A_{\text{c}}$ against its numerical values computed from the large-genus growth of $F^{(0)}_g$, with a very clean match.}
\label{fig:quintic_coni_action_num}
\end{figure}
%%%%%%%%%%%%%%%%%%%%%%%%%%%%%%%%%%%%%%%%%%%%%%%%%%%%%%%%%%%%%%%%%

%%%%%%%%%%%%%%%%%%%%%%%%%%%%%%%%%%%%%%%%%%%%%%%%%%%%%%%%%%%%%%%%%
\subsubsection*{Analysis of Large-Degree Growth}
%%%%%%%%%%%%%%%%%%%%%%%%%%%%%%%%%%%%%%%%%%%%%%%%%%%%%%%%%%%%%%%%%

The large-degree growth at fixed genus was already considered in \cite{bcov93}, being in the same universality class as local $\BP^2$ or local $\BP^1 \times \BP^1$. In this case, the familiar asymptotic formula holds, 
\begin{equation}
N^{\tQuintic}_{\boldsymbol{g},d} \sim c_{\boldsymbol{g}}\, d^{2\boldsymbol{g}-3}\, \rme^{d t_{\text{c}}} \left( \log d \right)^{2\boldsymbol{g}-2},
\end{equation}
\noindent
where now $t_{\text{c}} := t(z = 5^{-5}) = 7.58995\ldots$. One can numerically check the value of this critical exponent, $t_{\text{c}}$, as well as the powers in  $d^{2\boldsymbol{g}-3}$ and $\left( \log d \right)^{2\boldsymbol{g}-2}$, using the same asymptotic techniques described for local $\BP^2$ and ABJM. We show these numerical results in figures~\ref{fig:comp_geom_5_large_degree_tc} and \ref{fig:comp_geom_5_large_degree_2g-3_and_2g-2}.

%%%%%%%%%%%%%%%%%%%%%%%%%%%%%%%%%%%%%%%%%%%%%%%%%%%%%%%%%%%%%%%%%
\begin{figure}[t!]
\centering
\includegraphics[width=0.6\textwidth]{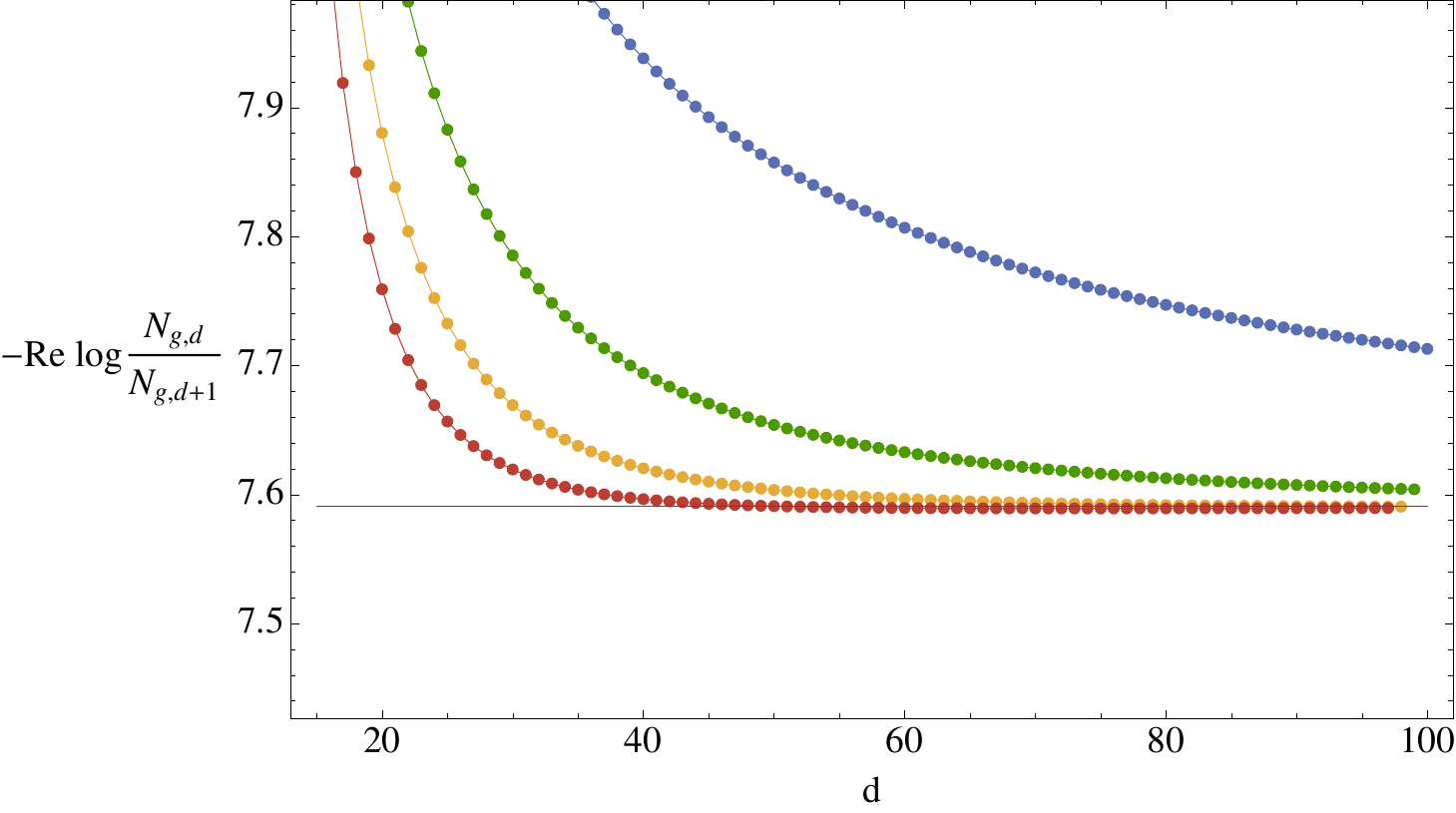}
\caption{Quintic: The exponent $t_{\text{c}}$ in the growth of $N_{\boldsymbol{g},d}$ is captured from the ratio of two consecutive GW invariants, when the degree is large. We plot that ratio alongside three Richardson extrapolations, which are clearly converging faster towards the expected result (up to a numerical relative error of about $0.2\%$).}
\label{fig:comp_geom_5_large_degree_tc}
\end{figure}
%%%%%%%%%%%%%%%%%%%%%%%%%%%%%%%%%%%%%%%%%%%%%%%%%%%%%%%%%%%%%%%%%

%%%%%%%%%%%%%%%%%%%%%%%%%%%%%%%%%%%%%%%%%%%%%%%%%%%%%%%%%%%%%%%%%
\begin{figure}[t!]
\centering
\includegraphics[width=0.46\textwidth]{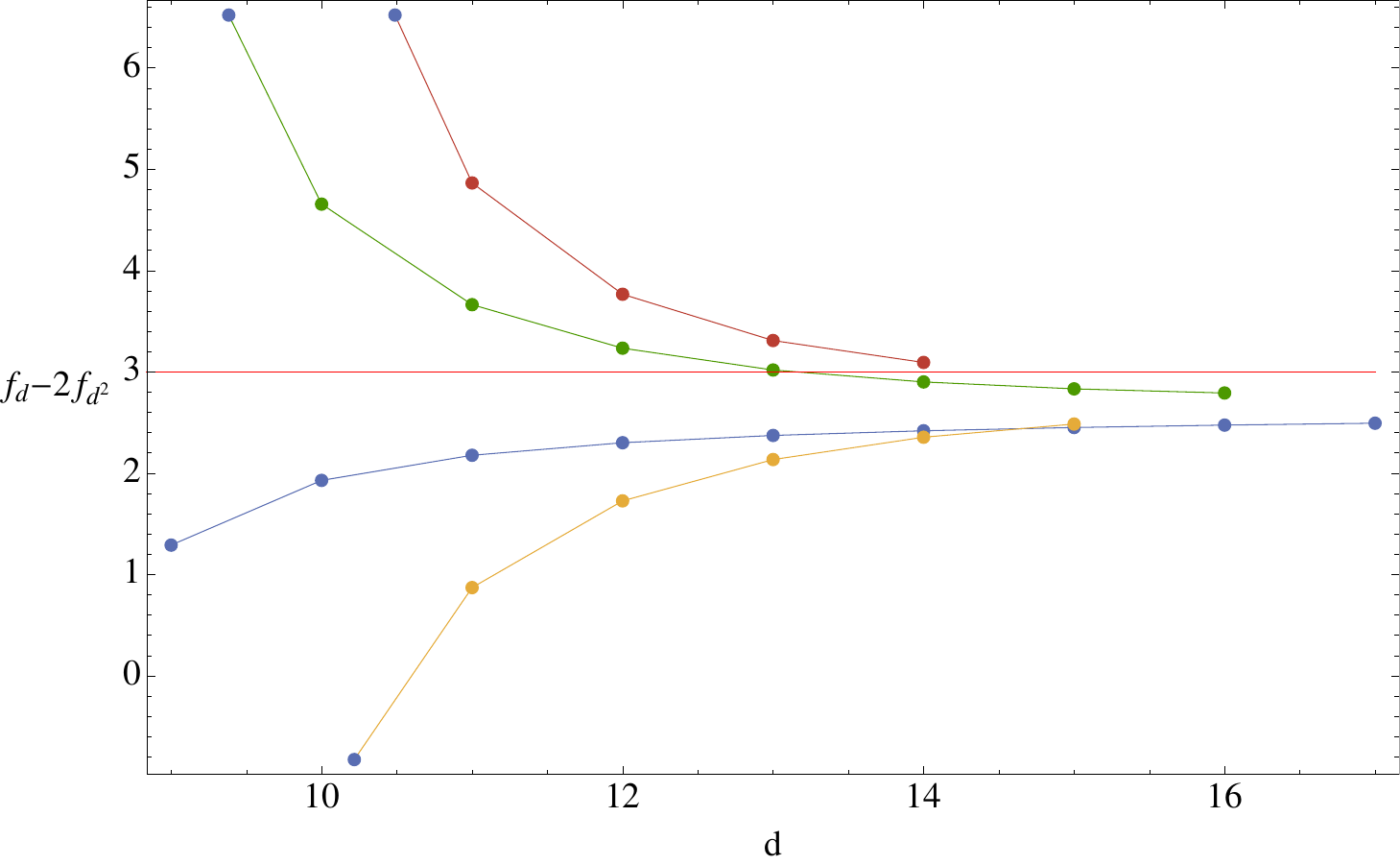}
\hspace{0.05\textwidth}
\includegraphics[width=0.46\textwidth]{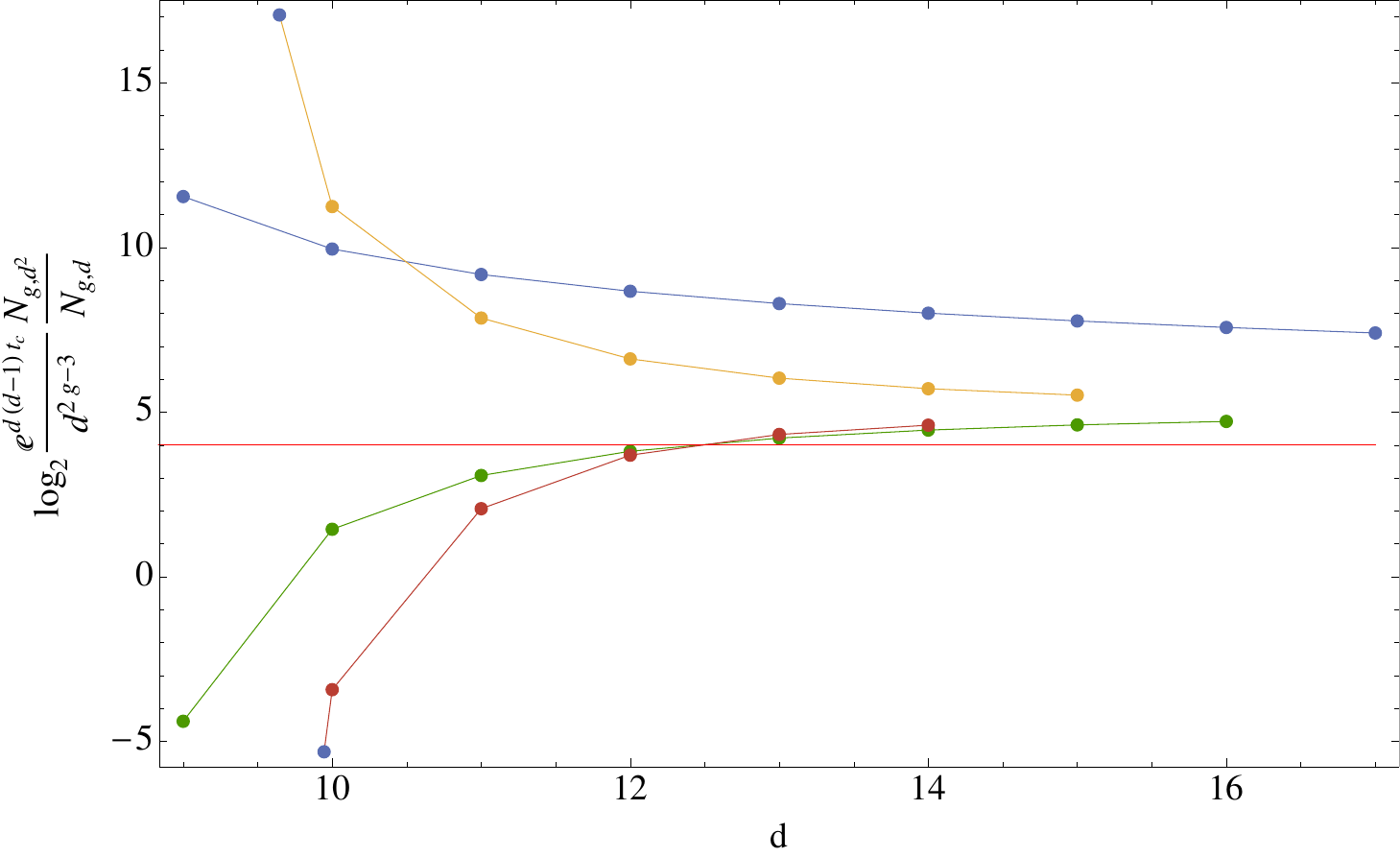}
\caption{Quintic: On the left we address the exponent $2\boldsymbol{g}-3$, which is the leading large-order term in $f_d - 2 f_{d^2}$. We have data up to $d=325$ so that the horizontal axis can only reach $d = 17$. The plot illustrates the first few Richardson transforms for $g=3$, converging faster towards the expected result (up to a numerical relative error of about $3\%$). On the right we address the exponent $2\boldsymbol{g}-2$ of the logarithm $\log d$, which is the leading term in the sequence \eqref{eq:localP2_gammacombination}. We plot the first few Richardson transforms for  $g=3$, converging faster towards the expected result (up to a numerical relative error of about $15\%$).}
\label{fig:comp_geom_5_large_degree_2g-3_and_2g-2}
\end{figure}
%%%%%%%%%%%%%%%%%%%%%%%%%%%%%%%%%%%%%%%%%%%%%%%%%%%%%%%%%%%%%%%%%

%%%%%%%%%%%%%%%%%%%%%%%%%%%%%%%%%%%%%%%%%%%%%%%%%%%%%%%%%%%%%%%%%
\subsubsection*{Analysis of Large-Genus Growth}
%%%%%%%%%%%%%%%%%%%%%%%%%%%%%%%%%%%%%%%%%%%%%%%%%%%%%%%%%%%%%%%%%

The fixed-degree, large-genus expansion of the GW invariants determines the $abc$-coefficients for the quintic in the same way it did for other geometries. In appendix \ref{sec:appendix:quintic} we present a sample of other such coefficients. Unfortunately, the scarce data available limits the analysis of the $b$-coefficients, and we have no other large-genus results to present for this example.

%%%%%%%%%%%%%%%%%%%%%%%%%%%%%%%%%%%%%%%%%%%%%%%%%%%%%%%%%%%%%%%%%
\subsubsection*{Combined/Diagonal Large-Growth in Genus and Degree}
%%%%%%%%%%%%%%%%%%%%%%%%%%%%%%%%%%%%%%%%%%%%%%%%%%%%%%%%%%%%%%%%%

As in previous examples, also for the quintic we find two saddle-points which are illustrated in figure~\ref{fig:quintic_leading_degrees}. They are associated to the K\"ahler instanton action, located at $d= (2g-3)/t$, and to the conifold instanton action, located at $d = a_0 (Q) + a_1 (Q)\, g$, just like before. 

%%%%%%%%%%%%%%%%%%%%%%%%%%%%%%%%%%%%%%%%%%%%%%%%%%%%%%%%%%%%%%%%%
\begin{figure}[t!]
\begin{center}
\includegraphics[width=0.46\textwidth]{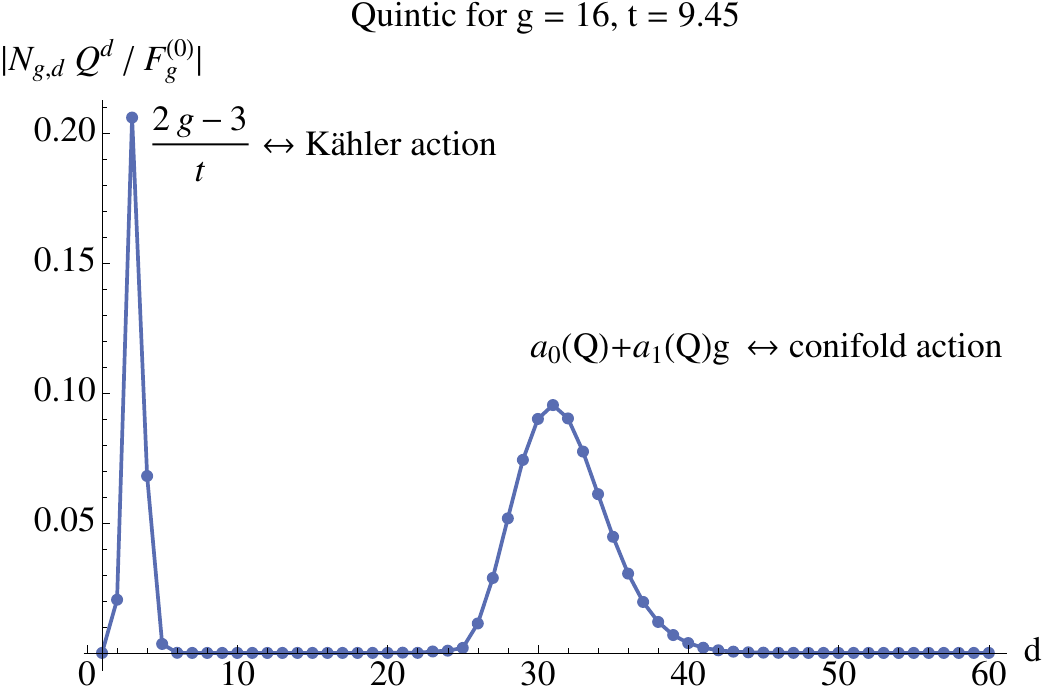}
\end{center}
\caption{Quintic: Graphical representation of which GW invariants contribute the most to a free energy $F^{(0)}_g (Q)$, for fixed values of $g$ and $Q=\rme^{-t}$. The values of $g$ and $t$ are carefully chosen so that both saddles are clearly visible in the same plot.}
\label{fig:quintic_leading_degrees}
\end{figure}
%%%%%%%%%%%%%%%%%%%%%%%%%%%%%%%%%%%%%%%%%%%%%%%%%%%%%%%%%%%%%%%%%

%%%%%%%%%%%%%%%%%%%%%%%%%%%%%%%%%%%%%%%%%%%%%%%%%%%%%%%%%%%%%%%%%
\subsubsection*{K\"ahler Leading Degree}
%%%%%%%%%%%%%%%%%%%%%%%%%%%%%%%%%%%%%%%%%%%%%%%%%%%%%%%%%%%%%%%%%

The scarce amount of available data does not let us check that the leading degree is $d= (2g-3)/t$ in the same way as we did for local $\BP^2$ or ABJM. Nevertheless, we can assume that this indeed holds, and then explore the asymptotics of $\left. N^{\tQuintic}_{g,d}\, Q^d \right|_{g= \frac{t}{2}d+\sh}$, just as we did in \eqref{eq:GW_localP2_large_order_Kahler}. We find that the same formula applies, but with the appropriate GV invariant $n_0^{(1)} = 2875$. Computational tests on the validity of such expression are shown in figures~\ref{fig:ALL_Kahler_Action} and~\ref{fig:ALL_LOOPS}, with their details and discussion being the exact same as before.

%%%%%%%%%%%%%%%%%%%%%%%%%%%%%%%%%%%%%%%%%%%%%%%%%%%%%%%%%%%%%%%%%
\subsubsection*{Conifold Leading Degree}
%%%%%%%%%%%%%%%%%%%%%%%%%%%%%%%%%%%%%%%%%%%%%%%%%%%%%%%%%%%%%%%%%

For the conifold leading degree we can do better. Figure~\ref{fig:comp_geom_5_con_1overa0a1} shows the dependence of $a_0$ and $a_1$ upon the modulus $t$. The inverse of the slope, $a_1^{-1}$, resembles a straight line
\begin{equation}
a_1 (Q)^{-1} = \left(-2.46 \pm 0.01\right) + \left(0.328 \pm 0.002\right) t, \qquad r^2 = 0.9995.
\end{equation}

%%%%%%%%%%%%%%%%%%%%%%%%%%%%%%%%%%%%%%%%%%%%%%%%%%%%%%%%%%%%%%%%%
\begin{figure}[t!]
\centering
\includegraphics[width=0.45\textwidth]{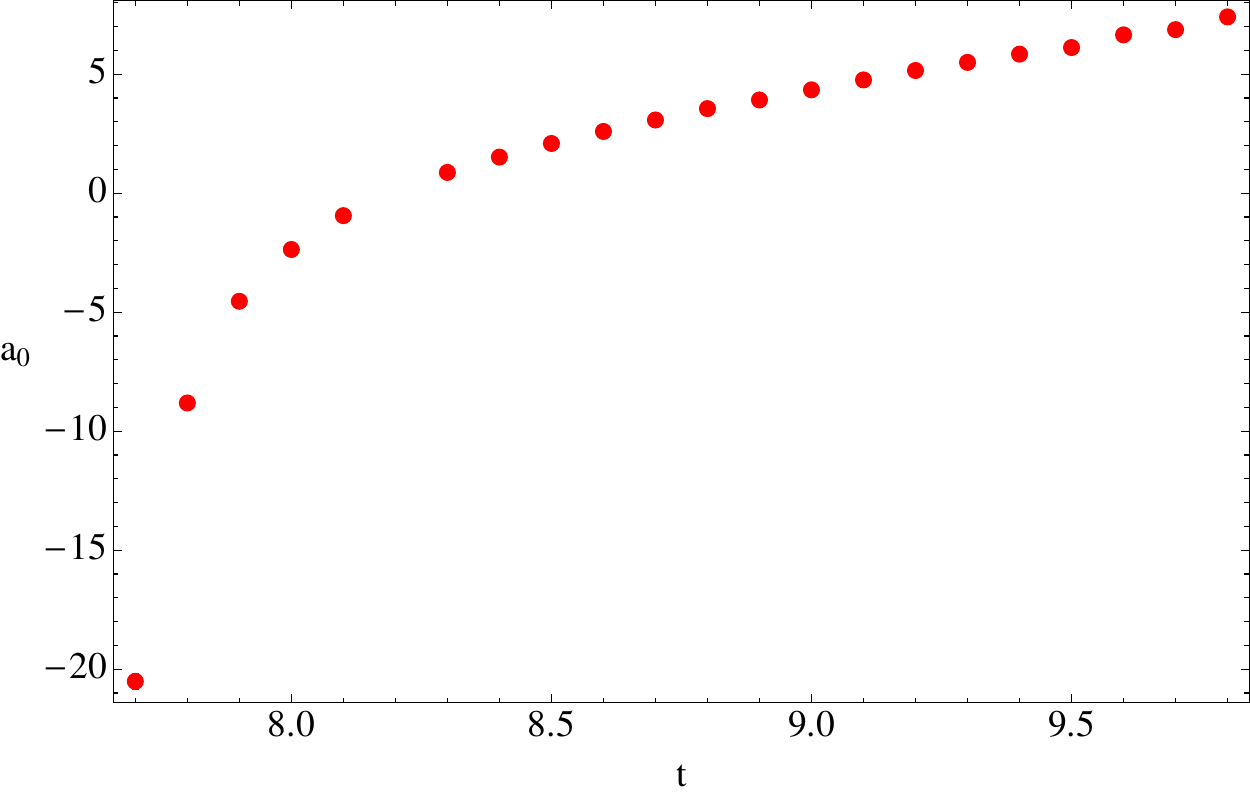}
\hspace{0.05\textwidth}
\includegraphics[width=0.45\textwidth]{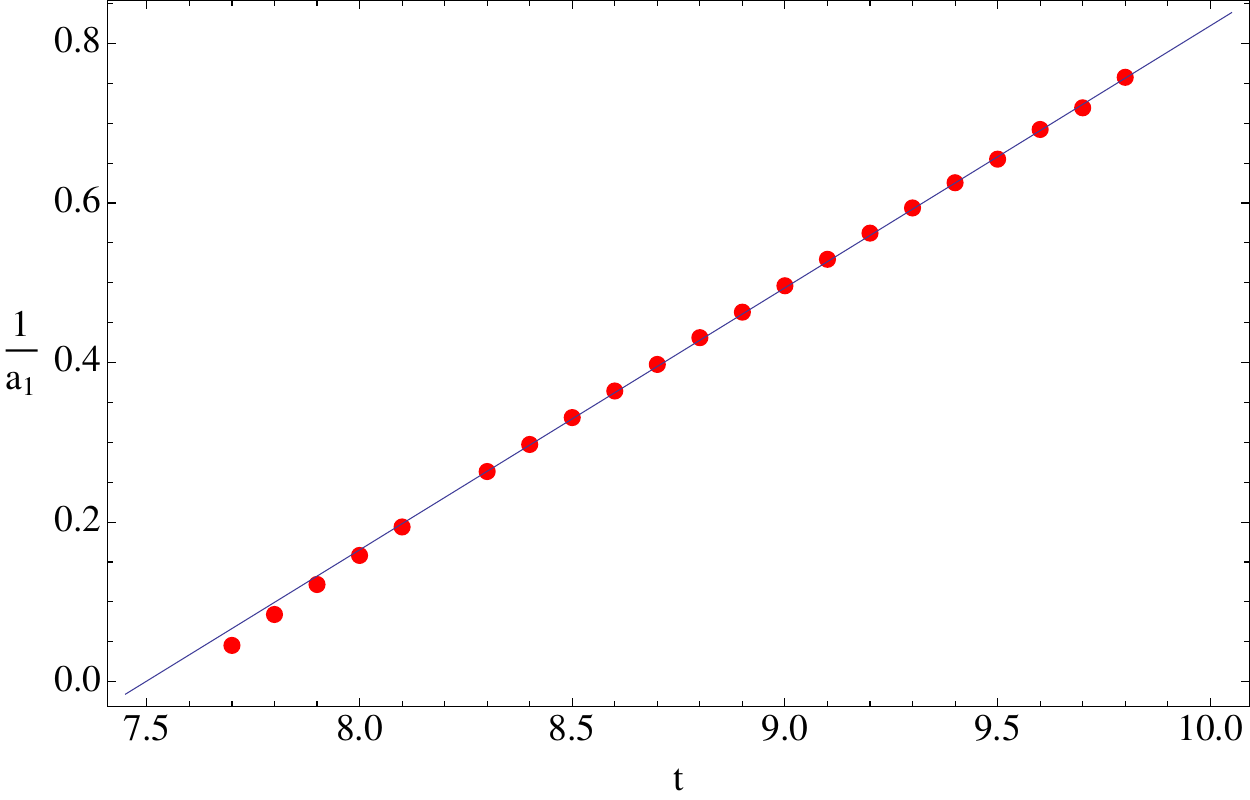}
\caption{Quintic: Numerical calculation of $a_0(Q)$ and $a_1(Q)$ associated to the conifold instanton action. We show the inverse of $a_1$ because the dependence seems linear, although we are not able to confirm this analytically. The plot for $a_0$ does not seem to lead to any linear dependence.}
\label{fig:comp_geom_5_con_1overa0a1}
\end{figure}
%%%%%%%%%%%%%%%%%%%%%%%%%%%%%%%%%%%%%%%%%%%%%%%%%%%%%%%%%%%%%%%%%

%%%%%%%%%%%%%%%%%%%%%%%%%%%%%%%%%%%%%%%%%%%%%%%%%%%%%%%%%%%%%%%%%
\acknowledgments
We would like to thank Aleksey Zinger for discussions, and Marcos Mari\~no, Jos\'e Mour\~ao and Jo\~ao Pimentel Nunes for useful comments on the draft. RS would further like to thank the University of Geneva for extended hospitality, where a part of this work was conducted. RV would further like to thank the DESY Theory Group for support and hospitality. This research was partially supported by the FCT-Portugal grants EXCL/MAT-GEO/0222/2012, UID/MAT/04459/2013 and PTDC/MAT-GEO/3319/2014. The research of RS was partially supported by the Swiss-NSF grant NCCR 51NF40-141869 ``The Mathematics of Physics'' (SwissMAP).
%%%%%%%%%%%%%%%%%%%%%%%%%%%%%%%%%%%%%%%%%%%%%%%%%%%%%%%%%%%%%%%%%

\newpage

%%%%%%%%%%%%%%%%%%%%%%%%%%%%%%%%%%%%%%%%%%%%%%%%%%%%%%%%%%%%%%%%%
%%%%%%%%%%%%%%%%%%%%%%%%%%%%%%%%%%%%%%%%%%%%%%%%%%%%%%%%%%%%%%%%%
\appendix
%%%%%%%%%%%%%%%%%%%%%%%%%%%%%%%%%%%%%%%%%%%%%%%%%%%%%%%%%%%%%%%%%
%%%%%%%%%%%%%%%%%%%%%%%%%%%%%%%%%%%%%%%%%%%%%%%%%%%%%%%%%%%%%%%%%

%%%%%%%%%%%%%%%%%%%%%%%%%%%%%%%%%%%%%%%%%%%%%%%%%%%%%%%%%%%%%%%%%
%%%%%%%%%%%%%%%%%%%%%%%%%%%%%%%%%%%%%%%%%%%%%%%%%%%%%%%%%%%%%%%%%
\section{Analysis of the $abc$-Coefficients}\label{sec:abc_coefficients}
%%%%%%%%%%%%%%%%%%%%%%%%%%%%%%%%%%%%%%%%%%%%%%%%%%%%%%%%%%%%%%%%%
%%%%%%%%%%%%%%%%%%%%%%%%%%%%%%%%%%%%%%%%%%%%%%%%%%%%%%%%%%%%%%%%%

When the GW invariants, $N_{g,d}$, may be written in terms of GV invariants, $n^{(d)}_g$, then there is a third representation in terms of some other integers we denoted by $a_d$, $b_{d,n}$, $c_d$. They appear naturally when considering the large-genus expansion of $N_{g,d}$. 

\begin{proposition}
The relation between GW and GV invariants, and $abc$-coefficients is
\begin{align}
\label{eq:app:GW_abc}
N_{g,d} &= f^\tconi_g \left\{ \sum_{m|d} a_m \left( \frac{d}{m} \right)^{2g-3} + \frac{2g}{B_{2g}}\, \frac{1}{d} \left( c_d\, \delta_{g,1} + \sum_{m=1}^{G(d)-1} b_{d,m}\, m^{2g-2} \right) \right\}, \\
a_d &= n^{(d)}_0, \\ 
b_{d,m} &= \sum_{k|d,m} (-1)^{\frac{m}{k}}\, \frac{2d}{k} \sum_{h=\frac{m}{k}+1}^{G(\frac{d}{k})} n^{(\frac{d}{k})}_h \binom{2h-2}{h-1+\frac{m}{k}}, \\
c_d &= \sum_{m|d} m \left\{ n_1^{(m)} + 2 \sum_{h=2}^{G(m)} \frac{n_{h}^{(m)}}{h}\, \binom{2h-3}{h-2} \right\}.
\end{align}
\noindent
Here $f^\tconi_0 = 1$, $f^\tconi_g = (-1)^{g-1}\, \frac{B_{2g}}{2g\left(2g-2\right)!}$ for $g\geq 1$, and $B_{2g}$ are the Bernoulli numbers. $G(d)$ satisfies $n^{(d)}_g = 0$ for $g > G(d)$. Since the GV invariants are integers, so are the $abc$-coefficients.
\end{proposition}

\begin{proof}
We start from the definition of GV invariants,
\begin{equation}
N_{g,d} = \sum_{h=0}^g c_{h,g} \sum_{m|d} n^{(m)}_h \left( \frac{d}{m} \right)^{2g-3} = \sum_{m|d} \sum_{h=0}^{G(m)} c_{h,g}\, n^{(m)}_h \left( \frac{d}{m} \right)^{2g-3},
\label{eq:app:GW_chg_GV}
\end{equation}
\noindent
where in the second equality we have noticed that $n^{(m)}_h = 0$ for $h>G(m)$. The coefficients $c_{h,g}$ generate $\left( 2 \sin \frac{x}{2} \right)^{2h-2} = \sum_{h=g}^{+\infty} c_{h,g}\, x^{2g-2}$, and they are explicitly given by
\begin{equation}
c_{0,g} = f^\tconi_g, \qquad c_{1,g}=\delta_{g,1}, \qquad c_{h,g} = (-1)^{g-1}\, \frac{2}{\left(2g-2\right)!} \sum_{k=1}^{h-1} \binom{2h-2}{h-1+k} (-1)^k\, k^{2g-2}.
\label{eq:app:chg_explicit}
\end{equation}
\noindent
Next, we organize the terms in \eqref{eq:app:GW_chg_GV} from more to less important as $g$ grows, using \eqref{eq:app:chg_explicit}. In order to do this, we split the $h$-sum in \eqref{eq:app:GW_chg_GV} into $h=0$, $h=1$, and $h\geq 2$. For $h\geq 2$ we can assume that also $g\geq 2$ and manipulate to arrive at \eqref{eq:app:GW_abc}. The steps are straightforward once one knows the goal, and they simply require the exchange of double sums, such as, \textit{e.g.}, $\sum_{h=2}^{G(m)} \sum_{k=1}^{h-1} = \sum_{k=1}^{G(m)-1} \sum_{h=k+1}^{G(m)}$; or relabelings, such as, \textit{e.g.}, $\sum_{m|d} f(m) = \sum_{m|d} f(d/m)$. The end result, after four of these manipulations, is
\begin{align}
N_{g,d} &= f^\tconi_g \left\{ \sum_{m|d} n^{(m)}_0 \left( \frac{d}{m} \right)^{2g-3} + \frac{2g}{B_{2g}}\, \frac{1}{d} \left(\delta_{g,1} \sum_{m|d} n^{(m)}_1\, m + \right. \right. \nonumber \\
& \qquad \qquad \qquad
\left. \left. + \delta_{g\geq 2} \sum_{m=1}^{G(d)-1} \sum_{k|d,m} \sum_{h=\frac{m}{k}+1}^{G(\frac{d}{k})} n^{(\frac{d}{k})}_h\, \frac{2d}{k}\, \binom{2h-2}{h-1+\frac{m}{k}}\, (-1)^{\frac{m}{k}}\, m^{2g-2} \right)  \right\},
\end{align}
\noindent
from where one immediately can read the $abc$-coefficients. 
\end{proof}

We finish this appendix with a short note on how the GV and GW invariants, and the $abc$-coefficients, may be laid out in Dirichlet series for each genus $g$, in contrast to the usual Taylor-series expansion in the form of free energies. If we define the generating functions
\begin{equation}
\GV_g (s) := \sum_{d=1}^{+\infty} \frac{n^{(d)}_g}{d^s}, \qquad \GW_g (s) := \sum_{d=1}^{+\infty} \frac{N_{g,d}}{d^s}, \qquad \TGW_g (s) := \frac{\GW_g(s)}{\zeta \left(s-(2g-3)\right)},
\end{equation}
\noindent
then it follows from \eqref{eq:app:GW_chg_GV} the linear transformations
\begin{equation}
\TGW_g (s) = \sum_{h=0}^g c_{h,g}\, \GV_h(s) \qquad \text{and} \qquad \GV_g (s) = \sum_{h=0}^g \alpha_{g,h}\, \TGW_h(s).
\end{equation}
\noindent
Here, the $\alpha$-coefficients arise from the generating function defined in \cite{bp00},
\be
\left( \frac{\arcsin(\sqrt{r}/2)}{\sqrt{r}/2} \right)^{2g-2} =: \sum_{h=0}^{+\infty} \alpha_{g+h,g}\, r^h.
\ee

We can also define Dirichlet series for the $abc$-coefficients, as
\begin{equation}
\MA(s) := \sum_{d=1}^{+\infty} \frac{a_d}{d^s}, \qquad \HB_{2g-2}(s) := \frac{(-1)^{g-1}}{\left(2g-2\right)!}\, \frac{1}{\zeta \left(s-(2g-3)\right)} \sum_{d=1}^{+\infty} \frac{\mb_{2g-2}(d)}{d^s},
\end{equation}
\noindent
where
\be
\mb_{2g-2} (d) := \frac{1}{d} \left( c_d\, \delta_{g,1} + \sum_{n=1}^{G(d)-1} b_{d,n}\, n^{2g-2} \right).
\ee
\noindent
They are linearly related to the previous Dirichlet series as
\begin{equation}
\TGW_g (s) = f^\tconi_g \MA(s) + \HB_{2g-2}(s), \qquad \GV_0 (s) = \MA(s), \qquad \GV_g (s) = \sum_{h=1}^g \alpha_{g,h}\, \HB_{2h-2}(s).
\end{equation}
\noindent
The last expression can be inverted to define $\HB_{2g-2} (s)= \sum_{h=1}^g c_{h,g}\, \GV_h(s)$.

\newpage
%\newgeometry{left=2cm,right=2cm,bottom=2cm,top=2cm}

\begin{landscape}

%%%%%%%%%%%%%%%%%%%%%%%%%%%%%%%%%%%%%%%%%%%%%%%%%%%%%%%%%%%%%%%%%
%%%%%%%%%%%%%%%%%%%%%%%%%%%%%%%%%%%%%%%%%%%%%%%%%%%%%%%%%%%%%%%%%
\section{Large-Order Enumerative Data}\label{sec:appendix}
%%%%%%%%%%%%%%%%%%%%%%%%%%%%%%%%%%%%%%%%%%%%%%%%%%%%%%%%%%%%%%%%%
%%%%%%%%%%%%%%%%%%%%%%%%%%%%%%%%%%%%%%%%%%%%%%%%%%%%%%%%%%%%%%%%%

\vspace{-0.05cm}

%%%%%%%%%%%%%%%%%%%%%%%%%%%%%%%%%%%%%%%%%%%%%%%%%%%%%%%%%%%%%%%%%
\subsection{Local $\BP^2$}\label{sec:appendix:localP2}
%%%%%%%%%%%%%%%%%%%%%%%%%%%%%%%%%%%%%%%%%%%%%%%%%%%%%%%%%%%%%%%%%

\vspace{-0.05cm}

%%%%%%%%%%%%%%%%%%%%%%%%%%%%%%%%%%%%%%%%%%%%%%%%%%%%%%%%%%%%%%%%%
\subsubsection*{GW Invariants}
%%%%%%%%%%%%%%%%%%%%%%%%%%%%%%%%%%%%%%%%%%%%%%%%%%%%%%%%%%%%%%%%%

\hspace{-2cm}
\begin{minipage}{\textwidth}
\centering
\includegraphics[]{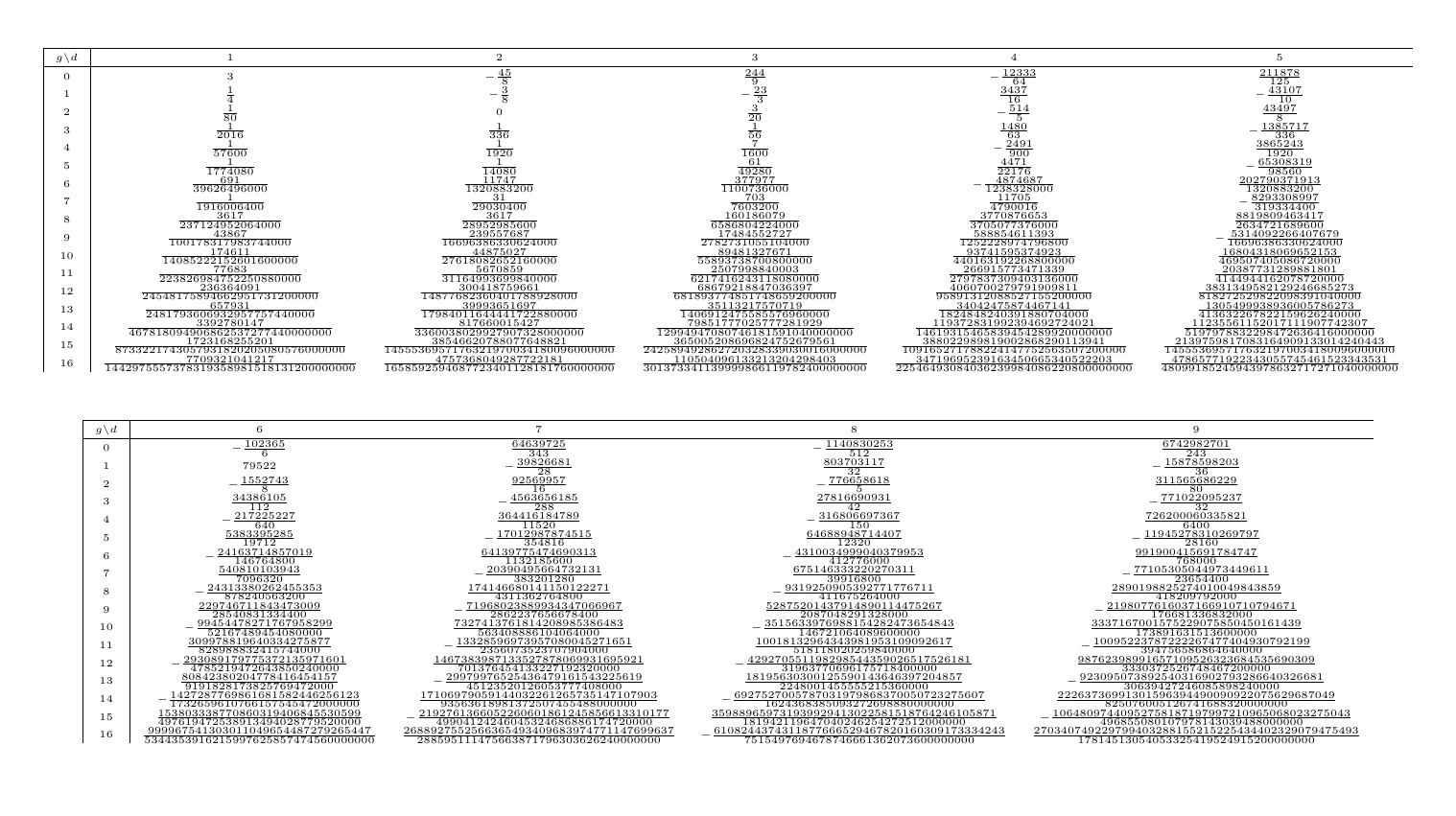}
\end{minipage}

%%%%%%%%%%%%%%%%%%%%%%%%%%%%%%%%%%%%%%%%%%%%%%%%%%%%%%%%%%%%%%%%%
\subsubsection*{$abc$-Coefficients}
%%%%%%%%%%%%%%%%%%%%%%%%%%%%%%%%%%%%%%%%%%%%%%%%%%%%%%%%%%%%%%%%%

\begin{minipage}{\textwidth}
\hspace{-2cm}
\begin{tabular}{cc}
\raisebox{5cm}{\includegraphics[]{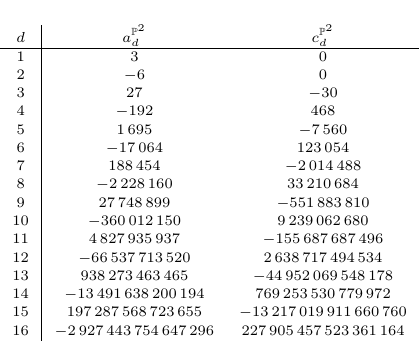}} \quad \quad
&
\includegraphics[]{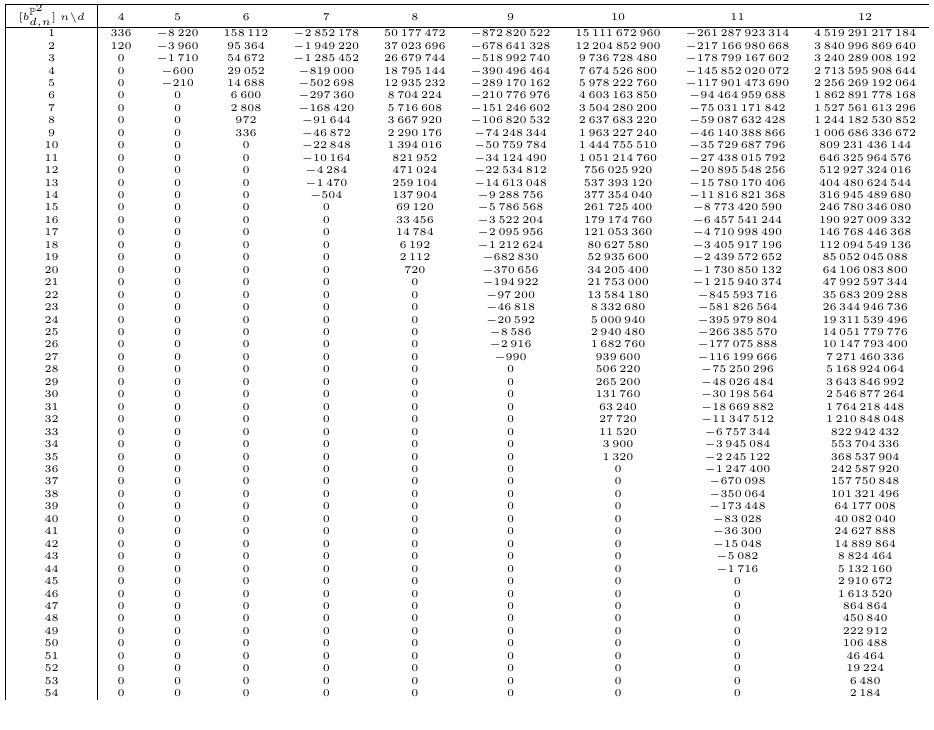}      
\end{tabular}
\end{minipage}

%%%%%%%%%%%%%%%%%%%%%%%%%%%%%%%%%%%%%%%%%%%%%%%%%%%%%%%%%%%%%%%%%
\subsection{Local $\BP^1\times\BP^1$}\label{sec:appendix:localP1xP1}
%%%%%%%%%%%%%%%%%%%%%%%%%%%%%%%%%%%%%%%%%%%%%%%%%%%%%%%%%%%%%%%%%

%%%%%%%%%%%%%%%%%%%%%%%%%%%%%%%%%%%%%%%%%%%%%%%%%%%%%%%%%%%%%%%%%
\subsubsection*{GW Invariants}
%%%%%%%%%%%%%%%%%%%%%%%%%%%%%%%%%%%%%%%%%%%%%%%%%%%%%%%%%%%%%%%%%

\hspace{-2cm}
\begin{minipage}{\textwidth}
\centering
\includegraphics[]{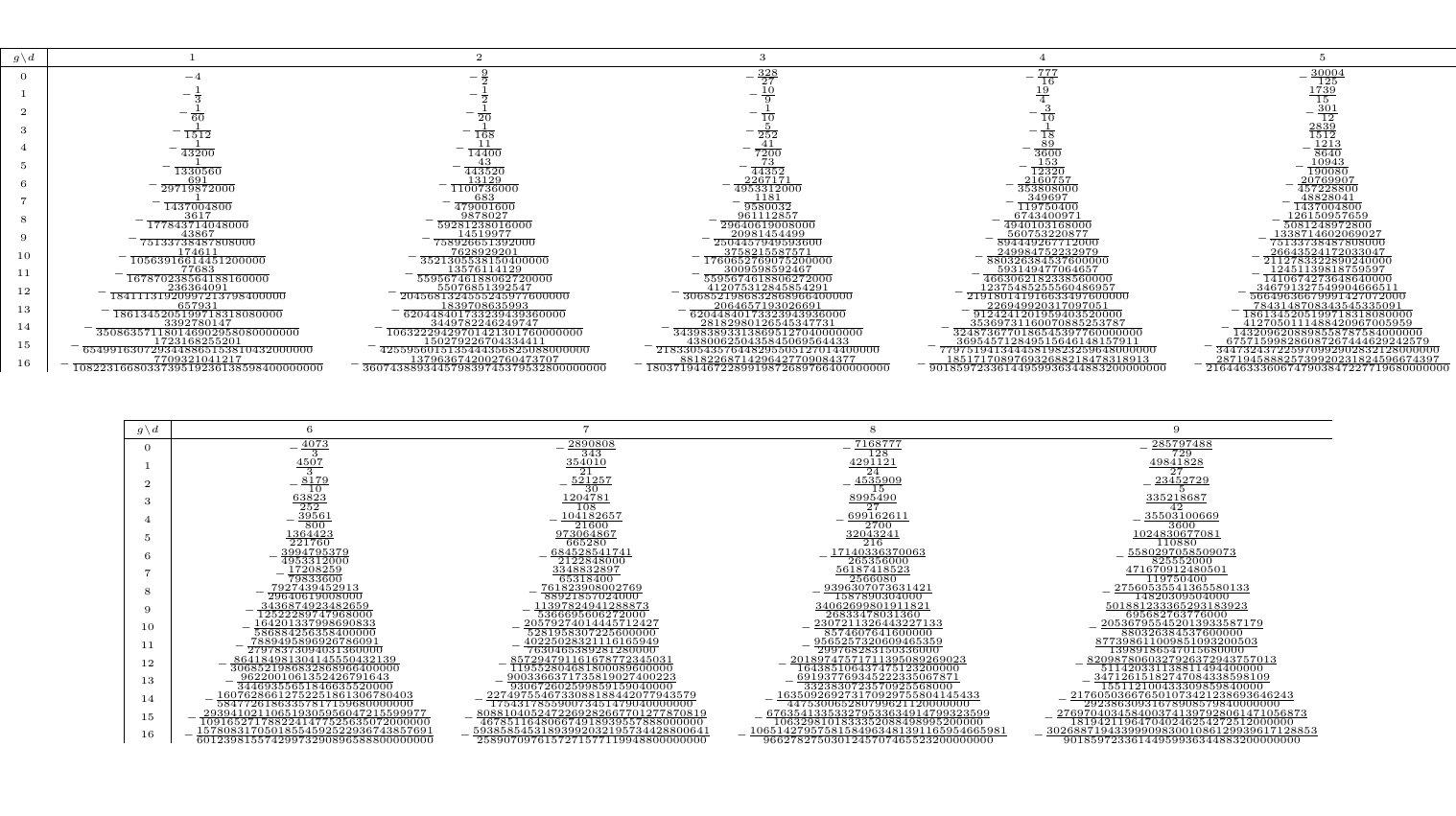}
\end{minipage}

%%%%%%%%%%%%%%%%%%%%%%%%%%%%%%%%%%%%%%%%%%%%%%%%%%%%%%%%%%%%%%%%%
\subsubsection*{$abc$-Coefficients}
%%%%%%%%%%%%%%%%%%%%%%%%%%%%%%%%%%%%%%%%%%%%%%%%%%%%%%%%%%%%%%%%%

\begin{minipage}{\textwidth}
\hspace{-2cm}
\begin{tabular}{cc}
\raisebox{2cm}{\includegraphics[]{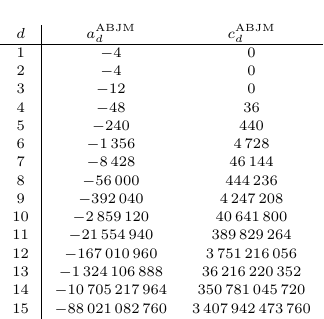}} \quad \quad
&
\includegraphics[]{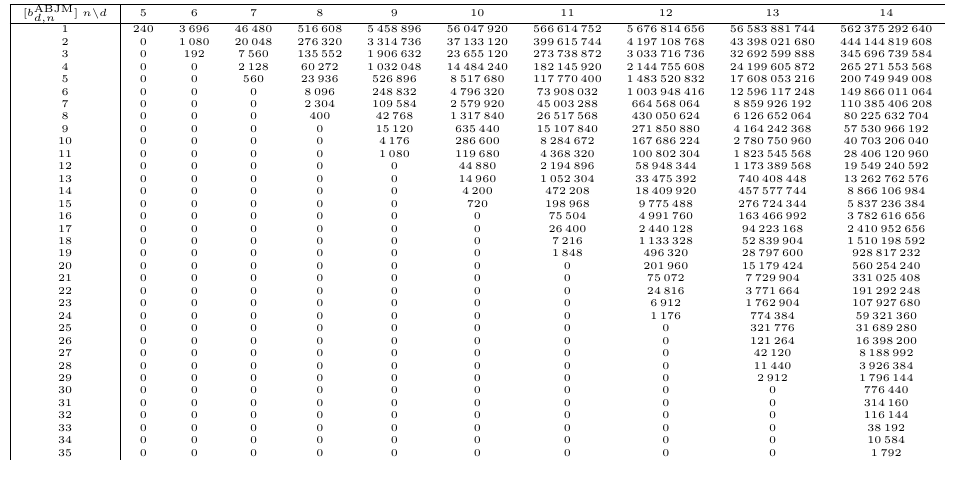}
\end{tabular}
\end{minipage}

%%%%%%%%%%%%%%%%%%%%%%%%%%%%%%%%%%%%%%%%%%%%%%%%%%%%%%%%%%%%%%%%%
\subsection{Local Curve $X_p$}\label{sec:appendix_local}
%%%%%%%%%%%%%%%%%%%%%%%%%%%%%%%%%%%%%%%%%%%%%%%%%%%%%%%%%%%%%%%%%

%%%%%%%%%%%%%%%%%%%%%%%%%%%%%%%%%%%%%%%%%%%%%%%%%%%%%%%%%%%%%%%%%
\subsubsection*{GW Invariants ($p=3$)}
%%%%%%%%%%%%%%%%%%%%%%%%%%%%%%%%%%%%%%%%%%%%%%%%%%%%%%%%%%%%%%%%%

\hspace{-2cm}
\begin{minipage}{\textwidth}
\centering
\includegraphics[]{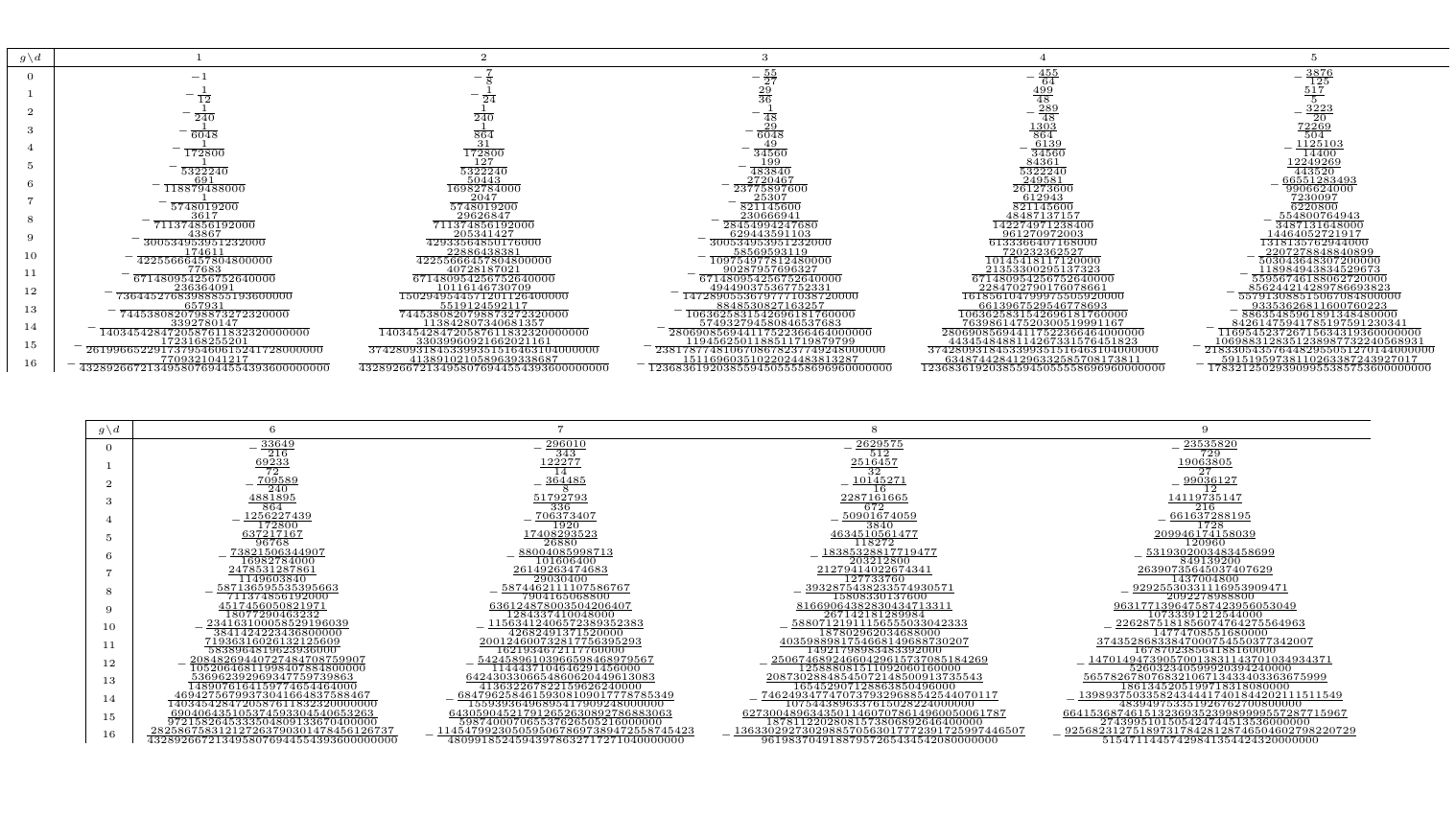}
\end{minipage}

%%%%%%%%%%%%%%%%%%%%%%%%%%%%%%%%%%%%%%%%%%%%%%%%%%%%%%%%%%%%%%%%%
\subsubsection*{GW Invariants ($p=4$)}    
%%%%%%%%%%%%%%%%%%%%%%%%%%%%%%%%%%%%%%%%%%%%%%%%%%%%%%%%%%%%%%%%%

\hspace{-2cm}
\begin{minipage}{\textwidth}
\centering
\includegraphics[]{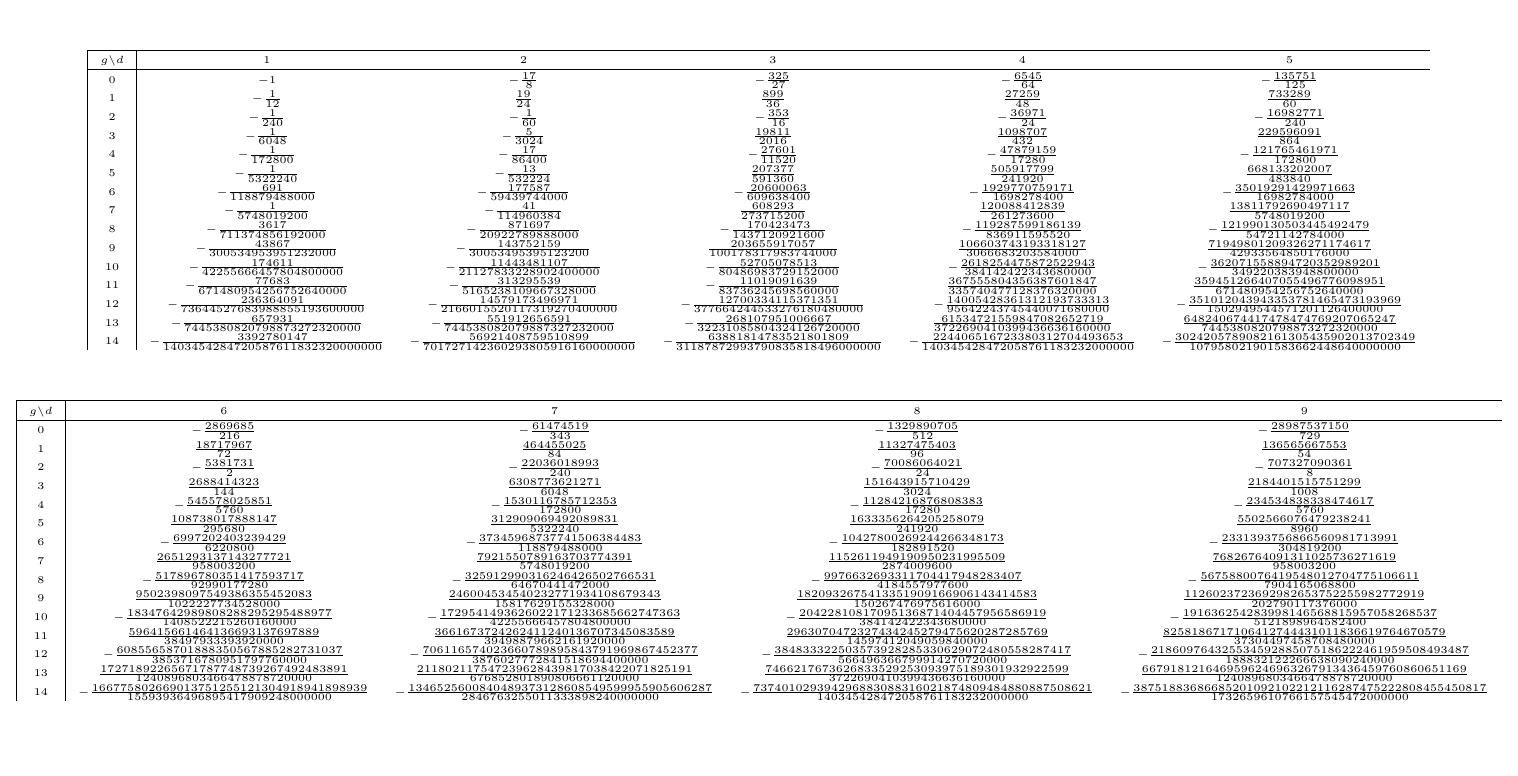}
\end{minipage}
    
%%%%%%%%%%%%%%%%%%%%%%%%%%%%%%%%%%%%%%%%%%%%%%%%%%%%%%%%%%%%%%%%%
\subsubsection*{GW Invariants ($p=5$)}
%%%%%%%%%%%%%%%%%%%%%%%%%%%%%%%%%%%%%%%%%%%%%%%%%%%%%%%%%%%%%%%%%

\hspace{-2.cm}
\begin{minipage}{\textwidth}
\centering
\includegraphics[]{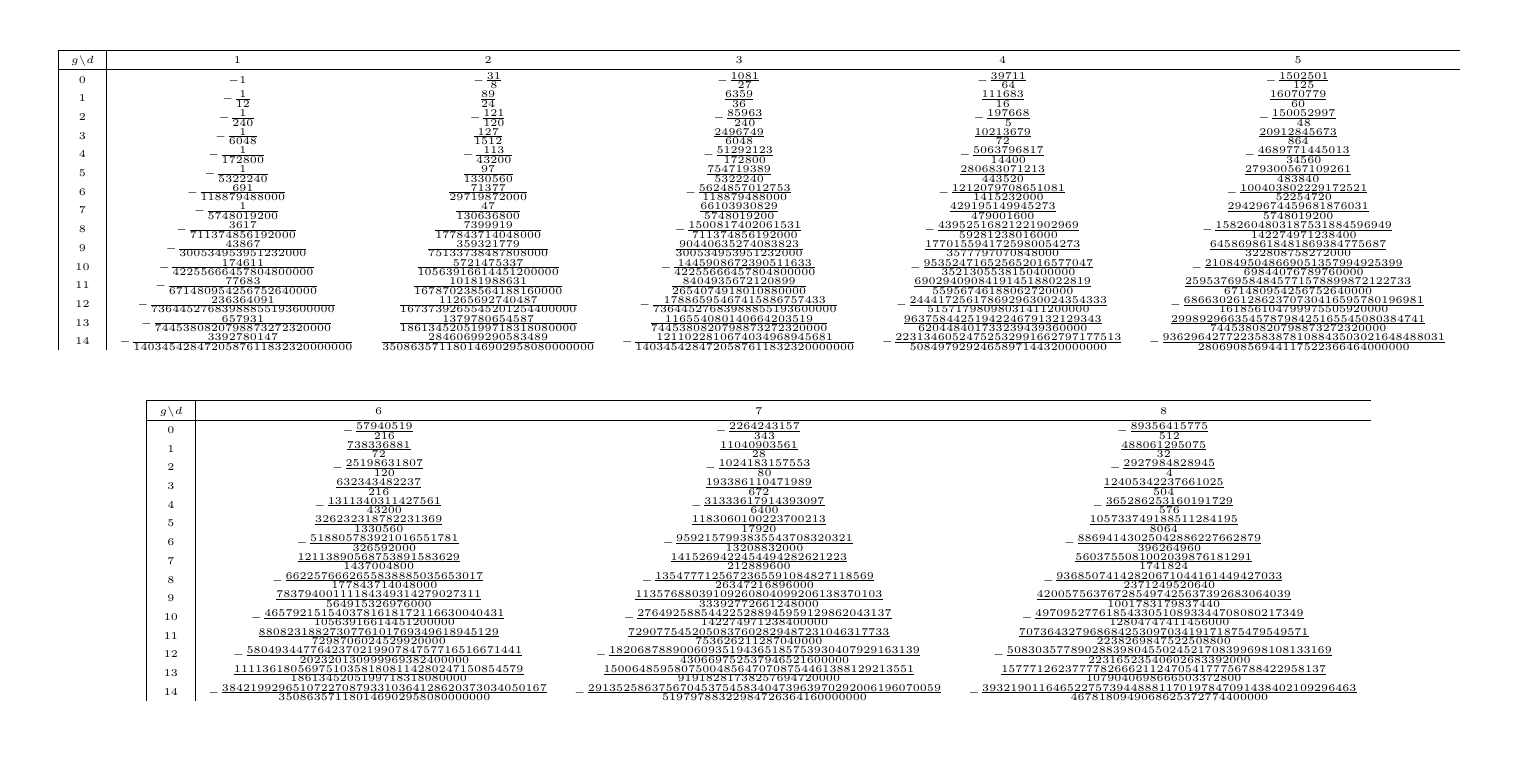}
\end{minipage}

%%%%%%%%%%%%%%%%%%%%%%%%%%%%%%%%%%%%%%%%%%%%%%%%%%%%%%%%%%%%%%%%%
\subsubsection*{$abc$-Coefficients ($p=3,4,5$)}
%%%%%%%%%%%%%%%%%%%%%%%%%%%%%%%%%%%%%%%%%%%%%%%%%%%%%%%%%%%%%%%%%

\begin{adjustwidth*}{-0cm}{-0.8cm}   
\begin{minipage}{\textwidth}
\hspace{-2cm}
\begin{tabular}{ccc}
\begin{tabular}{c}
\raisebox{0cm}{\includegraphics[]{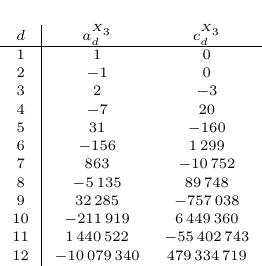}} \\
\raisebox{0cm}{\includegraphics[]{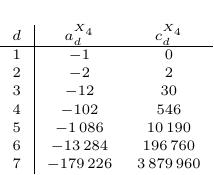}} \\   
\raisebox{12cm}{\includegraphics[]{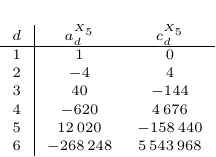}}
\end{tabular} 
&
\includegraphics[]{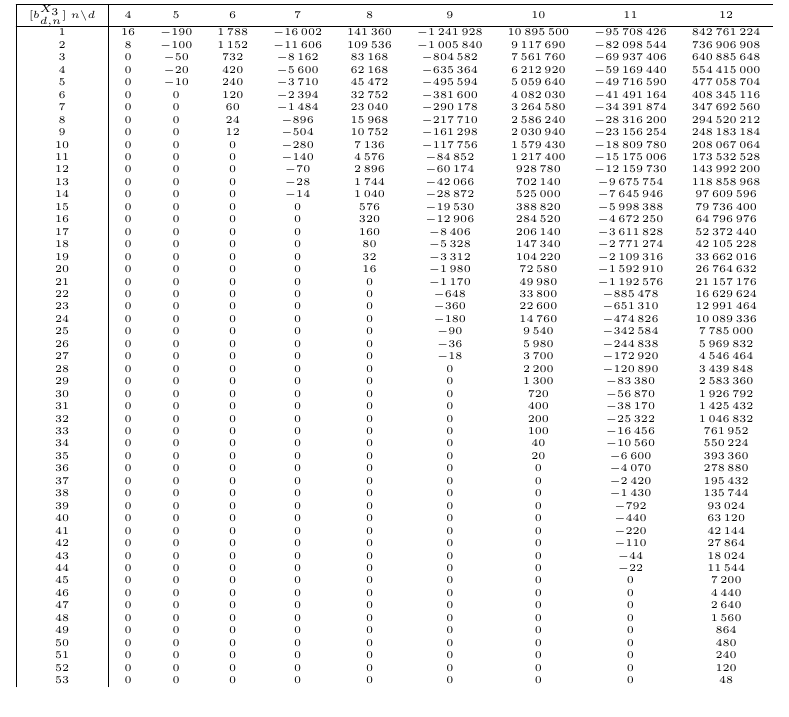}
\begin{tabular}{c}
\raisebox{0.2cm}{
\includegraphics[]{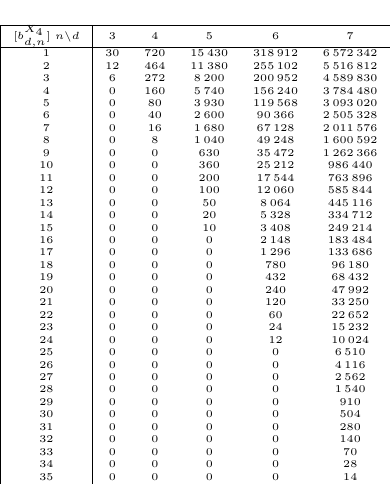}} \\
\raisebox{8cm}{\includegraphics[]{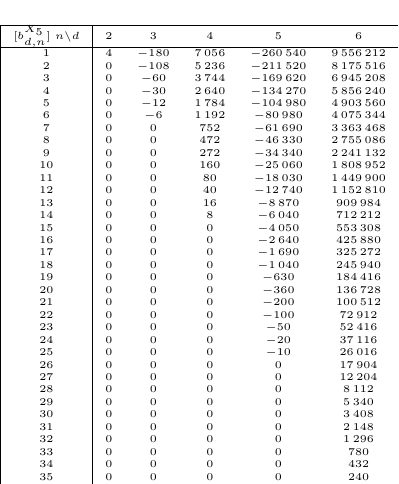}}
\end{tabular}
\end{tabular}
\end{minipage}
\end{adjustwidth*}

\end{landscape}

%%%%%%%%%%%%%%%%%%%%%%%%%%%%%%%%%%%%%%%%%%%%%%%%%%%%%%%%%%%%%%%%%
\subsubsection*{Large-Degree Expansion Coefficients $\widehat{c}_{j_0}^{(j)}$}
%%%%%%%%%%%%%%%%%%%%%%%%%%%%%%%%%%%%%%%%%%%%%%%%%%%%%%%%%%%%%%%%%

\begin{table}[ht]
\centering
\renewcommand{\arraystretch}{1.25}
\begin{tabular}{c|cccc}
\begin{picture}(20,20)(0,0)
\put(15,12){$j_0$}
\put(3.6,18){\line(1,-1){22.3}}
\put(5,0){$j$}
\end{picture}
& 1 & 2 & 3 & 4\\
\hline
1  & $\frac{\sqrt{2}}{3}(f-2)$ & & & \\
2  & $\frac{f^2 + 5f + 5}{18}$ & $(f-2)^2$ & & \\
3  & $-\frac{\sqrt{2} (f-2) \left(f^2+23 f-23\right)}{810}$ & $\frac{\sqrt{2} \left(f^2+5 f-5\right) (f-2)}{54}$  & $\frac{\sqrt{2} (f-2)^3}{81}$ & \\
4  & $\frac{-7 f^4+2 f^3-33 f^2+62 f-31}{1620}$ & $\frac{7 f^4-2 f^3+1113 f^2-2222 f+1111}{9720}$ & $\frac{(f-2)^2 \left(f^2+5 f-5\right)}{162}$ & $\frac{(f-2)^4}{486}$\\
\end{tabular}
%\caption{Coefficients $\widehat{c}_{j_0}^{(j)}$ of the large-degree expansion of GW invariants for the local curve.}
\label{tableXp_larged}
\end{table}

%%%%%%%%%%%%%%%%%%%%%%%%%%%%%%%%%%%%%%%%%%%%%%%%%%%%%%%%%%%%%%%%%
\subsubsection*{Free Energy Coefficients $a_{g,i}(p)$}
%%%%%%%%%%%%%%%%%%%%%%%%%%%%%%%%%%%%%%%%%%%%%%%%%%%%%%%%%%%%%%%%%

\begin{center}
\renewcommand{\arraystretch}{1.75}
\fontsize{13pt}{10pt}
\begin{tabular}{cc}
\begin{tabular}{c|c}
i & $a_{2,i}(p)$\\
\hline
1  & $-\frac{1}{240 f^5}$ \\
2  & $-\frac{-2 f^3+25 f^2+f+12}{2880 f^5}$ \\
3  & $\frac{-12 f^3+9 f^2-35 f+2}{2880 f^4}$ \\
4  & $-\frac{7 f+5}{2880 f^2}$ \\
5  & $\frac{f-1}{2880 f}$ 
\end{tabular} \,\,\,\,\,\,
&
\begin{tabular}{c|c}
i & $a_{3,i}(p)$\\
\hline
1  & $\frac{1}{6048 f^{10}}$ \\
2  & $\frac{3 f^4-70 f^3+497 f^2-630 f+280}{241920 f^{10}}$ \\
3  & $\frac{-137 f^5+1278 f^4-3045 f^3+3187 f^2-1883 f+360}{181440 f^{10}}$ \\
4  & $\frac{1741 f^6-8517 f^5+17136 f^4-21039 f^3+13013 f^2-4734 f+360}{362880 f^{10}}$ \\
5  & $-\frac{636 f^6-3031 f^5+7693 f^4-9638 f^3+7735 f^2-3031 f+636}{120960 f^9}$ \\ 
6  & $\frac{360 f^6-4734 f^5+13118 f^4-21039 f^3+17031 f^2-8517 f+1741}{362880 f^8}$ \\
7  & $\frac{360 f^5-1853 f^4+3187 f^3-3075 f^2+1278 f-137}{181440 f^7}$ \\
8  & $\frac{295 f^4-630 f^3+482 f^2-70 f+3}{241920 f^6}$ \\
9  & $\frac{f^2+12 f-1}{72576 f^3}$ \\
10  & $\frac{f^2-1}{725760 f^2}$
\end{tabular}
\end{tabular}
\end{center}
\noindent
If we define the following quantity (recall that $f \equiv \left( p-1 \right)^2$)
\be 
\bar{a}_{g,i} (f) = C_g\, f^{6(g-1)}\, a_{g,i}(f) + \binom{5(g-1)}{i} \left( f^i - f^{i+g-1} \right),
\ee
\noindent
we find that it has the following ``reflection'' property
\be 
\bar{a}_{g,i} (f) = f^{6(g-1)}\, \bar{a}_{g,5(g-1)-i} \left(\frac{1}{f}\right). \label{reflection}
\ee
\noindent
This would, in general, reduce the number of GW invariants needed to completely fix \eqref{FgBmodel}, from $5(g-1)$ down to $\lceil \frac{5(g-1)}{2} \rceil$. The available amount of data is unfortunately not enough to completely pinpoint a general expression for the constant $C_g$.

\begin{landscape}\fontsize{13pt}{10pt}
\renewcommand{\arraystretch}{1.75}

\begin{tabular}{c|c}
i & $a_{4,i}(p)$\\
\hline
1  & $-\frac{1}{172800 f^{15}}$ \\
2  & $\frac{2 f^5-75 f^4+994 f^3-5350 f^2+8461 f-2604}{14515200 f^{15}}$ \\
3  & $-\frac{2582 f^6-42087 f^5+243480 f^4-584534 f^3+616185 f^2-314250 f+45360}{43545600 f^{15}}$ \\
4  & $\frac{88290 f^7-910787 f^6+3434955 f^5-6337605 f^4+6666425 f^3-3968880 f^2+1197450 f-98280}{43545600 f^{15}}$ \\
5  & $-\frac{1403015 f^8-10868010 f^7+34735692 f^6-63129674 f^5+70900605 f^4-49330860 f^3+20133240 f^2-3982608 f+181440}{87091200 f^{15}}$ \\
6  & $\frac{3509560 f^9-25035695 f^8+83179010 f^7-163178700 f^6+202937816 f^5-163942655 f^4+83098100 f^3-24287100 f^2+3040776 f-60480}{87091200 f^{15}}$ \\
7  & $\frac{-1527084 f^9+13390590 f^8-52041408 f^7+115753075 f^6-164355872 f^5+152803338 f^4-92842411 f^3+34605080 f^2-6972948 f+437688}{43545600 f^{14}}$ \\
8  & $\frac{218844 f^9-3486474 f^8+17302540 f^7-46417988 f^6+76401669 f^5-82177936 f^4+57873320 f^3-26020704 f^2+6695295 f-763542}{21772800 f^{13}}$ \\
9  & $\frac{-60480 f^9+3040776 f^8-24287100 f^7+83108110 f^6-163942655 f^5+202937816 f^4-163188710 f^3+83179010 f^2-25035695 f+3509560}{87091200 f^{12}}$ \\
10  & $\frac{-181440 f^{15}+3982608 f^{14}-20127234 f^{13}+49330860 f^{12}-70900605 f^{11}+63123668 f^{10}-34735692 f^9+10868010 f^8-1403015 f^7}{87091200 f^{18}}$ \\
11  & $-\frac{98280 f^{15}-1198815 f^{14}+3968880 f^{13}-6666425 f^{12}+6338970 f^{11}-3434955 f^{10}+910787 f^9-88290 f^8}{43545600 f^{18}}$ \\
12  & $-\frac{44905 f^{15}-314250 f^{14}+616185 f^{13}-584079 f^{12}+243480 f^{11}-42087 f^{10}+2582 f^9}{43545600 f^{18}}$ \\
13  & $\frac{35 f^{16}-2604 f^{15}+8461 f^{14}-5385 f^{13}+994 f^{12}-75 f^{11}+2 f^{10}}{14515200 f^{18}}$ \\
14 & $\frac{5 f^{17}-84 f^{15}-5 f^{14}}{14515200 f^{18}}$ \\
15  & $\frac{f^{18}-f^{15}}{43545600 f^{18}}$
\end{tabular} 

%%%%%%%%%%%%%%%%%%%%%%%%%%%%%%%%%%%%%%%%%%%%%%%%%%%%%%%%%%%%%%%%%
\subsection{Hurwitz Theory} \label{sec:appendix_Hur_d}
%%%%%%%%%%%%%%%%%%%%%%%%%%%%%%%%%%%%%%%%%%%%%%%%%%%%%%%%%%%%%%%%%

%%%%%%%%%%%%%%%%%%%%%%%%%%%%%%%%%%%%%%%%%%%%%%%%%%%%%%%%%%%%%%%%%
\subsubsection*{GW Invariants}
%%%%%%%%%%%%%%%%%%%%%%%%%%%%%%%%%%%%%%%%%%%%%%%%%%%%%%%%%%%%%%%%%

\hspace{-2cm}
\begin{minipage}{\textwidth}
\centering
\includegraphics[]{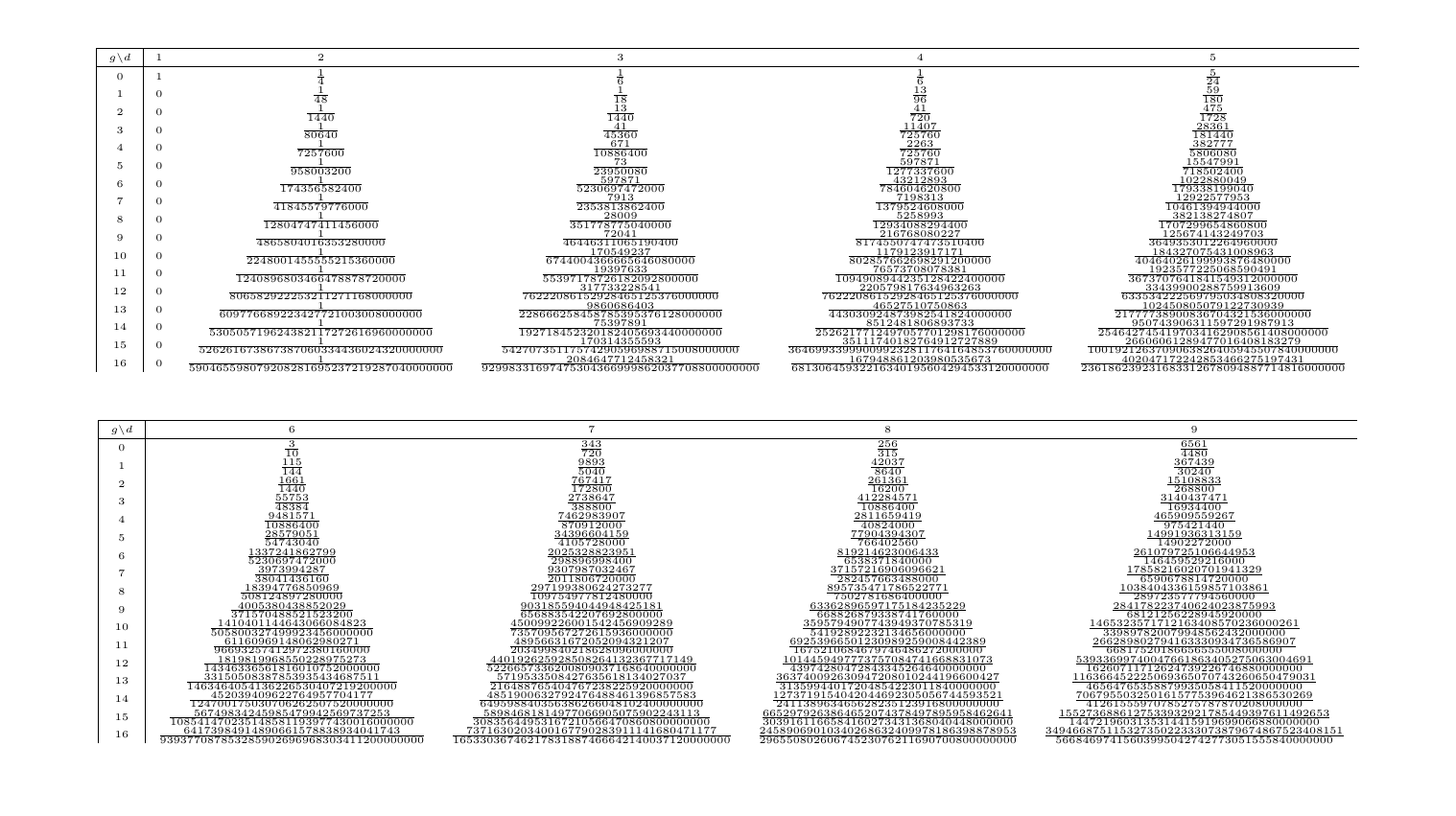}
\end{minipage}

%%%%%%%%%%%%%%%%%%%%%%%%%%%%%%%%%%%%%%%%%%%%%%%%%%%%%%%%%%%%%%%%%
\subsubsection*{Large-Degree Expansion Coefficients $\widehat{c}_{j_0}^{\tH,(j)}$}
%%%%%%%%%%%%%%%%%%%%%%%%%%%%%%%%%%%%%%%%%%%%%%%%%%%%%%%%%%%%%%%%%

\begin{table}[ht]
\centering
\renewcommand{\arraystretch}{1.25}
\begin{tabular}{c|ccccccc}
\begin{picture}(20,20)(0,0)
\put(15,12){$j_0$}
\put(3.6,18){\line(1,-1){22.3}}
\put(5,0){$j$}
\end{picture}
& 1 & 2 & 3 & 4 & 5 & 6 & 7\\
\hline
1  & $\frac{\sqrt{2}}{3}$ & & & & & & \\
2  & $\frac{1}{18}$ & $\frac{1}{9}$ & & & & & \\
3  & $-\frac{1}{405 \sqrt{2}}$ & $\frac{1}{27 \sqrt{2}}$ & $\frac{\sqrt{2}}{81}$ & & & & \\
4  & $-\frac{7}{1620}$ & $-\frac{7}{1620}$ & $\frac{1}{162}$ & $\frac{1}{486}$ & & & \\
5  & $-\frac{5}{2268 \sqrt{2}}$ & $-\frac{11}{3645 \sqrt{2}}$ & $\frac{11}{14580 \sqrt{2}}$ & $\frac{1}{729 \sqrt{2}}$ & $\frac{1}{3645 \sqrt{2}}$ & & \\
6  & $-\frac{88}{382725}$ & $-\frac{8941}{9185400}$ & $-\frac{29}{58320}$ & $\frac{37}{262440}$ & $\frac{1}{8748}$ & $\frac{1}{65610}$ & \\
7  & $\frac{101}{1020600\sqrt{2}}$ & $-\frac{487}{1837080\sqrt{2}}$ & $-\frac{11237}{27556200\sqrt{2}}$ & $-\frac{1}{9720\sqrt{2}}$ & $\frac{13}{393660\sqrt{2}}$ & $\frac{1}{65610\sqrt{2}}$ & $\frac{1}{688905\sqrt{2}}$\\
\end{tabular}
%\caption{Coefficients $\widehat{c}_{j_0}^{\tH,(j)}$ of the large-degree expansion of GW invariants in Hurwitz theory.}
\label{tableHur_larged}
\end{table}

%%%%%%%%%%%%%%%%%%%%%%%%%%%%%%%%%%%%%%%%%%%%%%%%%%%%%%%%%%%%%%%%%
\subsection{Quintic}\label{sec:appendix:quintic}
%%%%%%%%%%%%%%%%%%%%%%%%%%%%%%%%%%%%%%%%%%%%%%%%%%%%%%%%%%%%%%%%%

%%%%%%%%%%%%%%%%%%%%%%%%%%%%%%%%%%%%%%%%%%%%%%%%%%%%%%%%%%%%%%%%%
\subsubsection*{GW Invariants}
%%%%%%%%%%%%%%%%%%%%%%%%%%%%%%%%%%%%%%%%%%%%%%%%%%%%%%%%%%%%%%%%%

\hspace{-2cm}
    \begin{minipage}{\textwidth}
       \centering
      \includegraphics[]{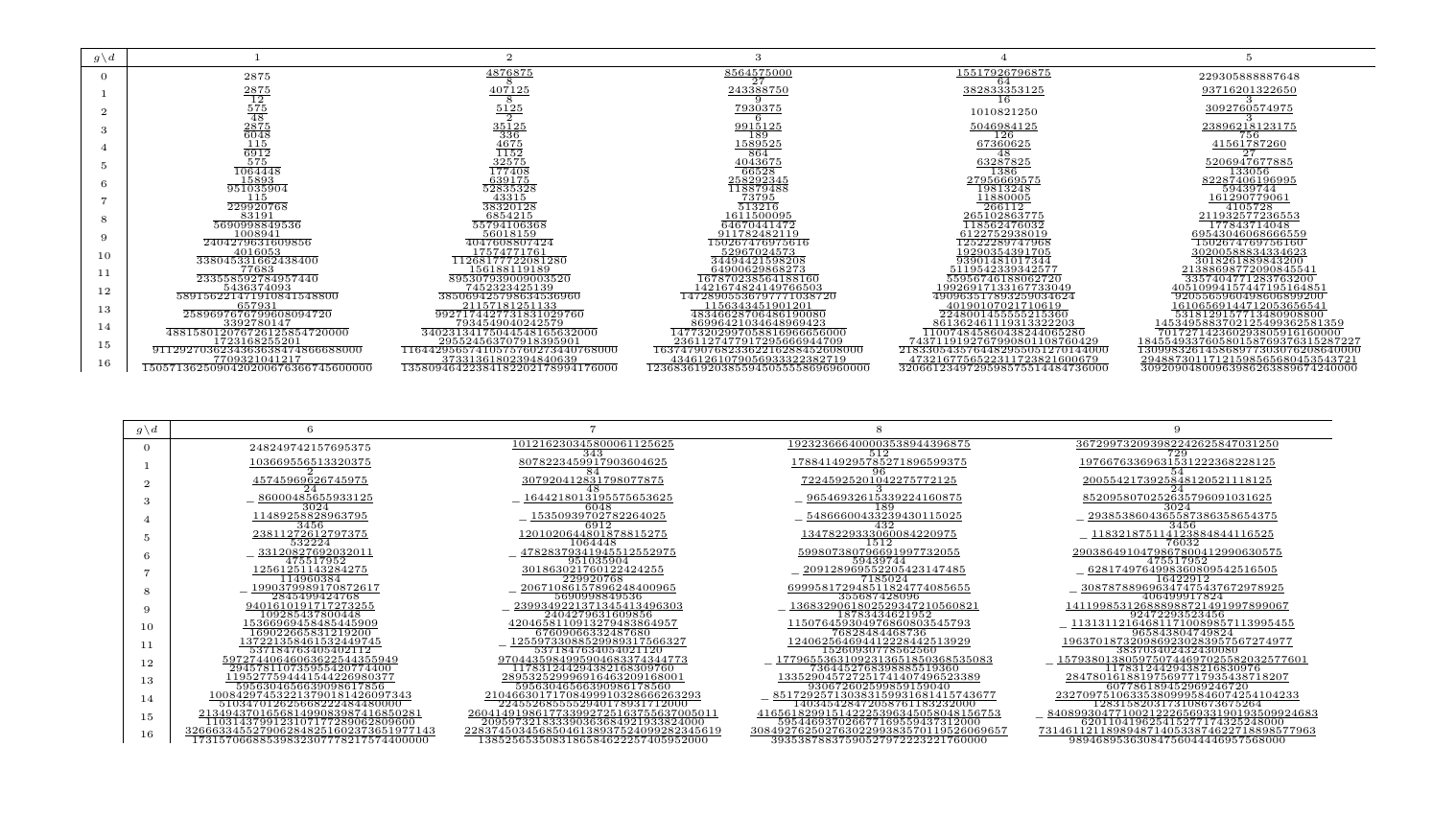}
    \end{minipage}

%%%%%%%%%%%%%%%%%%%%%%%%%%%%%%%%%%%%%%%%%%%%%%%%%%%%%%%%%%%%%%%%%
\subsubsection*{$abc$-Coefficients}
%%%%%%%%%%%%%%%%%%%%%%%%%%%%%%%%%%%%%%%%%%%%%%%%%%%%%%%%%%%%%%%%%

    \begin{minipage}{\textwidth}
    \hspace{-2cm}
      \begin{center}
      \raisebox{2cm}{\includegraphics[]{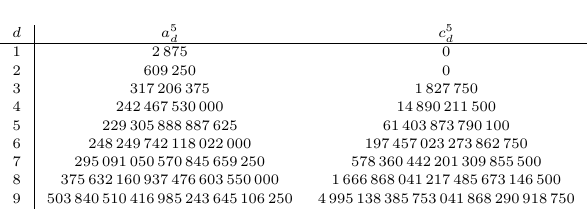}} \quad \quad
      \\
      \includegraphics[]{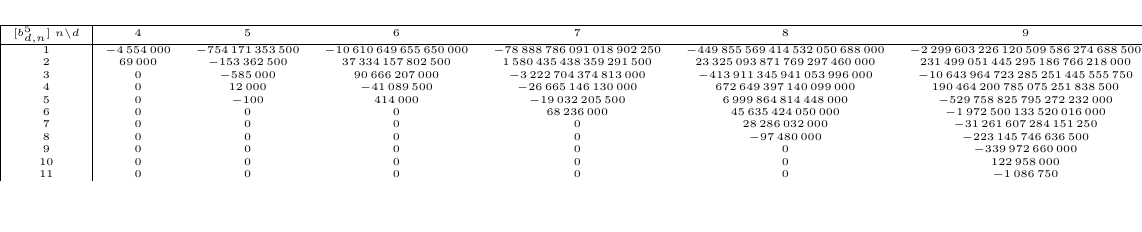}
      \end{center}
    \end{minipage}
    
\end{landscape}

%%%%%%%%%%%%%%%%%%%%%%%%%%%%%%%%%%%%%%%%%%%%%%%%%%%%%%%%%%%%%%%%%
%%%%%%%%%%%%%%%%%%%%%%%%%%%%%%%%%%%%%%%%%%%%%%%%%%%%%%%%%%%%%%%%%
\newpage
%%%%%%%%%%%%%%%%%%%%%%%%%%%%%%%%%%%%%%%%%%%%%%%%%%%%%%%%%%%%%%%%%
%%%%%%%%%%%%%%%%%%%%%%%%%%%%%%%%%%%%%%%%%%%%%%%%%%%%%%%%%%%%%%%%%

\bibliographystyle{plain}
%\bibliography{papers}

\end{document}